\documentclass[12pt]{amsart}
\usepackage{amsmath, amsthm, amssymb,cite,enumitem}
\usepackage{fullpage}
\usepackage{color}
\usepackage{hyperref}
\usepackage[makeroom]{cancel}
\usepackage[english]{babel}
\usepackage{framed}
\usepackage[utf8]{inputenc}
\usepackage{amsmath}
\usepackage{amsfonts}
\usepackage{amssymb}
\usepackage{calligra}
\usepackage{calrsfs}
\usepackage{yfonts}
\usepackage[mathscr]{euscript}

\usepackage{tikz}

\newtheorem{theorem}{Theorem}
\newtheorem{lemma}{Lemma}
\newtheorem{proposition}[lemma]{Proposition}
\newtheorem{corollary}[lemma]{Corollary}

\newtheorem{remark}[lemma]{Remark}

\numberwithin{lemma}{section}

\renewcommand{\H}{\mathcal{H}}

\numberwithin{equation}{section}

\newcommand{\R}{{\mathbb R}}

\renewcommand{\H}{{\mathcal H }}

\newcommand\Tau{\mathcal{T}}
\newcommand{\ZZ}{{\mathcal Z}}
\newcommand{\hh}{h}

\usepackage[makeroom]{cancel}

\newcommand{\w}{{\ w}}

\renewcommand{\w}{{\ w}}

\newcommand{\ww}{{\omega}}

\newcommand{\tu}{\tilde{u}}
\newcommand{\tv}{\tilde{v}}

\begin{document}

\title{Almost global well-posedness for quasilinear  strongly coupled wave-Klein-Gordon systems in two space dimensions}
\author{Mihaela Ifrim}
\address{Department of Mathematics, University of Wisconsin, Madison}
\email{ifrim@math.wisc.edu}

\author{Annalaura Stingo}
\address{Department of Mathematics, University of California, Davis}
\email{astingo@ucdavis.edu}

\begin{abstract}
We prove almost global well-posedness for quasilinear strongly coupled wave-Klein-Gordon systems with small and localized data in two space dimensions.  We assume only mild decay on the data at infinity as well as minimal regularity. We  systematically investigate  all the possible quadratic null form type quasilinear strong coupling nonlinearities, and provide a new, robust approach for the proof. In a  second paper we will complete the present results to full global well-posedness. 
\end{abstract}

\date{10.28.2019}
\maketitle

\section{Introduction}

The problem we will address here is the  Cauchy problem  for the following quasilinear strongly coupled wave-Klein-Gordon system 
\begin{equation} \label{WKG}
\left\{
\begin{aligned}
&(\partial^2_t  - \Delta_x) u(t,x) = \mathbf{N_1}(v, \partial v) + \mathbf{N_2}(u, \partial v) \,, \\
&(\partial^2_t  - \Delta_x  + 1)v(t,x) = \mathbf{N_1}(v, \partial u) + \mathbf{N_2}(u, \partial u)\,,
\end{aligned} 
\right. \quad (t,x)\in \left[0,+\infty \right) \times \mathbb{R}^2,
\end{equation}
with initial conditions
\begin{equation}\label{data}
\left\{
\begin{aligned}
& (u,v)(0,x) = (u_0(x),v_0(x))\,,\\
& (\partial_t u,\partial_t v)(0,x) =  (u_1(x), v_1(x)).
\end{aligned}
\right.
\end{equation}

The nonlinearities $\mathbf{N_1} (\cdot, \cdot)$ and $\mathbf{N_2} (\cdot, \cdot)$ represent the wave-Klein-Gordon coupling via classical quadratic null
structures. Precisely, $\mathbf{N_1}(\cdot, \cdot)$  and $\mathbf{N_2}(\cdot, \cdot)$ will be  linear
combinations of the classical  quadratic null forms
\begin{equation}
\label{null forms}
\left\{
\begin{aligned}
&Q_{ij}(\phi, \psi) =\partial_i  \phi \partial_j \psi -\partial_i\psi  \partial_j \phi, \\
&Q_{0i}(\phi, \psi) =\partial_t  \phi \partial_i \psi -\partial_t\psi  \partial_i \phi, \\
&Q_{0}(\phi, \psi) =\partial_t \phi   \partial_t \psi -\nabla_x \psi \cdot \nabla_x   \phi.
\end{aligned}
\right.
\end{equation}

The main result we present in this paper asserts the almost global
existence of solutions to the above system, when initial data are
assumed to be small and localized. This is the first of a two paper
sequence, where the aim of the second paper is to improve the
almost-global well-posedness result to a global well-posedness
result. The reason we structure this work in two papers is that
they address very different aspects of the problem using essentially
disjoint ideas and methods.

\medskip

Compared with prior related  works, our novel contributions here include the following
\begin{itemize}
\item Our quasilinear structure provides a strong coupling between the
  wave and the Klein-Gordon equation, unlike any other prior works in
  two space dimensions (except for the second author's work
  \cite{stingo2}, that only applies to the $Q_{0}$ type
  nonlinearities).
\item We make no assumptions on the support of the initial
  data. Furthermore, we make very mild decay assumptions on the
  initial data at infinity. In particular, we use only two Klainerman
  vector fields in the analysis, which is optimal and below anything
  that has been done before.
\item Rather than using arbitrarily high regularity, here we work with
  very limited regularity initial data, e.g. our two vector fields bound
is simply in the energy space.
\item In terms of methods, our work is based on a combination
  of energy estimates localized to dyadic space-time regions, and
  pointwise interpolation type estimates within the same regions. This is  akin
  to ideas previously used by Metcalfe-Tataru-Tohaneanu \cite{mtt} in
  a linear setting, and is also related to Alinhac's ghost weight method
  \cite{alih}.
\end{itemize}

We remark that our methods allow for a larger array of weak quasilinear null form interactions
in the equations, as well as non-null $v$-$v$ interactions.  We focus our 
exposition to the case of strong interactions above simply because this case  is more difficult
 and has not been considered before except for \cite{stingo2}.
 \subsection{Motivation and a brief history} 

 The model we study here is physically motivated by problems arising
 in general relativity, where many similarly structured problems
 arise.  Most of the results known so far concern the
 wave-Klein-Gordon systems in the $3$ dimensional setting, but there
 is also a fair amount of work done in the $2$ dimensional case, where systems akin to  \eqref{WKG}, but with different types of nonlinearities, have been
 considered.  Since our result is set in the $2$-dimensional setting,
 we will focus mostly in explaining what has been done in this
 direction and how it relates to our result.

Regardless of the spatial dimension considered, one
 always has to understand and deal with resonant interactions.
In the wave-wave to wave bilinear interactions,  resonance 
occurs for parallel waves. This is where the null condition plays a major role,
as it cancels these interactions. In all other wave-Klein Gordon bilinear interactions
there is no true resonance; however, there is a near resonance for almost parallel waves
in the high frequency limit, which becomes stronger in a quasilinear setting. 
For this reason, the null condition is still important in the wave-Klein Gordon 
quasilinear interactions, perhaps less so in the semilinear ones.

The main difference between the two dimensional setting and higher dimensions
is due to the weaker dispersive decay in low dimension.
 In particular, this is the reason why our analysis and methods are much more 
involved than in the work done in the three dimensional case, and also why more is required 
in terms of the structure of the nonlinearity in two dimensions.

\bigskip

In what follows we review some of the work which is relevant to our
result, and which has been done for the wave-Klein-Gordon system. Some
relatively recent work in the three dimensional setting that relates
with this model started with the work of Georgiev \cite{georgiev}, and
Katayama \cite{katayama}, who proved the global existence of small
amplitude solutions to coupled systems of wave and Klein-Gordon
equations under certain suitable conditions on the
nonlinearity. These include the \textit{null condition} of Klainerman
\cite{sergiu} on self-interactions between wave components. Katayama's conditions
 imposed on the nonlinearities are weaker than the
\textit{strong null condition} used by Georgiev. Relevant to our work is also Delort's work on Klein-Gordon systems \cite{delort, delort1, delort2, delort3, delort4}. More recently, a related
problem was also studied by LeFloch, Ma \cite{leflochma2} and Wang \cite{wang} as a model for the full
Einstein-Klein-Gordon system.
There the authors prove global existence of solutions to
wave-Klein-Gordon systems with quasilinear quadratic nonlinearities
satisfying suitable conditions, when initial data are small, smooth
and compactly supported.  An idea used there is that of employing hyperbolic coordinates
in the forward light cone; this was first introduced in the wave context in the work of Tataru~\cite{tataru-hyp},
and later reintroduced by LeFloch-Ma in \cite{leflochma2} under the name \textit{hyperboloidal
  foliation method}.  In \cite{ionescupausader} Ionescu and Pausader also prove
global regularity and modified scattering in the case of small smooth
initial data that decay at suitable rates at infinity, but not
necessarily compactly supported.  We also cite a work by Dong-Wyatt \cite{DW20} in which global well-posedness is proved for a quadratic semilinear wave-Klein-Gordon interaction in which there are no derivatives on the wave component of the solution.
Global stability for the full Einstein-Klein-Gordon
system has been then proved by LeFloch-Ma \cite{leflochma1} in the case of small smooth perturbations that agree
with a Schwarzschild solution outside a compact set (see also Wang
\cite{Q.Wang}), and by Ionescu-Pausader \cite{IP2} in the case of unrestricted data.

Most of the results we know concerning global existence of small amplitude
solutions in lower space dimension are due to Ma. His results apply
to compactly supported Cauchy data (a restriction that our current
result avoids) so that the hyperboloidal foliation
method can be used, see \cite{ma:2D_tools}. In particular, in
\cite{ma:2D_quasilinear} Ma combines this method with a normal form
argument to treat some quasilinear quadratic nonlinearities, while
in \cite{ma} he treats wave-Klein-Gordon coupled system with more general quasilinear
terms with null structure, but only in the case of a weakly coupled system. We also cite \cite{ma:2D_semilinear, Ma2020, DuanMa20}, in which Ma studies the case of some semilinear
quadratic interactions. In \cite{ma:1D_semilinear} the restriction on
the support of initial data is bypassed for the one-dimensional
problem, but there only a semilinear cubic model
wave-Klein-Gordon system is discussed. An example of quadratic semilinear wave-Klein-Gordon system is also studied by Dong-Wyatt in \cite{DW20_2d}.
The only global well-posedness result known at present for strongly-coupled quadratic and quasilinear wave-Klein-Gordon systems is an example studied by the second author in \cite{stingo2}, where a $Q_0$-type interaction is considered.

\subsection{The linear system and energy functionals}
The system \eqref{WKG} is a nonlinear version of the linear diagonal system
\begin{equation} \label{WKG-lin}
\left\{
\begin{aligned}
&(\partial^2_t  - \Delta_x) u(t,x) = 0 \,,  \\
&(\partial^2_t  - \Delta_x  + 1)v(t,x) = 0\,,
\end{aligned} 
\right. \quad (t,x)\in \left(0,+\infty \right) \times \mathbb{R}^2.
\end{equation}
The linear system  \eqref{WKG-lin} has an associated conserved energy
 given by
\begin{equation}
\label{energy}
E(t; u,v)=\int _{\mathbb{R}^2}  u_t ^2 +u_x^2 +  v_t ^2 +v_x^2 + v^2\, dx.
\end{equation}
This is no longer a conserved quantity for the nonlinear system
\eqref{WKG}, but we will still use it to define the associated energy
space and also all our main function spaces. The system
\eqref{WKG-lin} is a well-posed linear evolution in the space $\mathcal{H}^0$
with norm
\[
\Vert (u[t],v[t]) \Vert^2_{\mathcal{H}^0}:=  
\Vert u\Vert^2_{\dot{H}^1}+ \Vert u_t\Vert^2_{L^2} + \Vert v \Vert^2_{{H}^1}+ \Vert v_t\Vert^2_{L^2},
\]
where we use the following notation for the
Cauchy data in \eqref{WKG} at time $t$:
\[
\left(u[t], v[t]\right):=(u(t), u_t(t), v(t), v_t(t)).
\]

The higher order energy  spaces 
for the system \eqref{WKG-lin} are the spaces $\mathcal{H}^n$ endowed with the norm
\[
\Vert (u_0, u_1, v_0, v_1)\Vert^2_{\mathcal{H}^n}:= 
\sum_{\vert \alpha \vert \leq n} \Vert \partial_x^{\alpha}(u_0, u_1, v_0, v_1)\Vert^2_{\mathcal{H}^0},
\]
where $n\geq 1$. We will also use the  energy spaces for the nonlinear system \eqref{WKG}.

Above and in the sequel we use notation conventions as follows:
$\partial$ denotes time and spatial derivatives, $\partial_x $
denotes only the spatial derivatives, $\nabla$ represents the
space-time gradient, and $\nabla_x$ represents the spatial
gradient only. Also, LHS (resp. RHS) will be an abbreviation for ‘‘\emph{left hand side}'' (resp. ‘‘\emph{right hand side}'').

\subsection{Scaling, criticality and local well-posedness}
One important notion that will guide our efforts in proving optimal
results in terms of regularity is given by the scaling of the
problem. The nonlinear terms play a crucial role in the long time
dynamics of the solution and also influence the critical regularity
close to which we seek to prove our local and then global existence
results. To properly explain the relation between the nonlinearity and
the critical homogeneous Sobolev space, we connect the higher
dimensions with the notion of criticality by means of the scaling
symmetry which our system \eqref{WKG} possesses in the high frequency limit
\[
\left\{
\begin{aligned}
&u(t,x)\rightarrow \lambda^{-1} u(\lambda t, \lambda x )\\
&v(t,x)\rightarrow \lambda^{-1} v(\lambda t, \lambda x ).\\
\end{aligned}
\right.
\]
This, in particular, leads to the critical Sobolev space ${\mathcal
  H}^{s_c}$ with $s_c=d/2+1$. For our problem the critical Sobolev
exponent is $s_c = 2$. In particular, it is not too difficult to show
that in two dimensions \eqref{WKG} is locally well-posed in
$\mathcal{H}^n$ for $n \geq 4$ (or $\mathcal H^{3+\epsilon}$ if we do not restrict ourselves
to integers).

To describe the lifespan of the solutions we define the time dependent control norms
\begin{equation}
\label{norm A}
A:= \sum_{\vert \alpha \vert =1}\Vert \partial^{\alpha}u\Vert_{L^{\infty}} +\sum_{\vert \alpha \vert =1}\Vert \partial^{\alpha}v\Vert_{L^{\infty}},
\end{equation}
respectively 
\begin{equation}
\label{norm B}
B:= \sum_{\vert \alpha \vert = 2}\Vert \partial^{\alpha}u\Vert_{L^{\infty}} +\sum_{\vert \alpha \vert  =  2}\Vert \partial^{\alpha}v\Vert_{L^{\infty}}.
\end{equation}
Here $A$ is a scale invariant quantity which will be required to
remain small throughout in order to preserve the hyperbolicity of the
problem. Then we have the following local result:

\begin{theorem}\label{t:local}
a) The problem \eqref{WKG} is locally well-posed for initial data in  $\H^n$, $n \geq 4$, with the additional property 
that $A$ is small.

b) Uniform finite speed of propagation holds for as long as $A$ remains small.

c) The solutions can be continued for as long as $ \int B \, dt$ remains finite, and for each $k \geq 0$ we  have the following energy estimate:
\begin{equation}
\| (u,v)(t)\|_{\H^k} \lesssim e^{c \int_0^t B(s)\, ds}  \| (u,v)(0)\|_{\H^k} .
\end{equation}
\end{theorem}
We remark that this  also shows  the continuation of higher regularity  of the solution for as long as $\int B \,dt$ remains finite.

\subsection{The main result}
To  study the small data long time well-posedness problem 
for the nonlinear evolution \eqref{WKG} one needs to add
some decay assumptions for the initial data to the mix. 
At this point we can already state a preliminary version of 
our main theorem, which clarifies the type of initial data we 
are considering.

\begin{theorem}\label{t:main-easy}
Let $\hh \geq 8$. Assume that the initial data $(u[0],v[0])$ for \eqref{WKG}
satisfies
\begin{equation}\label{data-main}
\| (u[0],v[0])\|_{\mathcal{H}^{2\hh}} + \| x \partial_x (u[0],v[0])\|_{ \mathcal{H}^{\hh}} + \| x^2 \partial_x^2 (u[0],v[0])\|_{\mathcal{H}^0} \leq \epsilon \ll 1.
\end{equation}
Then the equation \eqref{WKG} is almost globally well-posed in the same space i.e., the solution exists up to time $T_{\epsilon}=e^{\frac{c}{\epsilon}}$, where $c$ is a small positive universal constant.

\end{theorem}
Here  we made an effort to limit the decay assumptions, i.e. use only $x^2$ type decay, but we did not attempt  to fully optimize the choice of $\hh$.

\subsection{Vector fields and the main result revisited}

To provide a better form of the above theorem, one should also describe the
global bounds and decay properties of the solutions.  This analysis is closely related to the family of Killing vector fields associated to
our problem, i.e. of vector fields that commute with the linear evolution
\eqref{WKG-lin}.  We will also add to the list below the scaling
vector field $S$, which is not Killing but plays an important role in the proof of our
main result in Theorem~\ref{t:main} below. The commuting vector fields together
with the scaling vector field are as follows:
\begin{align}
 \label{deriv}  \partial_t, \ \partial_1,&\ \partial_2,\\
 \label{rot} \Omega_{ij}=x_j\partial_i&-x_i\partial_j,\\
 \label{Lor}\Omega_{0i}=t\partial_i&+x_i\partial_t\\
 \label{scal} \mathscr{S}=t\partial_t&+r\partial_r,
\end{align}
where $1\leq i \neq j \leq 2$, $r=\vert x\vert$ and
$\partial_r=\frac{x}{r} \cdot \nabla_x$. The expressions in
\eqref{deriv} correspond to translations in the coordinate directions;
\eqref{rot} correspond to rotations in the space variable $x$;
\eqref{rot} and \eqref{Lor} correspond to the Lorentz transformations;
finally, \eqref{scal} corresponds to dilations.  To obtain symmetrical
notations we will sometimes write $t=x_0$ and
$\partial_t=\partial_0$. Note that in \eqref{rot} we can restrict to
$1\leq i <j\leq 2$ by skew-symmetry. Thus we have a total of $8$
different vector fields.

We refer to all the vector fields \eqref{rot} and \eqref{Lor} as the
\emph{ Klainerman vector fields} and we will denote all of them by $Z$
\begin{equation}
\label{Z}
Z:= \left\{ \Omega_{ij}, \Omega_{0i} \right\}.
\end{equation}
We denote the full set of vector fields associated to the symmetries of the linear problem
as
\begin{equation}
\label{ZZ}
\ZZ: =\left\{\partial_0,\partial_1,\partial_2, \Omega_{ij}, \Omega_{0i} \right\}.
\end{equation}
For a multiindex $\gamma= (\alpha, \beta)$  we denote
\[
\ZZ^\gamma = \partial^\alpha Z^\beta,
\]
and define the size of such a multi-index by
\[
|\gamma| = |\alpha| + \hh |\beta|,
\]
where $\hh$ is a positive integer that will be specified later and describes the balance between Klainerman vector fields and regular derivatives in our analysis.
We use these vector fields in order to define the  higher order  counterparts of the energy functional  \eqref{energy}:

\begin{itemize}
\item[a)] the energy $E^n(t, u,v)$ measures the regularity in the  function space $\mathcal{H}^n$ of the solutions,
\begin{equation}
\label{high energy}
E^n(t, u,v):= \sum_{\vert \alpha \vert \leq n} E\left(  t; \partial^{\alpha} u, \partial^{\alpha} v\right).
\end{equation}
\item[b)] the energy $E^{[n]}(t, u,v)$  keeps track of $Z$ vector fields applied to the solution in addition to regular derivatives,
\begin{equation}
\label{high vf energy}
E^{[n]}(t, u,v):= \sum_{|\gamma| \leq n} E\left(  t; \ZZ^\gamma u, \ZZ^\gamma v\right).
\end{equation}
\end{itemize}

The energy functional \eqref{energy} represents the natural energy of
the Klein-Gordon equation together with the energy of the wave
equation. The functional $E^n$ is the energy associated to the
differentiated variables and, as usual, helps us control the $L^2$ norm
of these variables, equivalently saying it represents the higher
order energy that controls the $H^{n+1}$ Sobolev norms of the solutions
for $n\geq 3$. The last energy functional $E^{[n]}$ represents the
energy associated to the system \eqref{WKG} to which we have also
applied Klainerman vector fields. Using these energies, we are now
able to state a more precise version of our main theorem:

\begin{theorem}\label{t:main}
Assume that the initial data $(u[0],v[0])$ for \eqref{WKG}
satisfies
\begin{equation}\label{small-data}
\| (u[0],v[0])\|_{\mathcal H^{2\hh}} + \| x \partial_x (u[0],v[0])\|_{\mathcal H^{\hh}}
+ \| x^2 \partial_x^2 (u[0],v[0])\|_{\mathcal H^0} \leq \epsilon \ll 1.
\end{equation}
Then the equation \eqref{WKG} is almost globally well-posed in $\mathcal H^{2\hh}$,
with $L^2$ bounds as follows:
\begin{equation}
E^{[2\hh]}(t,u,v) \lesssim \epsilon^2,
\end{equation}
and pointwise bounds
\begin{equation}
|  \partial^j v | \lesssim \epsilon \langle t+r \rangle^{-1} , \qquad j = \overline{0,3},
\end{equation}
\begin{equation}
|  \partial^j u | \lesssim \epsilon \langle t +r \rangle^{-\frac12}  \langle t-r \rangle^{-\frac12}, \qquad j = \overline{1,3},
\end{equation}
\begin{equation}
|  \partial^j Z  u | \lesssim \epsilon, \qquad j = \overline{0,2}.
\end{equation}
\end{theorem}

\begin{remark}
The pointwise bounds stated in the theorem represent baseline estimates. In fact we obtain slightly 
better bounds in various regimes. These gains will be made specific later in the last section. 
\end{remark}

In the successor to this paper, we combine the bounds of this paper
with  asymptotic analysis for both the wave and the Klein-Gordon equation
in order to convert the above result into a global result:

\begin{theorem}\label{t:main-extra}
Assume that the initial data $(u[0],v[0])$ for \eqref{WKG}
satisfies
\begin{equation}
\| (u[0],v[0])\|_{\mathcal H^{2\hh}} + \| x \partial_x (u[0],v[0])\|_{\mathcal H^{\hh}}
+ \| x^2 \partial_x^2 (u[0],v[0])\|_{\mathcal H^0} \leq \epsilon.
\end{equation}
Then the equation \eqref{WKG} is globally well-posed in $\mathcal H^{2\hh}$,
with $L^2$ bounds as follows:
\begin{equation}
E^{[2\hh]}(t,u,v) \lesssim \epsilon^2 t^{C\epsilon},
\end{equation}
and pointwise bounds
\begin{equation}
|  \partial^j v | \lesssim \epsilon \langle t+r \rangle^{-1} , \qquad j = \overline{0,3},
\end{equation}
\begin{equation}
|  \partial^j u | \lesssim \epsilon \langle t +r\rangle^{-\frac12}  \langle t-r \rangle^{-\frac12}, \qquad j = \overline{1,3},
\end{equation}
\begin{equation}
|  \partial^j Z  u | \lesssim \epsilon, \qquad j = \overline{1,2}.
\end{equation}
\end{theorem}

\subsection{The structure of the paper}
We begin in the next section with an overview of the main steps of the
proof. We follow a standard approach in which our proof has two main
steps, (i) vector field energy estimates, and (ii) pointwise bounds
derived from energy estimates (sometimes called Klainerman-Sobolev
inequalities). We depart from the standard setting in that our energy
estimates are space-time $L^2$ \emph{local energy bounds}, localized to dyadic regions
$C^{\pm}_{TS}$, where $T$ stands for dyadic time, $S$ for the dyadic
distance to the cone, and $\pm$ for the interior/exterior cone. Similarly,
our pointwise bounds are akin to Sobolev embeddings or interpolation inequalities
in the same type of regions.

The energy estimates are carried in the next three sections, in three steps:
(i) for the linearized equation, (ii) for the solution and its higher derivatives,
and finally (iii) for the vector fields applied to the solution.

Finally, the last section is devoted to the pointwise bounds, which are derived from the 
local energy bounds via interpolation inequalities in the same $C^{\pm}_{TS}$, with the extra step of 
also using the wave or Klein-Gordon equation in several interesting cases.

\subsection*{Acknowledgments }  The first author was supported by a Luce Assistant Professorship, by the Sloan Foundation, and by an NSF CAREER grant DMS-1845037. The second author was supported by an AMS Simons travel grant. We would like to thank Daniel Tataru for his suggestions and useful conversations that helped us improve the result and the presentation.

\section{An overview of the proof}

We begin with a prerequisite for the proof, which has to do with the local in time theory
for our evolution \eqref{WKG}. The three main properties, also summarized in Theorem~\ref{t:local}, are as follows:

\begin{enumerate}
\item Local well-posedness in $\H^4$ (also in $\H^n$ for  $n \geq 4$).

\item Continuation of $\H^4$ solutions for as long as $\partial^2(u,v)$ remains bounded, 
also with propagation of higher regularity, i.e.  bounds in $\H^n$ for all $n$. 

\item Uniform finite speed of propagation as long as $|\nabla v|$ stays pointwise small.
\end{enumerate}

Given these three facts, our proof is set up as a bootstrap argument, where the 
bootstrap assumption is on pointwise decay bounds for the solution.
These bounds are as follows:
\begin{equation}
\label{boot010}
 \vert Z  u\vert \leq C \epsilon\langle t-r\rangle^{\delta},
\end{equation}
\begin{equation}
\label{boot020}
\vert \partial u \vert  \leq C\epsilon \langle  t+r \rangle ^{-\frac{1}{2}} \langle t-r \rangle ^{-\frac{1}{2}},
\end{equation}
\begin{equation}
\label{boot01}
 \vert Z \partial^j u\vert \leq C \epsilon \langle t-r \rangle ^{-\delta_1}, \qquad j = \overline{1,2},
\end{equation}
\begin{equation}
\label{boot02}
\vert \partial^{j}  u \vert  \leq C\epsilon \langle  t+r \rangle ^{-\frac{1}{2}} \langle t-r \rangle ^{-\frac{1}{2}-\delta_1}, \qquad  j = \overline{2,3},
\end{equation}
\begin{equation}
\label{boot03}
\vert \partial^j v\vert  \leq C\epsilon \langle t+r \rangle ^{-1-\delta_1} \langle t-r \rangle ^{\delta_1}  , \quad j = \overline{1,3}.
\end{equation}

Here $0 < \delta \ll \delta_1$  are fixed small positive universal constant. On the other hand $C$ is a large universal constant,  
which will be improved as part of  the conclusion of the proof. The proof is structured 
into two main steps. For expository purposes we provide first a simplified outline of these two steps, and refine this 
later.

\bigskip 

{\bf 1. Energy estimates.}
Here one considers a solution to \eqref{WKG} in a time interval $[0,T_0]$, which is a-priori assumed to satisfy 
the bootstrap assumptions \eqref{boot010}, \eqref{boot020}, \eqref{boot01}, \eqref{boot02} and \eqref{boot03}. Then the conclusion is that 
the solution $(u,v)$ satisfies the following energy estimates in $[0,T_0]$:
\begin{equation}\label{main-energy}
E^{[2\hh]}(u,v)(t)  \lesssim \langle t\rangle ^{\tilde C \epsilon} E^{[2\hh]}(u,v)(0), \qquad t \in [0,T_0] .
\end{equation}
Here $\tilde C$ is a large constant which depends on $C$ in our bootstrap assumption, $\tilde C \approx C$.
However, the implicit constant in \eqref{main-energy} cannot depend on $C$. No restriction is imposed on  the lifespan bound $T_0$.

\bigskip 

{\bf 2. Pointwise bounds.}
Here we assume that we have a solution $(u,v)$ to \eqref{WKG} in a time interval $[0,T_0]$, which satisfies 
the energy  bounds
\begin{equation}\label{main-energy-re}
E^{[2\hh]}(u,v)(t)  \lesssim \epsilon \langle t\rangle ^{\tilde C \epsilon} , \qquad t \in [0,T_0] .
\end{equation}
Then we show that the solution $(u,v)$ satisfies the pointwise bounds
\begin{equation}
\label{get010}
 \Vert Z   u\Vert_{L^{\infty}} \leq  \epsilon \langle t\rangle ^{\tilde C \epsilon} , 
 \end{equation}
\begin{equation}
\label{get020}
\vert \partial u\vert  \leq  \epsilon \langle t\rangle ^{\tilde C \epsilon}  \langle t+r \rangle ^{-\frac{1}{2}} \langle t-r \rangle^{-\frac{1}{2}},
\end{equation}
\begin{equation}
\label{get01}
 \Vert Z \partial^j  u\Vert_{L^{\infty}} \leq  \epsilon \langle t\rangle ^{\tilde C \epsilon}  \langle t-r \rangle ^{-2\delta_1}, \qquad j = \overline{1,2},
\end{equation}
\begin{equation}
\label{get02}
\vert \partial^{j} u\vert  \leq  \epsilon \langle t\rangle ^{\tilde C \epsilon}  \langle t+r \rangle ^{-\frac{1}{2}} \langle t-r \rangle^{-\frac{1}{2}-2\delta_1}, \qquad  j = \overline{2,3},
\end{equation}
\begin{equation}
\label{get03}
 \vert \partial^j v\vert  \leq  \epsilon \langle t\rangle ^{\tilde C \epsilon}  \langle t+r\rangle ^{-1-2\delta_1}
 \langle t-r \rangle^{2\delta_1}
 , \quad j = \overline{0,3}.
\end{equation}
Here the lifespan $T_0$ is again arbitrary. 

\bigskip

In both steps, the time $T_0$ is arbitrary. However, in order to close the bootstrap argument one needs to recover \eqref{boot010}, \eqref{boot020}, 
\eqref{boot01}, \eqref{boot02} and \eqref{boot03} from \eqref{get010}, \eqref{get020}, \eqref{get01}, \eqref{get02} and \eqref{get03}.
This requires 
\[
T_0^{ \tilde C \epsilon} \ll C,
\]
which  is satisfied provided that
\[
T_0 \ll e^{\frac{c}{\epsilon}},
\]
i.e. our almost global result.

In the classical results on global or almost global well-posedness in
$3+1$ dimensions, one uses a large number of vector fields both in
the energy estimates and in the pointwise bounds, and the argument
works exactly as outlined above.  Notably, both steps require only
fixed time bounds, and the pointwise bounds are akin to an improved
form of the Sobolev embeddings, which are now referred to as
Klainerman-Sobolev estimates.

By contrast, such a strategy would be too naive in our work, both because 
we work in $2+1$ dimensions and there is less dispersive decay, and because our problem is strongly quasilinear.
Instead, a good portion of our analysis happens in space-time regions which are adapted to the light
cone geometry. Thus the next step is to describe our decomposition of the 
space-time.

We first consider a dyadic decomposition in time into sets
\begin{equation}
\label{ct}
C_T:=\left\{ T\leq t\leq 2T \right\}.
\end{equation}
 Further we dyadically decompose each of the $C_T$'s with respect to the size
of $t-r$, which  measures how far or close we are to the cone
\begin{equation}
\label{cts}
\begin{aligned}
&C_{TS}^{+}:=\left\{ (t,x)\, :\,   S\leq t-r\leq 2S, \, T\leq t\leq 2T \right\},  \mbox{ where } 1\leq S\lesssim T, \\
&C_{TS}^{-}:=\left\{ (t,x)\, :\,   S\leq r-t\leq 2S, \, T\leq t\leq 2T \right\}, \mbox{ where } 1\leq S \lesssim T,
\end{aligned}
\end{equation}
see Figure~\ref{f:D0}.

\begin{figure}[h]
\begin{center}
\begin{tikzpicture}[scale=1.5]

\draw[->] (0,-0.3) -- (0,3.5);
\draw[->] (-0.3,0) -- (5,0);
\node[below] at (4.9,0) {\small $r$};
\node[left] at (0,3.4) {\small $t$};

\draw [domain=0:3.5] plot(\x, \x);
\node[right] at (3.4,3.4) {\tiny $t=r$};

\draw[dashed] [domain=0:5] plot (\x, 1.5);
\node[left] at (0,1.5) {\small $T$};
\draw[dashed] [domain=0:5.3] plot (\x, 3);
\node[left] at (0,3) {\small $2T$};
 
\draw[red, thick] [domain=0.5:2] plot(\x+.5, \x+1);
 \draw[red, thick] [domain=.5:2] plot(\x, \x+1);
 \draw[red, thick] [domain=0:0.5] plot(\x+.5, 1.5);
 \draw[red, thick] [domain=1.5:2] plot(\x+.5, 3);
 
  \node[red,thick] at (0.9, 2.4) {\tiny $C^+_{TS}$};
 
 \draw[red, thick] [domain=2.5:4] plot(\x-.5, \x-1);
 \draw[red, thick] [domain=3.5:5] plot(\x-1, \x-2);
\draw[red, thick] [domain=2.5:3] plot(\x-.5, 1.5);
 \draw[red, thick] [domain=4:4.5] plot(\x-.5, 3);
   \node[red,thick] at (3.9, 2.4) {\tiny $C^{-}_{TS}$};

\fill[red!40,nearly transparent] (2,1.5) -- (2.5,1.5) -- (4,3) -- (3.5,3) -- cycle;

\fill[red!40,nearly transparent]  (0.5,1.5) -- (1,1.5) -- (2.5,3) -- (2, 3) -- cycle;

\end{tikzpicture}
\caption{1D vertical section of space-time regions $C^{\pm}_{TS}$}
\label{f:D0}
\end{center}
\end{figure}
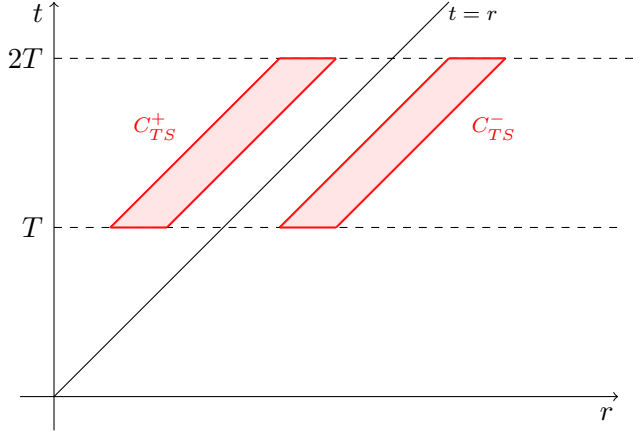

That still leaves the exterior region
\begin{equation}
\label{ctout}
C_T^{out} :=\left\{ T\leq t\leq 2T, \  r \gg T \right\}.
\end{equation}
Here $C^+_{TS}$ represents a spherically symmetric dyadic region
inside the cone with width $S$, distance $S$ from the cone, and time
length $T$.  $C^-_{TS}$ is the similar region outside the cone where,
far from the cone, we would have $T\lesssim S$. To simplify the
exposition we will use the notation $C_{TS}$ as a shorthand for either
$C^+_{TS}$ or $C^-_{TS}$. Such a decomposition has been introduced before 
by Metcalfe-Tataru-Tohaneanu~\cite{mtt} in a linear setting; we largely follow their notations.

In the above definition of the $C_{TS}$ sets we limit $S$ to $S \geq 1$ because our assumptions are invariant with respect to unit size translations. In particular,  this leaves out a conical shell region along the side of the cone $t=r$, which intersects both the interior and the exterior of the cone. To also include this region in our analysis we redefine 
\begin{equation}
\label{ct1}
C_{T1}:=\left\{ (t,x)\, :\,   \vert t-r\vert \leq 2, \, T\leq t\leq 2T \right\},  \mbox{ where } S\sim 1.
\end{equation}

This decomposition plays roles as follows in the two steps above:

\begin{enumerate}
\item While the energy estimate \eqref{main-energy} holds as stated, a key part of its proof involves 
separately estimating  the energy growth generated in each of the $C^{\pm}_{TS}$ set.
This in turn requires improved local energy  bounds for the solutions in such space-time regions.
On the upside, at the conclusion of this step we obtain not only fixed time energy estimates 
but also localized energy estimates in $C^{\pm}_{TS}$. 

\item The pointwise bounds in the second step are proved locally in each of the $C^{\pm}_{TS}$
regions, based on the local energy bounds there rather than the global fixed time bounds.
\end{enumerate}
 
One downside of the localizations required in our proof of the pointwise bounds is that it is rather delicate 
to close the arguments in the fixed time interval $[0,T_0]$, as it would require dealing with Sobolev type embeddings 
based on vector fields not necessarily compatible with the boundary. This is not a critical problem and there 
 are multiple ways to deal with it, for instance by choosing the boundaries more carefully than simply time slices.
Here, instead, we completely bypass the issue in a different way, by truncating the nonlinearity. 
Precisely, suppose we want to solve the  equation \eqref{WKG} up to some time $T_0$.
Then we consider a  smooth cutoff function $\chi_{T_0}$, which is supported in $[0, 2T_{0}]$ and equals $1$ in the time interval $[0, T_0]$. Thus,  $\chi_{T_0}$ selects the region $t < T_0$, and replaces the equation \eqref{WKG} with the truncated version  
\begin{equation} \label{WKG-cut}
\left\{
\begin{aligned}
&(\partial^2_t  - \Delta_x) u(t,x) = \chi_{T_0}(t) \left[  \mathbf{N_1}(v, \partial v)+ \mathbf{N_2}(u, \partial v)  \right]\,,  \\
&(\partial^2_t  - \Delta_x  + 1)v(t,x) =\chi_{T_0}(t)  \left[ \mathbf{N_1}(v, \partial_1 u) +  \mathbf{N_2}(u, \partial u)\right]\,
\end{aligned} \right.
 \quad (t,x)\in \left[0,+\infty \right) \times \mathbb{R}^2.
\end{equation}
Such a cutoff will make no difference in the proof, but instead insures that beyond time $2T_0$ the solution $(u,v)$
solves the corresponding linear constant coefficient problem.

This is similar to an idea introduced by Bourgain in the study of the semilinear dispersive equations \cite{bourgain}
 with a similar purpose, i.e. to avoid sharp time truncations in function spaces. 
 
Another feature of our proof is that we use the finite speed of propagation to isolate and consider separately 
the exterior region $\{ r \gg T\}$. Precisely, for large $R$ the problem localizes to the region 
\[
\{ |x| \approx R, \quad |t| \ll R \}.
\]
In this region the weights defining the initial data size are all constant, so it is enough to carry out 
standard energy estimates and obtain pointwise bounds via Sobolev embeddings at fixed time. 
This analysis is carried out in the next section, where we prove Theorem~\ref{t:local}. As a consequence of Theorem~\ref{t:local} and Sobolev embeddings, we immediately obtain the following:

\begin{proposition}\label{p:out}
Assume that the initial data $(u,v)[0]$ for \eqref{WKG} satisfy \eqref{data-main}. Then the 
equations \eqref{WKG} and \eqref{WKG-cut} are globally well-posed in the region $C^{out}:=\{ t \leq \frac{1+|x|}{4} \}$, with energy bounds
\begin{equation}
\| \partial^{\leq 2\hh} (u,v)[t] \|_{\mathcal H^{0}} + \| x \partial ^{\leq \hh}\partial (u,v)[t]\|_{\mathcal H^{0}}
+ \| x^2 \partial^2 (u,v)[t]\|_{\mathcal H^0} \leq \epsilon, 
\end{equation}
and pointwise bounds 
\begin{equation}
\begin{aligned}
&\|\langle x\rangle ^{1+\delta}\partial^{j} u \|_{L^\infty}   \lesssim \epsilon, \quad  j=\overline{2,3}, \\
&  \|\langle x \rangle ^{1+\delta} \partial^{j} v\|_{L^\infty} \lesssim \epsilon, \quad  j=\overline{0,3},
\end{aligned}
\end{equation}
and
\begin{equation}
\|\langle x\rangle \partial u \|_{L^\infty}  \lesssim \epsilon.
\end{equation}
\end{proposition}

For later use, we also state an alternative form of the above proposition, where the smallness assumption on the initial data
is replaced by bootstrap assumptions akin to \eqref{boot020}, \eqref{boot02}, \eqref{boot03} but restricted to the exterior region. Denoting 
\begin{equation}
\begin{aligned}
    E^{out, [2\hh]}(u,v)(t) = \ & \| \partial^{\leq 2\hh} (u,v)[t] \|_{\mathcal H^{0}(C^{out})} + \| x \partial ^{\leq \hh}\partial  (u,v)[t]\|_{\mathcal H^{0}(C^{out})}\\
& \qquad + \| x^2 \partial^2 (u,v)[t]\|_{\mathcal H^0(C^{out})}
\end{aligned}
\end{equation}
we have

\begin{proposition}\label{p:out-boot}
Let $(u,v)$ be a solution for \eqref{WKG} in $C^{out}$
which satisfies the bounds
\[
| \partial u|+|\partial v| \ll 1, \qquad |\partial^2 u|+|\partial^2 v| \ll \langle x \rangle^{-1}.
\]
Then we have the uniform global bounds
\[
 E^{out, [2\hh]}(u,v)(t) \lesssim E^{out, [2\hh]}(u,v)(0).
\]
\end{proposition}
As we will see in the next section, the proof of this proposition is a step in the proof of the previous proposition.

For the bulk part, where we track the evolution of the vector field energy $E^{[2\hh]}(u,v)$, we define 
a stronger norm $X^T$ for $(u,v)$ associated to dyadic time intervals, as well as a similar norm $Y^T$ 
for the right hand side of the equation. These norms will be introduced later in the paper; their definitions are given in \eqref{XT} for $X^T$, respectively in \eqref{YT} for $Y^T$. Then we replace the energy bound \eqref{main-energy}
with the stronger $X^T$ and $Y^T$ bounds:

\begin{proposition}\label{p:energy}
Let $(u,v)$ be a solution to \eqref{WKG} or \eqref{WKG-cut} in $[0,T_0]$ which satisfies the bootstrap bounds
\eqref{boot010}, \eqref{boot020}, \eqref{boot01}, \eqref{boot02} and \eqref{boot03}. Then we have 
\begin{equation}\label{main-energy+}
\| \ZZ^\gamma (u,v)\|_{X^T}  \lesssim \epsilon  T^{\tilde C \epsilon} , \qquad  |\gamma| \leq 2\hh, \quad T\in [0,T_0].
\end{equation}
In addition, 
\begin{equation}\label{main-RHS+}
\| \ZZ^\gamma (\Box u,(\Box+1)v)\|_{Y^T}  \lesssim \epsilon  T^{\tilde C \epsilon}, \qquad  |\gamma| \leq \hh,  \quad  T\in [0,T_0].
\end{equation}
\end{proposition}
The bounds in this  Proposition will be proved in Section~\ref{s:linearization} where we consider the linearized equations, in Section~\ref{s:higher} which is devoted to the higher energy estimates, and in Section~\ref{s:fields} where we establish the vector fields bounds. 

In this context, our pointwise bounds will be linear and localized to dyadic time regions:
\begin{proposition}\label{p:KS}
Let $(u,v)$ be functions in $C_T$ which satisfy the bounds 
\begin{equation}\label{main-energy++}
\| \ZZ^\gamma (u,v)\|_{X^T}  \leq 1, \qquad  |\gamma| \leq 2\hh,
\end{equation}
as well as
\begin{equation}\label{main-RHS++}
\| \ZZ^\gamma (\Box u,(\Box+1)v)\|_{Y^T}  \leq 1, \qquad  |\gamma| \leq \hh.
\end{equation}
Then we have  the pointwise bounds
\begin{equation}
\label{get01+0}
 \Vert Z  u\Vert_{L^{\infty}} \lesssim 1, 
\end{equation}
\begin{equation}
\label{get02+0}
\vert \partial  u \vert  \lesssim T^{-\frac12}  S^{-\frac12}, 
\end{equation}
\begin{equation}
\label{get01+}
 \Vert Z \partial^j  u\Vert_{L^{\infty}} \lesssim  S^{-2\delta_1},  \qquad j = \overline{1,2},
\end{equation}
\begin{equation}
\label{get02+}
\vert \partial^{j}  u \vert  \lesssim T^{-\frac12}  S^{-\frac12-2\delta_1}, \qquad  j = \overline{2,3},
\end{equation}
\begin{equation}
\label{get03+}
 \vert \partial^j v\vert  \lesssim T^{-1-2\delta_1} S^{2\delta_1}, \quad j = \overline{0,3}.
\end{equation}
\end{proposition}

Taken together, the last three propositions imply the conclusion of our main result in Theorem~\ref{t:main}.
This Proposition is proved in the last section of the paper.

\section{Local well-posedness, continuation and the 
 exterior region \texorpdfstring{$C^{out}$}{}}

In this section we prove Theorem~\ref{t:local}. As a consequence, we derive Proposition~\ref{p:out}.
\begin{proof}[Proof of Theorem~\ref{t:local}]
 The proof of this theorem is very similar to the proof of the local well-posedness for quasilinear wave equations (see for instance H\"ormander \cite{Hormander}, Sogge \cite{Sogge2008}, Racke \cite{Racke}), as well as the more modern treatment in Ifrim-Tataru \cite{MR4557379} . We sketch here its main steps. 
 \bigskip
 
 \textit{(i) Energy estimates and the quasilinear energy.}
 A key part of the argument is played by energy estimates, which we discuss
 here in a simpler setting, for the inhomogeneous linear problem 
 \begin{equation}
\label{WKG-lin0}
\left\{
\begin{aligned}
&(\partial_t^2-\Delta_{x})U(t,x)= \mathbf{N_1}(v, \partial V)  + \mathbf{N_2}(u, \partial V) + F \\
&(\partial_t^2-\Delta_{x}+1)V(t,x)=\mathbf{N_1}(v, \partial U) + \mathbf{N_2}(u, \partial U)+G,\\
\end{aligned}
\right.
\end{equation}
with initial conditions 
\begin{equation}\label{initial data}
\left\{
\begin{aligned}
& (U,V)(0,x) = (U_0(x),V_0(x))\,,\\
& (\partial_t U,\partial_t V)(0,x) =  (U_1(x), V_1(x)).
\end{aligned}
\right.
\end{equation}
At leading order this system agrees with the linearized equation discussed in the next section, Section~\ref{s:linearization}.
 
 Our starting point in the proof of the energy estimates is the energy functional associated to the corresponding linear equation
\[
E\left(U,V\right) := \frac{1}{2}\int_{\mathbb{R}^2} U_t ^2 + U_x^2 +V_t^2 + V_x^2 +V^2 \, dx= \int_{\mathbb{R}^2}e_0(t,x) \, dx,
\]
where $e_0$ is the linear energy density
\[
e_0(t,x)=\frac{1}{2}[U_t ^2 + U_x^2 +V_t^2 + V_x^2 +V^2 ].
\]
This would be the obvious candidate for the energy functional with
respect to which we would like to prove the energy estimates required for
the local well-posedness result.  However, the right hand side 
of the equations \eqref{WKG-lin0} contains second order derivatives of $(U,V)$,
so if one tries to prove energy bounds via this functional, there would be a loss of derivatives. This  is a common issue when working with
quasilinear nondiagonalisable hyperbolic systems of PDEs. To avoid
this loss of derivatives we consider a quasilinear type modification
of this energy, which has the form
\begin{equation}
\label{Equasidef}
E^{quasi}(U,V):= E\left(U,V\right)+ \int_{\mathbb{R}^2}B_1(v; U, V) + B_2(u; U, V)\, dx=\int_{\mathbb{R}^2} e^{quasi}(t,x) \, dx ,
\end{equation}
where the quasilinear energy density is
\[
e^{quasi}(t,x):=e_0(t,x)+B_1(v; U, V) +  B_2(u; U, V).
\]
Here the trilinear forms $B_1$ and $B_2$ are associated to the null forms 
$\mathbf N_1$, $\mathbf N_2$ in \eqref{WKG} in a linear fashion. Precisely,
 corresponding to the three bilinear forms in \eqref{null forms} we have the associated corrections
\begin{equation}\label{corrections0}
\left\{
\begin{aligned}
& B_{0i}(w; U,V):= w_t\,  U_i\, \partial V, \\
& B_{ij}(w; U, V):= w_iU_j\, \partial V - w_j \,U_i\, \partial V, \\
 & B_{0}(w; U, V):= - w_x\cdot U_x\, \partial V,
 \end{aligned}
 \right.
\end{equation}
for $w=u,v$.

We will use the above energy functional in order to study the well-posedness
of \eqref{WKG-lin0} in $\mathcal H^0$. For this we assume that $(u,v)$
are known, and that we control in a pointwise fashion  their associated control parameters $(A,B)$ introduced in \eqref{norm A}, \eqref{norm B}. Then we have
\begin{lemma}\label{l:quasi-lin}
Assume that $A \leq \delta \ll 1$ and $B \in L^\infty$.
Then the equation \eqref{WKG-lin0} is well-posed in $\H^0$  and the following properties hold:

(i) Energy equivalence:
\begin{equation}\label{e-equiv}
    E^{quasi}(t,U,V) = (1+O(\delta)) \| (U,V)[t]\|_{\mathcal H^0}^2.
\end{equation}

(ii) Energy estimate:
\begin{equation}\label{e-bd}
\frac{d}{dt}   E^{quasi}(t,U,V) \lesssim B(t) \| (U,V)[t]\|_{\mathcal H^0}^2  
+ \| (U,V)[t]\|_{\mathcal H^0} \| (F,G)[t]\|_{L^2}.
\end{equation}

\end{lemma}

\begin{proof} Part (i) is trivial. For clarity we prove (ii) in the homogeneous case i.e. when $F,G=0$, and leave the minor inhomogeneous adaptation to the reader. To see what is needed for the energy estimate computation we
begin by deriving the density flux relation associated to our energy
density $e^{quasi}$. We begin with the first component of $e^{quasi}$,
which is $e_0$:
\[
\partial_te_0(t,x)=\sum_{j=1}^2\partial_{x_j}\left( U_tU_j +V_tV_j\right) +U_t\Box U+V_t(\Box+1)V.
\]
The last two terms can be expanded as follows
\begin{equation}
\label{FG0}
\left\{
\begin{aligned}
&U_t\Box U= U_t\left( \mathbf{N_1}(v,\partial V) + \mathbf{N_2}(u, \partial V)\right) ,\\
&V_t(\Box+1)V= V_t\left( \mathbf{N_1}(v, \partial U)+\mathbf{N_2}(u, \partial U) \right).
\end{aligned}
\right.
\end{equation}

Next we turn our attention to the corrections
\begin{equation}
\label{dtB0}
 \partial_t B_i(w; U,V)=B_i(w_t; U, V)+ B_i(w; U_t, V)+ B_i(w; U, V_t),\qquad w=u,v,  \quad i=\overline{1,2}.
\end{equation}
Here we will combine the first, respectively the second, terms on both RHS in the above equations with $\partial_tB_i(v;U,V)$, respectively with $\partial_tB_i(u;U,V)$.
We obtain that
\[
\left\{
\begin{aligned}
& U_t \mathbf{N_1}(v, \partial V)+ V_t \mathbf{N_1}(v, \partial U) +  \partial_t B_1(v; U, V)=\partial_{x}C_1(\partial v, \partial U, \partial V)+ D_1(\partial^2 v, \partial U, \partial V),\\
&U_t \mathbf{N_2}(u, \partial V)+ V_t \mathbf{N_2}(u, \partial U) +  \partial_t B_2(u; U, V)=\partial_{x}C_2(\partial u, \partial U, \partial V)+ D_2(\partial^2 u, \partial U, \partial V),
\end{aligned}
\right.
\]
where $C_i$ and $D_i$ are trilinear forms. 
Their structure is unimportant here, 
but will be investigated later in the next section.

Summing up all these terms we obtain the following energy flux relation for
 solution to \eqref{WKG-lin0}:
\begin{equation}
\label{dteq0}
\partial_t e^{quasi}(t,x)=\sum_{j=1}^2\partial_j f_j+g,
\end{equation}
where the fluxes $f_j$ have schematically the expressions 
\begin{equation}
\label{fj0}
f_j=U_jU_t+ V_jV_t+ C_1(\partial v, \partial U, \partial V) +  C_2(\partial u, \partial U, \partial V),
\end{equation}
and the source $g$ has the form
\begin{equation}
\label{g0}
g= D_1(\partial^2 v, \partial U, \partial V)+ D_2(\partial^2 u, \partial U, \partial V).
\end{equation}
To complete the proof of the energy estimates we use the relation \eqref{dteq0}  to obtain
\[
\frac{d}{dt}\int_0^t E^{quasi}(t, U,V)=\int _0^t g\, dx,
\]
where it remains to bound the RHS. We bound the $\partial ^2u$ and $\partial^2 v$ factors in $g$ in $L^{\infty}$ by $B(t)$. The $\partial V$ and $\partial U$  factors are bounded by the energy. Then the conclusion of the Lemma follows.

\end{proof} 
 
 \bigskip
 
\textit{(ii) Local existence.} 
We construct a local solution to \eqref{WKG}-\eqref{data} using an iteration scheme. We set
\[
(u^{-1}, v^{-1})\equiv 0,
\]
and define $(u^m, v^m),\  m=0,1,\dots,$ inductively by
\begin{equation}\label{WKGm}
\left\{
\begin{aligned}
&(\partial^2_t  - \Delta_x) u^m(t,x) = \mathbf{N_1}(v^{m-1}, \partial v^m) +\mathbf{N_2}(u^{m-1}, \partial v^m)\,, \\
&(\partial^2_t  - \Delta_x  + 1)v^m(t,x) = \mathbf{N_1}(v^{m-1}, \partial u^m)+ \mathbf{N_2}(u^{m-1}, \partial u^m)\,,
\end{aligned} 
\right. \ \  (t,x)\in \left[0,+\infty \right) \times \mathbb{R}^2,
\end{equation}
with
\begin{equation} \label{data_m}
\left\{
\begin{aligned}
& (u^m,v^m)(0,x) =  (u_0(x),v_0(x))\,,\\
& (\partial_t u^m,\partial_t v^m)(0,x) =  (u_1(x), v_1(x)).
\end{aligned}
\right.
\end{equation}
We assume the data to be in $\mathcal{S}$ so that, by the local existence theorem for linear equations, the above system admits a $\mathcal{C}^\infty$ solution for every $m$. We can later remove this assumption by an approximation argument.

To set the notations we assume that at the initial time $t=0$ we have the Lipschitz bound 
\begin{equation}
    \label{a0}
A(0):=\sum_{\vert \alpha \vert =1}\Vert \partial^{\alpha} u(0)\Vert_{L^{\infty}} +\sum_{\vert \alpha \vert =1}\Vert \partial^{\alpha}v(0)\Vert_{L^{\infty}} \le \delta \ll 1,
\end{equation}
as well as the Sobolev bound
\begin{equation}
\label{v0u0}
\|(u, v)[0]\|_{\H^n}\le M.
\end{equation}

Then we claim that there exists a time $T_0$ sufficiently small depending on $\delta$ and $M$ such that the following bounds for the sequence $ (u^m, v^m)$ hold for all $t\in [0, T_0]$:
\begin{equation}
    \label{am}
A^{m}(t):=\sum_{\vert \alpha \vert =1}\Vert \partial^{\alpha} u^{m}(t)\Vert_{L^{\infty}} +\sum_{\vert \alpha \vert =1}\Vert \partial^{\alpha}v^{m}(t)\Vert_{L^{\infty}} \leq 2\delta, 
\end{equation}
as well as the  uniform energy bounds
\begin{equation} \label{bound u_m v_m}
\|(u^m(t), v^m(t))\|_{\H^n}\le 2 M \qquad \forall \ 0\le t\le T_0,\ \forall m\ge 0.
\end{equation}
When $m=0$, the functions $(u^0, v^0)$ solve the linear system \eqref{WKG-lin}, for which the energy is conserved. Bound \eqref{bound u_m v_m} is hence trivially satisfied. As for $A^0$, we use Sobolev embeddings
 \[
A^0( t) \leq A^0(0) + \int_{0}^t \|(\partial u^0_t,\partial v^0_t)(s)\|_{L^\infty} \, ds
\leq \delta + CMT_0,
\]
where $C$ is some positive universal constant. Then we can choose and $T_0$ small enough (e.g. $T_0<\delta/(CM)$) to obtain \eqref{am}.

Let us now suppose that our claim holds true for $m-1$, $m\ge 1$, and prove it for index $m$.

To measure $\|(u^m(t), v^m(t))\|_{\H^n}^2$ we will use the energies
\begin{equation}\label{E_m}
E^{quasi, n}(t, u^m, v^m)= \sum_{|\alpha| \leq n} E^{quasi}_{m-1}(t, D^\alpha_x u^m, D^\alpha_x v^m),
\end{equation}
where $E^{quasi}_{m-1}$ is obtained from $E^{quasi}$ by substituting 
$(u,v)$ with $(u^{m-1},v^{m-1})$. Thanks to the smallness assumption on $A^{m-1}$, 
the bound \eqref{e-equiv} holds for $E^{quasi}_{m-1}$,
so we will harmlessly replace $ \|(u^m, v^m)[t]\|_{\H^n}^2$ by 
$E^{quasi,n}(t, u^m, v^m)$ in \eqref{v0u0} and \eqref{bound u_m v_m}.

We start by differentiating the system \eqref{WKGm}.
For any $0 \leq k\le n$, the differentiated variables $(\partial^k_x u^m, \partial^k_x v^m)$ solve the following system:
\[
\left\{
\begin{aligned}
&(\partial^2_t  - \Delta_x) \partial^k_x u^m(t,x) = \mathbf{N_1}(v^{m-1}, \partial \partial^k_x v^m) + \mathbf{N_2}(u^{m-1}, \partial \partial^k_x v^m) +\mathbf F_k  \\
&(\partial^2_t  - \Delta_x  + 1) \partial^k_x v^m(t,x) = \mathbf{N_1}(v^{m-1}, \partial \partial^k_x u^m)+\mathbf{N_2}(u^{m-1}, \partial \partial^k_x u^m)
+\mathbf G_k\,,
\end{aligned} 
\right.
\]
where $\mathbf{F}_k$ and $\mathbf{G}_k$ are given by
\[
\left\{
\begin{aligned}
&\mathbf{F}_k:= \sum_{\substack{i+j=k\\ j<k}} \mathbf{N_1}(\partial^i_x v^{m-1}, \partial \partial^j_x v^m) + \sum_{\substack{i+j=k\\ j<k}} \mathbf{N_2}(\partial^i_x u^{m-1}, \partial \partial^j_x v^m)\\
& \mathbf{G}_k:=\sum_{\substack{i+j=k\\ j<k}} \mathbf{N_1}(\partial^i_x v^{m-1}, \partial \partial^j_x u^m)+ \sum_{\substack{i+j=k\\ j<k}} \mathbf{N_2}(\partial^i_x u^{m-1}, \partial \partial^j_x u^m) .
\end{aligned}
\right.
\]
We seek to apply Lemma~\ref{l:quasi-lin} for this system. By Sobolev embeddings 
we control 
\[
B^{m-1}:= B(u^{m-1},v^{m-1}) \lesssim M ,
\]
therefore by \eqref{e-bd} we have
\begin{equation} \label{ee-k}
\frac{d}{dt}E^{quasi}_{m-1}(\partial^k_x u^m, \partial^k_x v^m)
\lesssim M \| (\partial^k_x u^m, \partial^k_x v^m)\|_{\H^0}^2
+ \| (\partial^k_x u^m, \partial^k_x v^m)\|_{\H^0} 
\|(\mathbf F_k, \mathbf G_k)\|_{L^2}
\end{equation}
We claim that $(\mathbf F_k, \mathbf G_k)$ can be estimated as follows:
\begin{equation}\label{fg-k}
\|(\mathbf F_k, \mathbf G_k)\|_{L^2} \lesssim M \|(u^m,v^m)\|_{\H^n}, \qquad 
k \leq n.
\end{equation}
Indeed, using H\"older inequality and the Gagliardo-Niremberg interpolation inequality we see that
\begin{equation}\label{GN}
\begin{aligned}
\| \mathbf{N}(\partial^i_x v^{m-1}, \partial \partial^j_x v^m)\|_{L^2} &\le \|\partial \partial^i_x v^{m-1}\|_{L^{\frac{2(k-1)}{i-1}}} \|\partial \partial^{j+1}_x v^m \|_{L^{\frac{2(k-1)}{j}}} \\
& \le  \|\partial \partial^k_x v^{m-1}\|^{\alpha}_{L^2}\|\partial^2 v^{m-1}\|^{1-\alpha}_{L^\infty}   \|\partial \partial^{k}_x v^{m}\|^{\beta}_{L^2}\|\partial^2 v^{m}\|^{1-\beta}_{L^\infty}
\end{aligned}
\end{equation}
with $\alpha=\frac{i-1}{k-1}, \beta=\frac{j}{k-1}$,
and by Sobolev embeddings
\[
\begin{aligned}
& \| \mathbf{N}(\partial^i_x v^{m-1}, \partial \partial^j_x v^m)\|_{L^2}\\
& \hspace{10pt} \le \|(u^{m-1}, v^{m-1})(\tau)\|^\alpha_{\H^k}  \|(u^{m-1}, v^{m-1})(\tau)\|^{1-\alpha}_{\H^3}  \|(u^{m}, v^{m})(\tau)\|^\beta_{\H^k} \|(u^{m}, v^{m})(\tau)\|^{1-\beta}_{\H^3}\\
& \hspace{10pt} \le \|(u^{m-1}, v^{m-1})(\tau)\|_{\H^n}  \|(u^{m}, v^{m})(\tau)\|_{\H^n}.
\end{aligned}
\]
This estimate applies for $\mathbf{N}=\mathbf{N_1}$, and also for $\mathbf{N}=\mathbf{N_2}$ where $v$ can be freely replaced by $u$. This proves \eqref{fg-k}.
We substitute \eqref{fg-k} in \eqref{ee-kc} and sum over $k \leq n$. 
Then we obtain the energy relation
\[
\frac{d}{dt} E^{quasi,n}(u^m,v^m) \lesssim M \| (u^m,v^m)\|_{\H^n}^2
\approx M E^{quasi,n}(u^m,v^m),
\]
and by Gronwall's lemma
\[
E^{quasi,n}(t, u^m, v^m)\le e^{CMt} E^{quasi}(0, u^m, v^m)  \qquad \forall \ 0\le t\le T_0
\]
for some positive constant $C$. Then we can choose and $T_0$ small enough (e.g. $T_0<1/(2CM)$) to obtain \eqref{bound u_m v_m}. 

On the other hand, the uniform bound of $(\partial u^m,\partial v^m)$ is proved as for the case $m=0$ using Sobolev embeddings,
 \[
A^m( t) \leq A^m(0) + \int_{0}^t \|(\partial u^m_t,\partial v^m_t)(s)\|_{L^\infty} \, ds
\leq \delta + 2CMT_0,
\]
which proves that $A^m(t) \leq 2\delta$ if $T_0$ is chosen small enough.

The  remaining step is to prove the convergence of the sequence of approximate solutions $(u^m, v^m)$ as $m\rightarrow\infty$. As the problem is quasilinear, the convergence can only be shown in  a weaker topology. It is enough for our goal to show that $(u^m-u^{m-1}, v^m-v^{m-1})$ is a Cauchy sequence in $C^0([0,T_0]; \H^0)$. The limit $(u,v)$ will hence automatically belong to $\mathcal{H}^n$, satisfy \eqref{WKG} together with the uniform in time bound
\[
\|(u,v)(t)\|_{\H^n}\le 2M \quad \forall 0\le t\le T_0.
\]
From \eqref{WKGm} we see that the differences 
$(\tu^m,\tv^m) = (u^m-u^{m-1}, v^{m}-v^{m-1})$ solve the following Cauchy problem
\[
\left\{
\begin{aligned}
&\Box \tu^m(t,x) = \mathbf{N_1}(v^{m-1}, \partial \tv^{m}) +
\mathbf{N_2}(v^{m-1}, \partial \tv^{m})
+ \mathbf{N_1}(\tv^{m-1}, \partial v^{m}) + \mathbf{N_2}(\tv^{m-1}, \partial v^{m}) \,, \\
&(\Box  + 1)\tv^m(t,x) = 
\mathbf{N_1}(v^{m-1}, \partial \tu^{m}) +
\mathbf{N_2}(v^{m-1}, \partial \tu^{m})
+ \mathbf{N_1}(\tu^{m-1}, \partial v^{m}) + \mathbf{N_2}(\tu^{m-1}, \partial v^{m})
\end{aligned} 
\right.
\]
with initial data $(\tu^m,\tv^m)[0] = 0$.

We now view the last two terms in each equation as source terms, and apply 
Lemma~\ref{l:quasi-lin} to obtain
\[
\frac{d}{dt} E^{quasi}_{m-1} (\tu^m,\tv^m) \lesssim 
M (\|  (\tu^m,\tv^m)\|_{\H^0}^2 + \|  (\tu^m,\tv^m)\|_{\H^0} \|  (\tu^{m-1},\tv^{m-1})\|_{\H^0}).
\]
Then by Gronwall's inequality we obtain
\begin{align*}
&\|(\tu^m, \tv^m)(t)\|_{\H^0}  \le CM e^{CMT_0}\int_0^t  \|(\tu^{m-1}, \tv^{m-1})(\tau)\|_{\H^0}\, d\tau
\end{align*}
for all $0\le t \le T_0$. 
By iteration
\[
\begin{aligned}
\|(\tu^m , \tv^m)(t)\|_{\H^0} &\lesssim (CM)^m e^{mCMT_0} \int_{0\le \tau_1\le \tau_2 \le \dots\le \tau_m\le t}  \|(u^0, v^0)(\tau_1)\|_{\H^0}\, d\tau_1 \dots d\tau_m \\
& \le \frac{(CMt)^m}{m!}e^{mCMT_0}\sup_{t\in[0,T_0]}\|(u^0, v^0)(t)\|_{\H^0},
\end{aligned}
\]
which implies that the series of general term  $(u^m, v^m)$ converges in $\mathcal{C}^0([0,T_0]; \H^0)$ and concludes the proof of the existence part $(a)$ of the theorem.

\bigskip
\textit{(iii) Uniqueness of solutions.} This follows by the  same arguments as above. We assume we have two solutions of \eqref{WKG}, $(u^1,v^1)$, $(u^2,v^2)$, we subtract them, and obtain a similar system as above for the difference $(\tu, \tv):=(u^1,v^1)-(u^2,v^2)$ with zero Cauchy data. Then we apply the energy estimates in Lemma~\ref{l:quasi-lin} followed by Gronwall's inequality to show $(\tu,\tv)=0$.
For more details, see the proof in (iv) below which yields a stronger result.

\bigskip

\textit{(iv) Uniform finite speed of propagation.} 
Here we consider two solutions of \eqref{WKG}, $(u^1,v^1)$ and  $(u^2,v^2)$.
We assume that their initial data coincide in a ball $B(x_0,R)$, and show that the two solutions  have to agree in the cone
\[
C = \{ 2t + |x-x_0| < R  \}.
\]
For the difference $(\tu,\tv)$ we have the equation
\[
\left\{
\begin{aligned}
&\Box \tu(t,x) = \mathbf{N_1}(v^{1}, \partial \tv) +
\mathbf{N_2}(u^{1}, \partial \tv)
+ \mathbf{N_1}(\tv, \partial v^{2}) + \mathbf{N_2}(\tu, \partial v^2) \,, \\
&(\Box  + 1)\tv(t,x) = 
\mathbf{N_1}(v^1, \partial \tu) +
\mathbf{N_2}(u^1, \partial \tu)
+ \mathbf{N_1}(\tv, \partial u^2) + \mathbf{N_2}(\tu, \partial u^{2}).
\end{aligned} 
\right.
\]
We view the last two terms on the right as source terms and the rest as 
the equation \eqref{WKG-lin0} with $(u,v) = (u^1,v^1)$ and $(U,V)= (\tu,\tv)$.
Then the energy flux relation \eqref{dteq0} remains valid, with the contribution of the source terms included in $D_1$ and $D_2$ in \eqref{g0}.

We integrate the energy flux relation  \eqref{dteq0} on the cone section
$C_{[0,t_0]}=C\cap [0,t_0]$ to obtain 
\[
\int_{C_{t_0}} e^{quasi}\, dx = \int_{C_0} e^{quasi}\, dx + \int_{C_{[0,t_0]}} g\, dx dt
+ F
\]
where the flux $F$ is an integral over the lateral surface of the cone section
which we denote by $\partial C_{[0,t]}$,
\[
F = \int_{\partial C_{[0,t_0]}} -  e^{quasi} + \frac12 \frac{(x-x_0)_j}{|x-x_0|} f_j \, dx.
\]
Since $A \ll 1$, it easily follows that the contribution of the cubic terms to $F$
is negligible and then that $F \leq 0$. Then 
\[
\int_{C_{t_0}} e^{quasi} \, dx \leq  \int_{C_0} e^{quasi} \, dx + B \int_{ C_{[0,t_0]}}
e^{quasi} \, dx dt.
\]
At the initial time $t = 0$ we have $(\tu,\tv) = 0$ so by Gronwall's inequality we obtain 
$e^{quasi} = 0$ inside $C$, which gives $(\tu,\tv) = 0 $ in $C$.

\bigskip

\bigskip

\textit{(v) Continuation of the solution}. We start by differentiating the system \eqref{WKGm}.
For any $0 \leq k\le n$, the differentiated variables $(\partial^k_x u, \partial^k_x v)$ solve the following system:
\[
\left\{
\begin{aligned}
&(\partial^2_t  - \Delta_x) \partial^k_x u(t,x) = \mathbf{N_1}(v, \partial \partial^k_x v) + \mathbf{N_2}(u, \partial \partial^k_x v) +\mathbf F_k  \\
&(\partial^2_t  - \Delta_x  + 1) \partial^k_x v(t,x) = \mathbf{N_1}(v, \partial \partial^k_x u)+\mathbf{N_2}(u, \partial \partial^k_x u)
+\mathbf G_k\,,
\end{aligned} 
\right.
\]
where $\mathbf{F}_k$ and $\mathbf{G}_k$ are given by
\[
\left\{
\begin{aligned}
&\mathbf{F}_k:= \sum_{\substack{i+j=k\\ j<k}} \mathbf{N_1}(\partial^i_x v, \partial \partial^j_x v) + \sum_{\substack{i+j=k\\ j<k}} \mathbf{N_2}(\partial^i_x u, \partial \partial^j_x v)\\
& \mathbf{G}_k:=\sum_{\substack{i+j=k\\ j<k}} \mathbf{N_1}(\partial^i_x v, \partial \partial^j_x u)+ \sum_{\substack{i+j=k\\ j<k}} \mathbf{N_2}(\partial^i_x u, \partial \partial^j_x u) . 
\end{aligned}
\right.
\]
We seek to apply Lemma~\ref{l:quasi-lin} for this system. By \eqref{e-bd} we have
\begin{equation} \label{ee-kc}
\frac{d}{dt}E^{quasi}(\partial^k_x u, \partial^k_x v)
\lesssim B \| (\partial^k_x u, \partial^k_x v)\|_{\H^0}^2
+ \| (\partial^k_x u, \partial^k_x v)\|_{\H^0} 
\|(\mathbf F_k, \mathbf G_k)\|_{L^2}.
\end{equation}
It suffices to show  that $(\mathbf F_k, \mathbf G_k)$ can be estimated as follows:
\begin{equation}\label{fg-kc}
\|(\mathbf F_k, \mathbf G_k)\|_{L^2} \lesssim B \|(\partial^k u,\partial^k v)\|_{\H^0}, \qquad 
k \leq n.
\end{equation}
Indeed, using H\"older inequality and the Gagliardo-Niremberg interpolation inequality we see that
\begin{equation}\label{GN+}
\begin{aligned}
\| \mathbf{N}(\partial^i_x v, \partial \partial^j_x v)\|_{L^2} &\le \|\partial \partial^i_x v\|_{L^{\frac{2(k-1)}{i-1}}} \|\partial \partial^{j+1}_x v \|_{L^{\frac{2(k-1)}{j}}} \\
& \le  \|\partial \partial^k_x v\|^{\alpha}_{L^2}\|\partial^2 v\|^{1-\alpha}_{L^\infty}   \|\partial \partial^{k}_x v\|^{\beta}_{L^2}\|\partial^2 v\|^{1-\beta}_{L^\infty}
\end{aligned}
\end{equation}
with $\alpha=\frac{i-1}{k-1}, \beta=\frac{j}{k-1}$.
Since $\alpha +\beta =1$
\[
\begin{aligned}
& \| \mathbf{N}(\partial^i_x v, \partial \partial^j_x v)\|_{L^2}  \le \| (\partial^2 u, \partial^2v )\|_{L^{\infty}}  \|(\partial^ku, \partial^kv)\|_{\H^0}.
\end{aligned}
\]
This estimate applies for $\mathbf{N}=\mathbf{N_1}$, and also for $\mathbf{N}=\mathbf{N_2}$ where $v$ can be freely replaced by $u$. This proves \eqref{fg-kc}.

We substitute to obtain the energy relation
\[
\frac{d}{dt} E^{quasi}(\partial^k_xu,\partial^k_xv) \lesssim B \| (\partial^k u,\partial^kv)\|^2_{\H^0}
\approx B E^{quasi}(\partial^k_x u,\partial^k_xv),
\]
and by Gronwall's lemma
\[
E^{quasi}(t, \partial^k_xu, \partial^k_x v)\le e^{CBt} E^{quasi}(0, \partial^k_xu, \partial^k_x v)  \qquad \forall \ 0\le t\le T
\]
for some positive constant $C$. 

\end{proof}

\begin{proof}[Proof of Proposition~\ref{p:out}]

The exterior region corresponds to $t \leq \frac{1+\vert x\vert }{4}$.  Because of the finite
speed of propagation property, this region can be treated separately as long as
$\partial u$ and $\partial v$ remain small. Precisely, fix a dyadic $R > 1$
and consider the solution $(u,v)$ to \eqref{WKG} in the region 
\[
C_R^{out} = \left\{ R < |x| < 2R, \ 0 \leq t < \frac{1+|x|}{4} \right\}.
\]
By the finite speed of propagation previously proved,
the solution in this region is uniquely determined by the data 
in 
\[
A_R = \{ R/2 \leq |x| \leq 4R \}.
\]
In the region $A_R$ our hypothesis guarantees that 
we have the bounds
\begin{equation}\label{R-est}
    \| (u,v)[0]\|_{\H^{2h}} \lesssim \epsilon, \qquad \| (\partial^2 u,\partial^2 v)[0]\|_{\H^0} \lesssim \epsilon R^{-2}.
\end{equation}

We first restrict the data to $A_R$ and then extend them to all $\R^2$ 
so that \eqref{R-est} still holds. To obtain a bound in $C_R^{out}$ it suffices
to solve the equation up to time $T_R = R/2$. A local in time solution exists.
For this solution we make the bootstrap assumption 
\begin{equation}
\| \partial(u,v)\|_{L^\infty} \leq \sqrt \epsilon, \qquad \|  \partial^2(u,v)\|_{L^\infty} \leq  R^{-1} \sqrt \epsilon
\end{equation}
in a time interval $[0,T]$ with $T \leq T_R$. Applying the energy estimates in 
Theorem~\ref{t:local}, we can propagate the energy bounds in \eqref{R-est} 
up to time $T$. Then we can use Sobolev embeddings to get pointwise bounds from 
the energy bounds. 

For the first derivatives this yields
\[
\|\partial u\|_{L^\infty} \lesssim \|\partial u\|_{L^2}^\frac12
\|\partial^3 u\|_{L^2}^\frac12 \lesssim \epsilon R^{-1},
\]
and similarly for $v$. This suffices in order to improve the bootstrap assumption.

For the second derivatives this yields
\[
\|\partial^2 u\|_{L^\infty} \lesssim  \log R \|\partial^3 u\|_{L^2} + R^{-2} \|\partial^2 u\|_{H^2} \lesssim \epsilon R^{-2}\log R,
\]
and similarly for $v$. This again suffices in order to improve the bootstrap assumption.

As the bootstrap assumption can be improved for all $T < T_R$, it follows that the solution $(u,v)$ extends to time $T_R$ and satisfies the above pointwise bounds. 
The proof of the proposition is concluded.

\end{proof}
\section{Linearized equation: Alinhac's approach}
\label{s:linearization}
In this section we derive the linearized wave-Klein-Gordon system and
prove the energy estimates for them. More precisely we prove quadratic energy estimates in $\mathcal{H} ^0$, which apply
to large data problem. 

The solutions for the linearized wave-Klein-Gordon system around a
solution $(u,v)$ are denoted by $(U, V)$. With this notation in place
the linearized system takes the form
\begin{equation}
\label{general linearization}
\left\{
\begin{aligned}
&(\partial_t^2-\Delta_{x})U(t,x)= \mathbf{N_1}(v, \partial V) +\mathbf{N_1}(V,\partial v)  + \mathbf{N_2}(u, \partial V) +\mathbf{N_2}(U,\partial v) \\
&(\partial_t^2-\Delta_{x}+1)V(t,x)=\mathbf{N_1}(v, \partial U)+\mathbf{N_1}(V, \partial u) + \mathbf{N_2}(u, \partial U)+\mathbf{N_2}(U, \partial u)  .\\
\end{aligned}
\right.
\end{equation}
We recall that for $\mathbf{N_1}(\cdot, \cdot)$  and $\mathbf{N_2}(\cdot, \cdot)$ we will take linear combinations of the classical admissible quadratic null forms 
\begin{equation}
\label{null forms-re}
\left\{
\begin{aligned}
&Q_{ij}(\phi, \psi) =\partial_i  \phi \partial_j \psi -\partial_i\psi  \partial_j \phi, \\
&Q_{0i}(\phi, \psi) =\partial_t  \phi \partial_i \psi -\partial_t\psi  \partial_i \phi, \\
&Q_{0}(\phi, \psi) =\partial_t \phi   \partial_t \psi -\nabla_x \psi \cdot \nabla_x   \phi.
\end{aligned}
\right.
\end{equation}

To obtain energy estimates for the linearized equation we consider the same energy and energy density as in the proof of the local well-posedness result,
\begin{equation}
\label{Equasidef-re}
E^{quasi}(U,V):= E\left(U,V\right)+ \int_{\mathbb{R}^2}B_1(v; U, V) + B_2(u; U, V)\, dx=\int_{\mathbb{R}^2} e^{quasi}(t,x) \, dx ,
\end{equation}
where the quasilinear energy density is
\[
e^{quasi}(t,x):=e_0(t,x)+B_1(v; U, V) +  B_2(u; U, V),
\]
and  $e_0$ is the linear energy density
\[
e_0(t,x)=\frac{1}{2}[U_t ^2 + U_x^2 +V_t^2 + V_x^2 +V^2 ].
\]

\smallskip

Our energy estimates will be proven under the following uniform bound
assumptions (which are a part of our bootstrap argument) 
\begin{equation} \label{boot1.1}
|Zu|\le C\epsilon \langle t-r\rangle^{\delta} 
\end{equation}
\begin{equation}
\label{boot1}
|Z \partial u |\leq C \epsilon\langle t-r\rangle^{-\delta_1} 
\end{equation}
\begin{equation}
\label{boot1.2}
| Z\partial^2u|\le C\epsilon \langle t-r\rangle^{-\delta_1}
\end{equation}
\begin{equation}\label{boot2.1}
\vert \partial u\vert \leq C\epsilon \langle t+r \rangle^{-\frac{1}{2}}\langle t-r\rangle^{-\frac{1}{2}}
\end{equation}
\begin{equation}
\label{boot2}
\vert \partial^{j} u\vert  \leq C\epsilon \langle t+r \rangle^{-\frac{1}{2}} \langle t-r \rangle ^{-\frac{1}{2}-\delta_1}, \quad j=\overline{2,3},
\end{equation}
\begin{equation}
\label{boot3}
\vert \partial^j v\vert  \leq C\epsilon \langle t+r \rangle^{-1-\delta_1} \langle t-r \rangle^{\delta_1}, \quad j=\overline{1,3},
\end{equation}
where $C$ is a large positive constant and $0<\delta \ll \delta_1$. 

We observe that under the above assumptions, and for $\epsilon$ sufficiently small, we have 
\[
E^{quasi} (U, V) \approx  E(U,V)
\]
in the sense that
\[
\frac{1}{2} E^{quasi} (U,V) \leq E(U,V)\leq 2E^{quasi} (U,V).
\]

Our main result for the linearized equation  is as follows:

\begin{proposition} \label{p:linear} Assume the solutions to the main
  equations \eqref{WKG} or \eqref{WKG-cut} satisfy the bounds \eqref{boot1.1}-\eqref{boot3} in some time interval $[0, T]$. Then
  the linearized equation \eqref{general linearization}, subject to
  the constraints in \eqref{null forms-re}, is well-posed in $[0, T]$ and
  the solution satisfies
\begin{equation}
\label{est1}
E^{quasi} (U, V)(t)\lesssim t^{\tilde{C}\epsilon}E^{quasi}(U,V)(0), \quad t\in [0,T],
\end{equation} 
where $\tilde{C} \  \approx C$ is a positive constant.
\end{proposition}

Along the way we will establish a larger family of bounds for $U,V$. These are collected together 
at the end of the section in Corollaries~\ref{c:linear}, \ref{c:nonhlinear}, which can be viewed as a stronger form of 
the above proposition.

\begin{proof}
  The difficulty we encounter here is that we do not have a good
  estimate at a fixed time for the time derivative of the energy in
  order to directly apply a Gronwall's type inequality. To address this
  issues, the key idea is to obtain a ``good'' energy inequality. We
  are led to consider the energy growth on dyadic time scales $[T,
  2T]$. Within such a dyadic time interval it suffices to prove
\begin{equation}
\label{quasi}
\sup_{t\in [T, 2T]}E^{quasi} (U,V)(t) \le (1+\epsilon C)E^{quasi} (U, V)(T).
\end{equation}
As a preliminary step, we determine the growth of the $E^{quasi}(U,V)(t)$ on such time interval:
\begin{equation}
\label{quasi1}
E^{quasi}(U,V)(\tilde{T})-E^{quasi}(U,V)(T)=\int_T^{\tilde{T}}\frac{d}{dt}E^{quasi}(U,V)(t)\, dt \qquad \tilde{T}\in [T, 2T],
\end{equation}
where, after expanding, the RHS is a trilinear form integrated in
space time, rather than at fixed time. To estimate this integral,
i.e. to get
\begin{equation}
\label{quasi2}
\left| \int_T^{\tilde{T}}\frac{d}{dt}E^{quasi}(U,V)(t)\, dt\right| \lesssim \epsilon E^{quasi}(U,V) (T),
\end{equation}
we will first need to obtain the $L^2$ space-time bounds for $U$ and
$V$ and their derivatives over various space-time regions relative to
the distance to the cone.

To understand what is needed for the energy estimate computation we
begin by derivating the density flux relation associated to our energy
density $e^{quasi}$. We begin with the first component of $e^{quasi}$,
which is $e_0$:
\[
\partial_te_0(t,x)=\sum_{j=1}^2\partial_{x_j}\left( U_tU_j +V_tV_j\right) +U_t\Box U+V_t(\Box+1)V.
\]
The last two terms can be expanded as follows
\begin{equation}
\label{FG}
\left\{
\begin{aligned}
&U_t\Box U= U_t\left( \mathbf{N_1}(v, \partial V)+\mathbf{N_1}(V, \partial v) + \mathbf{N_2}(u, \partial V)+\mathbf{N_2}(U, \partial v) \right) ,\\
&V_t(\Box+1)V= V_t\left( \mathbf{N_1}(v, \partial U)+\mathbf{N_1}(V, \partial u)+\mathbf{N_2}(u, \partial U)+\mathbf{N_2}(U, \partial u)\right).
\end{aligned}
\right.
\end{equation}

Next we turn our attention to the quasilinear correction
\begin{equation}
\label{dtB}
\begin{split}
\partial_t B_1(v; U,V)= & \ B_1(v_t; U, V)+ B_1(v; U_t, V)+ B_1(v; U, V_t)
\\
\partial_t B_2(u; U,V)= & \ B_2(u_t; U, V)+ B_2(u; U_t, V)+ B_2(u; U, V_t).
\end{split}
\end{equation}
Here we will combine the first, respectively third, terms in both RHS in equations \eqref{FG} with $\partial_tB_1(v;U,V)$, respectively with $\partial_tB_2(u;U,V)$.
Schematically, we obtain that
\[
\left\{
\begin{aligned}
&U_t \mathbf{N_1}(v, \partial V)+ V_t \mathbf{N_1}(v,\partial U) +  \partial_t B_1(v; U, V)=\partial_{x}C_1(\partial v, \partial U, \partial V)+ D_1(\partial^2 v, \partial U, \partial V)
\\
&U_t \mathbf{N_2}(u, \partial V)+ V_t \mathbf{N_2}(u, \partial U) +  
\partial_t B_2(u; U, V)=\partial_{x}C_2(\partial u, \partial U, \partial V)+ D_2(\partial^2 u, \partial U, \partial V),
\end{aligned}
\right.
\]
where $C_1,C_2$ and $D_1,D_2$ are algebraic trilinear forms. Here we  need  
to take a closer look at the structure of $D_1$ and $D_2$. Indeed, a simple 
direct computation shows that both of them have a null structure,
\begin{equation}\label{D-null}
\left\{
\begin{aligned}
D_1(\partial^2 v, \partial U, \partial V)= & \ \mathbf{N}(\partial v, U) \partial V
+ \mathbf{N}(\partial v, V) \partial U
\\
D_2(\partial^2 u, \partial U, \partial V)= & \ \mathbf{N}(\partial u, U) \partial V
+ \mathbf{N}(\partial u, V) \partial U
\end{aligned}
\right.
\end{equation}
where $\mathbf{N}(\cdot, \cdot)$ denote null forms, i.e. linear combinations of the null forms in \eqref{null forms-re}. The relation for $D_2$ is important for us, while the one for $D_1$ is less critical because of the better $t^{-1}$ decay enjoyed by the Klein-Gordon component $v$.

We also note that $B_j$ and $C_j$ do not have a null structure in general, 
as it can be seen by examining \eqref{corrections0}. However, they are matched and we will take advantage of this later on.

Summing up all these terms we obtain the following energy flux relation for
 solution to inhomogeneous linearized problem:
\begin{equation}
\label{dteq}
\partial_te^{quasi}(t,x)=\sum_{j=1}^2\partial_j f_j+g,
\end{equation}
where the fluxes $f_j$ have the expressions 
\begin{equation}
\label{fj}
f_j=U_jU_t+ V_jV_t+ C_1(\partial v, \partial U, \partial V) +  C_2(\partial u, \partial U, \partial V),
\end{equation}
and the source $g$ is a trilinear form with null structure and has components
as follows:
\begin{equation}
\label{g}
\begin{split}
g= & D_1(\partial^2 v, \partial U, \partial V)+ D_2(\partial^2 u, \partial U, \partial V).
\end{split}
\end{equation}

To complete the proof of the energy estimates we will have to bound
the source terms using the energy. The $v$-terms $D_1(\partial^2 v, \partial U, \partial V)$ will be well-behaved
because $\partial^2 v$ has $t^{-1}$ decay, but the $u$-terms $D_2(\partial^2 u, \partial U, \partial V)$ do not share this
property. Instead, for these terms we need a different idea which
takes advantage of their null structure.

This leads us to Alinhac's approach which establishes an improved
version of the ``standard'' energy inequality (by this we mean the
inequality corresponding to the multiplier $\partial_t$ case); such an
inequality yields, besides the usual fixed time energy bound,
a bound of the (weighted) $L^2$ norm in both variables $x$ and $t$ of
some \emph{good derivatives} of $(U, V)$ in special regions.

The special regions mentioned above are exactly the sets $C_{TS}^\pm$  
introduced earlier in \eqref{ct}, \eqref{cts}, \eqref{ctout}, and \eqref{ct1},
which provide a double dyadic decomposition of the space-time relative 
to the size of $t$ and the size of $t-r$ which  measures how far or close we are to
the cone.  To simplify the exposition we will use the notation $C_{TS}$ as a shorthand for either $C^+_{TS}$ or $C^-_{TS}$. 

 The good derivatives alluded to above  are exactly  the tangential
derivatives relative to the cones $\{ t - r = const\}$,
or equivalently\footnote{from the perspective of the estimates}, relative to the hyperboloids $\{t^2- x^2 = const\}$. These surfaces can be viewed as providing 
nearly equivalent foliations of the sets $C_{TS}^\pm$.

\begin{lemma}\label{l:alih} Assume the solutions $(u,v)$ to the main equations
  \eqref{WKG} or \eqref{WKG-cut} satisfy the bounds \eqref{boot1.1}-\eqref{boot3} over the space-time regions $C_{T}$. Then the solution
  $(U,V)$ of the linearized equation \eqref{general linearization}
satisfies
\begin{multline}\label{cts-bd}
\sup_{1\le S\lesssim T}\int_{C_{TS}}   \frac{1}{S}  \left\{  \left(  V_j +\frac{x_j}{r}V_t\right)^2 +   \left(  U_j +\frac{x_j}{r}U_t\right)^2 +  V^2 \right\}\, dx dt 
\lesssim \sup_{t\in [T, 2T]}E^{quasi}(U, V)(t).
\end{multline} 
\end{lemma}

Here we only consider the $C_{TS}$ regions with $S \lesssim T$, as the outer region
$C_T^{out}$ is uninteresting from this perspective.
Concerning the directional derivatives in the lemma, we note that they have a special structure:

\begin{remark}The  quantities appearing in \eqref{cts-bd}, i.e. $ V_j +\frac{x_j}{r}V_t$ and $ U_j +\frac{x_j}{r}U_t$, represent the  derivatives of $V$, respectively $U$, in the tangential directions to the cones $C=\left\{t-r=const\right\}$. We denote the them by
\[
\Tau_j= \partial_j +\frac{x_j}{r}\partial_t.
\]
\end{remark}
We further remark that we have the trivial bound
\begin{equation}
\label{close}
\sup_{1\le S \lesssim T}\int_{C_{TS}}   \frac{1}{T} \left( |\nabla U|^2 +  |\nabla V|^2  \right)dx dt 
\lesssim \sup_{t\in [T, 2T]}E^{quasi}(U, V)(t),
\end{equation}
which can be viewed as the natural complement of \eqref{cts-bd} for nontangential derivatives. In other words, \eqref{cts-bd} and \eqref{close} should be viewed 
as a pair. This last bound also shows that  \eqref{cts-bd} becomes trivial if 
$S \approx T$. Thus, in the proof we will be concerned with the case $1 \leq S \ll T$.

Closely related to the last comment, an important observation that applies in the $C_{TS}$ regions is that we can connect the vector fields $\Tau$, tangent to the cones, to the corresponding vector fields $Z$, tangent to the hyperboloids that foliate both the interior or the exterior of the cone:
\[
H_\rho = \{ t^2 - x^2 = \pm \rho^2 \}.
\]
Here we consider a hyperboloid  which intersects $C_{TS}^+$ provided that $\rho^2 \approx TS$. Since $S \ge 1$, this in particular requires that 
\[
T \lesssim \rho^2 \lesssim T^2.
\]
In this setting we note that vector fields $Z$ and $\Tau$ are related in general via
\begin{equation}
\label{ZT rel gen}
Z = t\Tau - \frac{x}{r}(t-r)\partial_t
\end{equation}
and, in particular in the $C_{TS}$ regions, by
\begin{equation}
\label{ZT rel CTS}
Z \approx T\Tau - S\partial_t.
\end{equation}

\
As the tangent planes to both the hyperboloids and cones are close to each other, via the estimate given in \eqref{close}, we give an alternative statement of Lemma \ref{l:alih} in terms of the $Z$ vector fields 
\begin{lemma}
\label{l:equiv lemma}
Under the same assumptions as in the Lemma \ref{l:alih}, we have
\begin{multline}\label{cts-bd-re}
\sup_{1 \leq S \lesssim T}\int_{C_{TS}}   \frac{1}{S}  \left\{  T^{-2} (|ZU|^2 + |ZU|^2) +  V^2 \right\}
+ \frac{1}{T}  (|\nabla U|^2 + |\nabla V|^2) 
\, dx dt
\lesssim \sup_{t\in [T, 2T]}E^{quasi}(U, V)(t).
\end{multline} 

\end{lemma}

\begin{proof}[Proof of Lemma \ref{l:alih}] 
We consider the following weighted version of the  energy $E (U,V)$
\begin{equation}
\label{weighted energy}
E_a\left(U,V\right) := \int_{\mathbb{R}^2}e^a e_0(t,x)\, dx,
\end{equation}
and similarly the weighted version of the quasilinear energy $E^{quasi}(U, V)$
\begin{equation}
\label{Alinach-quasi}
\begin{aligned}
E_a^{quasi}(U,V)&:=  \int_{\mathbb{R}^2}e^a e^{quasi}(t,x)\, dx . \\
\end{aligned}
\end{equation}

Here $e^a$ is a \emph{ghost weight}, which will be chosen such that
$a$ is bounded and the weight $e^a$ ultimately disappears from the inequalities.
Precisely, we will choose $a$ of the form
\begin{equation}\label{choose-a}
a (t,r): =-A (t-r),
\end{equation}
where $A$ is a bounded nondecreasing function. Then the  gain in the estimates will come from the contribution of $A'(t-r)$, which will be chosen to be positive. 

We can further specialize the choice of  the function $A(t-r)$ and 
separately adapt it to each dyadic space-time regions $C_{TS}$ for $1 \leq S \ll T$. Precisely, 
we can chose it so that 
\begin{equation}\label{A'}
A'(t-r)\approx \frac{1}{S}, \qquad \mbox{ for } \vert t- r \vert \approx S, \quad \mbox{ and  $A'(t-r) = 0$  elsewhere. }
\end{equation}

For such functions $A$ we  need to understand the time derivative
of $E_a^{quasi} (U, V)$.  Thus, using the equation \eqref{dteq} we have
\begin{equation*}
\begin{aligned}
\frac{d}{dt}E_a^{quasi}(U, V)&=  \int_{\mathbb{R}^2} \frac{d}{dt} [ e^a e^{quasi} (t,x)]\, dx  \\
& =  \int_{\mathbb{R}^2}  e^a a_t \, e^{quasi}  +e^a (\partial_j f_j + g)\, dx . \\
& =  \int_{\mathbb{R}^2}  e^a (a_t \, e^{quasi}  - a_j f_j) \, dx
+   \int_{\mathbb{R}^2} e^a g\, dx .
\end{aligned}
\end{equation*}
The second integrand involves the function $g$, which is a trilinear form with a null structure. The terms in the first integrand do not separately have a null structure,
so we will take a closer look at them together for our choice of $a$ as above.
Separating the quadratic and the cubic contributions, we write
\begin{equation*}
\begin{aligned}
a_t e^{quasi}  - a_j f_j & = -A'(t-r) (e^{quasi} + \frac{x_j}r f_j)
\\
& = -A'(t-r)
(Q_2(\partial U,\partial V) + Q_{3,1}(\partial v,\partial U,\partial V) + Q_{3,2}(\partial u,\partial U,\partial V)) ,
\end{aligned}
\end{equation*}
where $Q_2$ represents the quadratic term,
\[
Q_2(U, V) = e_0(\partial U,\partial V) + \frac{x_j}r  (U_t U_j+V_j V_t),
\]
and $Q_{3,1},Q_{3,2}$ represent the cubic terms,
\[
\aligned
Q_{3,i}(\partial v,\partial U,\partial V) = B_i(\partial v,\partial U,\partial V) +  \frac{x}r C_i(\partial w,\partial U,\partial V).
\endaligned
\]
Recombining the terms in $Q_2$ one obtains
\[
Q_2(U,V) = \left(  V_j +\frac{x_j}{r}V_t\right)^2 +   \left(  U_j +\frac{x_j}{r}U_t\right)^2 + V^2 ,
\]
which is exactly as in \eqref{cts-bd}. On the other hand, a short algebraic computation reveals the following structure for $Q_{3,i}$:
\[
Q_{3,i}(\partial w,\partial U,\partial V) = D_{1,i}(\Tau w,\partial U,\partial V) + D_{2,i}(\partial w,\partial U,\partial V),
\]
where 
\begin{itemize}
    \item $D_{1,i}$ has no null structure but only uses a tangential derivative 
    of $w$ ,
    \item $D_{2,j}$ has a null structure, i.e. can be represented as
    \[
     D_{2,j}(\partial w,\partial U,\partial V) =\mathbf{N}( w, U)\partial V + \mathbf{N}(w, V)\partial U.
    \]
\end{itemize}
Thus we  obtain
\begin{equation}\label{main-ee}
\begin{aligned}
\frac{d}{dt}E_a^{quasi}(U, V) & + \int_{\mathbb{R}^2} e^a A'(t-r)  Q_2(U,V)\, dx
\\
 =&  -  \int_{\mathbb{R}^2} e^a A'(t-r)( 
 D_{1,1}(\Tau v,U,V) + D_{1,2}(\Tau u,U,V))  \, dx
\\
& - \int_{\mathbb{R}^2} e^a A'(t-r)( 
 D_{2,1}(\partial v,\partial U,\partial V) + D_{2,2}(\partial u,\partial U,\partial V))  \, dx
\\
& +   \int_{\mathbb{R}^2} e^a ( 
 D_1(\partial^2 v, \partial U, \partial V)+ D_2(\partial^2 u, \partial U, \partial V)) \, dx .
\end{aligned}
\end{equation}
Now we integrate this relation between $T$ and $2T$. With our choice 
for $A$, the first integral in the left hand side  controls the expression on the left in Lemma~\ref{l:alih}. It remains to estimate the remaining terms on the right 
perturbatively.

\medskip

\emph{1. The contributions of $D_{1,j}$.} Here we use our bootstrap assumption 
to estimate 
\[
| \Tau u |+|\Tau v| \lesssim \epsilon T^{-1}S^{\delta},
\]
which implies that 
\[
 |D_{1,1}(\Tau v,\partial U,\partial V)| +|D_{1,2}(\Tau u,\partial U,\partial V)|
 \lesssim \  \epsilon T^{-1}S^{\delta} |\partial U| |\partial V|.
\]
Since $\vert A'\vert \lesssim S^{-1}$, this suffices in order to bound their contribution by the energy.

\medskip 
 \emph{2. The $v$-terms in $D_{2,1}$ and $D_1$. } Their contribution is 
\begin{equation}
\label{v-terms}
\int_{T}^{2T} \int_{\mathbb{R}^2} e^a [A'(t-r)D_{2,1}(\partial v,\partial U,\partial V)  + D_{1}(\partial^2 v,\partial U,\partial V)]\,  dxdt 
\end{equation}
which are all bounded using \eqref{boot3} for the $v$-factors and the energy $E_a^{quasi}$ for the $U$ and $V$ terms by 
\begin{equation}\label{le_Equasi}
\lesssim \epsilon \sup_{t\in [T, 2T]}E^{quasi}(U, V)(t).
\end{equation}

\medskip 

 \emph {3. The $u$-terms in $D_{2,2}$ and $D_2$.} These have the form
\begin{equation}
\label{u-terms}
\int_{T}^{2T}\int_{\mathbb{R}^2} e^a [A'(t-r)D_{2,2}(\partial u,\partial U,\partial V)  + D_{2}(\partial^2 u,\partial U,\partial V)] \, dx dt,
\end{equation}
which we need to process further. In the region $C^{out}_T$ the pointwise bounds \eqref{boot2.1} and \eqref{boot2} give a $t^{-1}$ decay for both $\partial u$ and $\partial^2 u$, so this is identical to the case of the $v$ terms above.
It remains to consider the contribution  over $C_T^{in}:= C_T\setminus C^{out}_T$,
where we will exploit the null structure of $D_{2,2}$ and $D_2$, see \eqref{D-null}.

The key property is that all null forms can be expressed in the form
\begin{equation}
\label{rewrite2}
\mathbf{N}(\phi,\psi) =   \partial \phi \cdot \mathcal{T} \psi + \mathcal{T}\phi \cdot \partial \psi,
\end{equation}
or equivalently 
\begin{equation}
\label{rewrite1}
\mathbf{N}(\phi,\psi) = \frac{1}{t} \left( \partial \phi \cdot Z \psi + Z \phi \cdot \partial \psi   + 
(t-r) \partial \phi \cdot \partial \psi\right).
\end{equation}
 
We begin with $D_2$, which contains terms of the form $ \mathbf N(\partial u, V) \partial U$ and $\mathbf{N}(\partial u, U)\partial V$. We consider the first term, as the second will be similar. By \eqref{rewrite2} we have
\begin{equation}\label{rewrite}
\mathbf N(\partial u, V) =   \partial^2 u \cdot \mathcal{T} V + \mathcal{T}\partial u \cdot \partial V .
\end{equation}
For the last term we can directly use our bootstrap assumptions in
\eqref{boot1} and \eqref{boot2} to obtain the pointwise bounds
\[
|\mathcal{T}\partial u | + \left| \frac{t-r}{t}\partial^2 u \right| \lesssim \epsilon T^{-1}S^{-\delta_1}.
\]
 Hence the contributions of those  terms are estimated as in Case 1 by \eqref{le_Equasi}.

The contribution of the first term in \eqref{rewrite} to the integral in \eqref{u-terms} is more delicate because now the $\Tau$ vector field applies to $V$. So instead we split the integral over $C^{in}_T$ into the sum of integrals over the $C_{TS}$ space-time regions, apply the Cauchy-Schwartz inequality in space-time, and H\"older's inequality in time in each such region, as well as \eqref{boot2}, to bound each of these integrals by
\[
\begin{aligned}
\left| \int_{C_{TS}} \partial U \, \partial^2 u \, \mathcal{T} V\, dx dt \right| &\lesssim \Vert \partial^2 u\Vert_{L^{\infty}_{C_{TS}}}  \Vert \partial U \Vert_{L^2_{C_{TS}}} \left\| \Tau V \right\|_{L^2_{C_{TS}}}   \\
&\lesssim\epsilon T^{-\frac{1}{2}} S^{-\frac{1}{2}-\delta_1} T^{\frac{1}{2}} \left( \sup_{t\in [T, 2T]}E^{quasi}(U, V)(t) \right)^{\frac{1}{2}} S^{\frac{1}{2}}  \left\| S^{-\frac{1}{2}} \Tau V \right\|_{L^2_{C_{TS}}}   \\
&\lesssim \epsilon  S^{-\delta_1} \left( \sup_{t\in [T, 2T]}E^{quasi}(U, V)(t) \right)^{\frac{1}{2}} \left\| S^{-\frac{1}{2}} \Tau V\right\|_{L^2_{C_{TS}}},
\end{aligned}
\]
and after a straightforward $S$ summation over $1 \leq S \leq T$, we get 
\[
\begin{aligned}
\left| \int_{ C_T^{in}}  e^a D_{2}(\partial^2 u,\partial U,\partial V)\, dx dt \right| 
\lesssim & \ \epsilon  \left( \sup_{t\in [T, 2T]}E^{quasi}(U, V)(t) \right)^{\frac{1}{2}} \sup_{1 \leq S \leq T} \left\| S^{-\frac{1}{2}} \Tau (U,V)\right\|_{L^2_{C_{TS}}}\\
& \ + \epsilon   \sup_{t\in [T, 2T]}E^{quasi}(U, V)(t)  .\\
\end{aligned}
\]

The bound for the contribution of $D_{2,2}$ is similar, with the difference that 
the integrand is now localized to a fixed dyadic region $C_{TS}$ and that we are using the bootstrap assumption \eqref{boot2.1} for $\partial u$ instead of \eqref{boot2} for $\partial^2 u$. 
Using also \eqref{A'} we obtain
\[
\begin{aligned}
\left| \int_{ C_T^{in}}\!\!\!  e^a A'(t-r)D_{2,2}(\partial u,\partial U,\partial V)  \, dx dt \right| \lesssim &  \ \epsilon S^{-1+\delta}  \left( \!\sup_{t\in [T, 2T]}E^{quasi}(U, V)(t) \!\right)^{\frac{1}{2}} \!\! \sup_{1 \leq S \leq T}\! \left\| S^{-\frac{1}{2}} \Tau (U,V)\right\|_{L^2_{C_{TS}}}\\
& \ + \epsilon S^{-1+\delta}  \sup_{t\in [T, 2T]}E^{quasi}(U, V)(t),  \!\!
\end{aligned}
\]
where the $S^{-1+\delta}$ gain insures the summation in $S$, but it is not otherwise needed in the sequel. 

Overall we have proved that
\[
\eqref{u-terms} \lesssim \epsilon \sup_{t\in [T, 2T]}E^{quasi}(U, V)(t) +\epsilon   \sup_{1\le S\lesssim T} \left\| S^{-\frac{1}{2}}  \Tau(U,V)\right\|^2_{L^2_{C_{TS}}}.
\]
\medskip

Summing all up all the contributions to the integrated form of \eqref{main-ee},
we obtain 
\[
\int_{C_{TS}} S^{-1}Q_2(U,V) \, dx dt \lesssim \epsilon\sup_{t\in [T, 2T]}E^{quasi}(U, V)(t) +\epsilon   \sup_{1\le S\lesssim T} \left\| S^{-\frac{1}{2}}  \Tau(U,V)\right\|^2_{L^2_{C_{TS}}}.
\]
Finally we take the supremum over $1 \leq S \leq T$. Then 
the last term on the right can be absorbed on the left, which concludes the proof of the lemma. 

\end{proof}

Now we conclude the proof of the Proposition~\ref{p:linear}. For this we repeat the
computation above with $a(t,r)=0$. We integrate the relation
\eqref{main-ee} from $T$ up to an arbitrary $t\in [T, 2T]$, and estimate
the RHS exactly as in the proof of the Lemma~\ref{l:equiv lemma}. We obtain
\begin{equation}
\label{this}
E^{quasi}_a(t) -E^{quasi}_a(T) \lesssim \epsilon \sup_{t_0\in [T, 2T]} E^{quasi}_a(t_0),
\end{equation}
and taking the supremum over $t\in [T, 2T]$ gives
\[
\sup_{t\in [T,2T]}E^{quasi}_a(t)\lesssim E^{quasi}_a(T),
\]
which concludes the proof of the proposition.
\end{proof}

A consequence of the proof of Proposition~\ref{p:linear} is that, in addition to the uniform energy bound in $[T,2T]$, 
we also gain uniform control of the localized energies in the left hand side of \eqref{cts-bd}.
It will be useful in effect to obtain a slight improvement over \eqref{cts-bd}, where we think of $C_{TS}^+$ 
as foliated by hyperboloids and obtain uniform $L^2$ bounds over each such hyperboloid.

To set the notations, consider a hyperboloid 
\[
H_\rho = \{ t^2 - x^2 = \rho^2 \},
\]
which intersects $C_{TS}^+$ provided that $\rho^2 \approx TS$. Since $S \ge 1$, this in particular requires that 
\begin{equation}
\label{delta}
T \lesssim \rho^2 \lesssim T^2.
\end{equation}
Then we have the following:

\begin{lemma}
Under the same assumptions as Proposition~\ref{p:linear}, the solution $(U,V)$ to \eqref{general linearization}
satisfies
\begin{equation}\label{hrho-bd}
\sup_{T \lesssim \rho^2 \lesssim T^2} \int_{H_\rho \cap C_T}  T^{-2} (|Z U|^2 + |ZV|^2  + \rho^2 (|\nabla U|^2 +|\nabla V|^2)) 
  + V^2\ dx  \lesssim E(U,V)(T).
\end{equation}
\end{lemma}

\begin{proof}
With small differences the  proof mimics the proof of Proposition~\ref{p:linear}.
We consider the domain 
\[
D = \{ (x,t) \in C_T; \ t^2 - x^2 \leq \rho^2 \},
\]

\begin{figure}[h]
\begin{center}
\begin{tikzpicture}[scale=1.9]

\draw[->] (0,-0.3) -- (0,2.5);
\draw[->] (-2.5,0) -- (2.5,0);
\node[below] at (2.4,0) {\small $x$};
\node[left] at (0,2.4) {\small $t$};

\draw[red, thick]  (-2.5,0.8) -- (2.5,0.8);
\node[below, left] at (0,0.7) {\tiny $T$};
\draw[red,thick]  (-2.5,1.6) -- (-1.058,1.6);
\node[above, left] at (0,1.7) {\tiny $2T$};
\draw[red,thick]  (1.058,1.6) -- (2.5,1.6);
\draw[dashed] (-1.03,1.6) -- (1.058,1.6);

\draw[red,thick] [domain=-1.058:1.058] plot(\x, {((1.2)^2+(\x)^2)^(0.5)});
\draw[dashed] [domain=-2:-1.058] plot(\x, {((1.2)^2+(\x)^2)^(0.5)});
\draw[dashed] [domain=1.058:2] plot(\x, {((1.2)^2+(\x)^2)^(0.5)});
\node[below, left] at (1.8,2.2) {\tiny $H_\rho$};
\node[red] at (-1.5,1.2) {\tiny $D$};

\draw[dashed] [domain=0:2.4] plot(\x,\x);
\draw[dashed] [domain=-2.4:0] plot(\x,-\x);
\node[right] at (2.2,2.2) {\tiny $t=|x|$};

\fill [red!40,nearly transparent, domain=-1.058:1.058, variable=\x]
  (-1.058, 0.8)
  -- plot ({\x},{((1.2)^2+(\x)^2)^(0.5)})
  -- (1.058, 0.8)
  -- cycle;

\fill[red!40,nearly transparent] (-2.5, 0.8) -- (-2.5,1.6) -- (-1.058,1.6) -- (-1.058, 0.8) -- cycle;
\fill[red!40,nearly transparent]  (1.058,1.6) -- (1.058, 0.8) -- (2.5, 0.8) -- (2.5, 1.6) -- cycle;

\end{tikzpicture}
\caption{Region $D$ in 1+1 space-time dimension}
\label{f:D}
\end{center}
\end{figure}
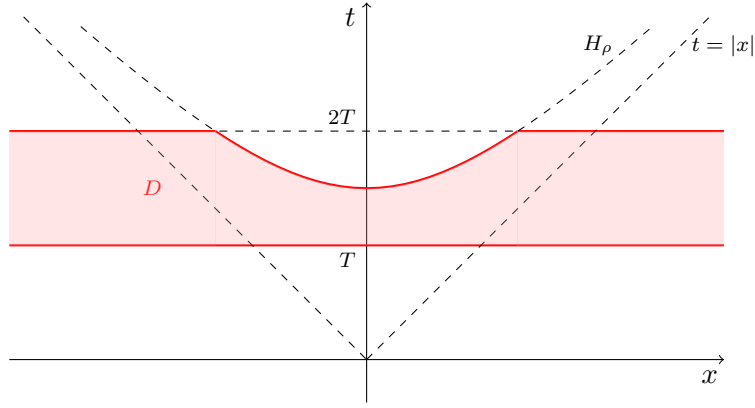

\noindent which represents the portion of $C_T$ below the hyperboloid $H_\rho$ (see figure \ref{f:D} which depicts the case when the hyperboloid intersects the surface $t=2T$, but not the $t=T$). Then we integrate the 
relation \eqref{dteq} over $D$, estimating the contribution of $g$ exactly as in the proof of  Proposition~\ref{p:linear}.
The contributions on the bottom $t = T$ and the top $t = 2T$ are simply the energies.
This yields the bound
\begin{equation}
\label{hiperboloid}
\int_{H_\rho \cap C_T} e^{quasi}(t,x) - \frac{x_j}{t} f_j\,  dx \lesssim E(U,V)(T).
\end{equation}
Here we have used the normal vector $n = (1,-\frac{x}{t})$ on $H_{\rho}$.
It remains to verify that the integrand is positive definite and controls the integrand in the left hand side 
of \eqref{hrho-bd}.

The leading contribution comes from $e_0$ and the quadratic part $f_{j0}$ of the $f_j$ and gives exactly the correct expressions. 
It remains to perturbatively estimate the cubic contributions to the integrand in \eqref{hiperboloid}, which under the assumption \eqref{delta}, has size
\[
\lesssim (|\nabla v|+|\nabla u|) |\nabla U| |\nabla V| \lesssim \frac{\epsilon}{\sqrt{ST}} (|\nabla U|^2+ |\nabla V|^2) \ll 
\epsilon\frac{\rho^2}{T^2} (|\nabla U|^2+ |\nabla V|^2) 
\]
as needed.
\end{proof}

To streamline the notations, it will help to introduce the following norm for functions $(U,V)$ in $C_T$,
\begin{equation}\label{XT}
\| (U,V)\|_{X^T}^2 := \sup_{t \in [T,2T]} E(U,V)(t) + LHS \mbox{\eqref{cts-bd}} + LHS  \mbox{\eqref{hrho-bd}}.
\end{equation}
Then we have the following 
\begin{corollary}\label{c:linear}
Under the same assumptions as Proposition~\ref{p:linear}, the solution $(U,V)$ to \eqref{general linearization}
satisfies
\begin{equation}\label{all-bd}
\|(U,V)\|_{X^T}^2  \lesssim E(U,V)(T).
\end{equation}
\end{corollary}

Another  direct consequence of Proposition ~\eqref{p:linear} and of above corollary  is the following
corollary which derives similar energy estimates but for the
non-homogeneous analogue of \eqref{general linearization}:
\begin{equation}
\label{nonhom general linearization}
\left\{
\begin{aligned}
&(\partial_t^2-\Delta_{x})U(t,x)= \mathbf{N_1}(v, \partial V) +\mathbf{N_1}(V,\partial v)  + \mathbf{N_2}(u, \partial V) +\mathbf{N_2}(U,\partial v) + \mathbf{F}(t,x) \\
&(\partial_t^2-\Delta_{x}+1)V(t,x)=\mathbf{N_1}(v, \partial U)+\mathbf{N_1}(V, \partial u) + \mathbf{N_2}(u, \partial U)+\mathbf{N_2}(U, \partial u) +\mathbf{G}(t,x) .\\
\end{aligned}
\right.
\end{equation}
where $\mathbf{F}$ and $\mathbf{G}$ are arbitrary functions of $t$ and $x$.

\begin{corollary} \label{c:nonhlinear} Assume the solutions to the
  main equations \eqref{WKG} satisfy the bounds \eqref{boot1.1}- \eqref{boot3} in some time interval $[0, T]$. Then
  the non-homogeneous linearized equation \eqref{nonhom general linearization} is well-posed in $[0, T]$ and the solution
  satisfies
\begin{equation}
\label{energy est inho}
\sup_{t\in [T,2T]}E^{quasi}(U,V)(t) \leq (1+\epsilon \tilde C)E^{quasi}(U,V)(T)
+ \Vert (\mathbf{F}, \mathbf{G})\Vert_{L^1_tL^2_x}^2 ,
\end{equation}
where $\tilde C \approx C$  with $C$ as in \eqref{boot1.1}-\eqref{boot3}. In addition, we have
\begin{equation}\label{all-bd-inhom}
\|(U,V)\|_{X^T}^2  \lesssim E(U,V)(T)+ \Vert (\mathbf{F}, \mathbf{G})\Vert_{L^1_tL^2_x}^2 .
\end{equation}
\end{corollary}
\begin{proof}
The proof of the corollary is a direct consequence of Proposition~\ref{p:linear} and of the variation of parameters principle (i.e. Duhamel's principle).
\end{proof}

\section{Higher order energy estimates}
\label{s:higher}
The main goal of this section is to establish energy bounds for $(u,
v)$ and their higher derivatives.  We will compare this
system with the linearized system which was studied in
Section~\ref{s:linearization} and use a large portion of the
estimates already obtained for the non-homogeneous linearized system
\eqref{nonhom general linearization}, as in Corollary~\ref{c:nonhlinear}.

We start with the equations \eqref{WKG} and differentiate them $n$
times. Here the variables that play the role of the linearized
variables $(U,V)$ are the $n$ times differentiated variables $(\partial^n u,\partial^n v)$,
which we will denote by $(u^n, v^n)$. We differentiate \eqref{WKG} $n$
times and separate the terms into leading order and lower order
contributions, interpreting the differentiated equation as a linearized equation with 
a source term, as in \eqref{nonhom general linearization}:
\begin{equation} \label{NWKG}
\left\{
\begin{aligned}
&(\partial^2_t  - \Delta_x) u^n(t,x) =  \mathbf{N_1}(v, \partial v^n)+ \mathbf{N_1}(v^n, \partial v)  + \mathbf{N_2}(u, \partial v^n)+ \mathbf{N_2}(u^n, \partial v)  + \mathbf{F^n}
\\
&(\partial^2_t  - \Delta_x  + 1)v^n(t,x) =   \mathbf{N_1}(v, \partial u^n) + \mathbf{N_1}(v^n, \partial u) + \mathbf{N_2}(u, \partial u^n) + \mathbf{N_2}(u^n, \partial u) + \mathbf{G^n},
\\\end{aligned} 
\right.
\end{equation}
where the source terms have the form
\[
\mathbf{F^n}(t,x) =\sum_{k=1}^{n-1}\mathbf{N}_1(v^k,\partial v^{n-k})
+\mathbf{N}_2(u^k, \partial v^{n-k}), \ \mathbf{G^n}(t,x) = \sum_{k=1}^{n-1}\mathbf{N}_1(v^k, \partial u^{n-k})+
\mathbf{N}_2(u^k, \partial u^{n-k}).
\]

Our energy estimates for the differentiated system will be proved under the same bootstrap
assumptions we previously imposed \eqref{boot1.1}-\eqref{boot3}.
The main result of this section is as follows:
\begin{proposition} \label{p:high-energy}
Let $n \geq 4$. Assume the solutions $(u,v)$ to the original main
  equations \eqref{WKG} or \eqref{WKG-cut} are defined in $\H^n$  in some time interval $[0, T]$, 
and satisfy the bootstrap  bounds  \eqref{boot1.1}- \eqref{boot3}.  Then the following bound holds
\begin{equation}
\label{d1}
E^{n} (u, v)(t)\lesssim t^{\tilde{C}\epsilon}E^{n}(u,v)(0), \quad t\in [0,T],
\end{equation} 
for some positive constant  $\tilde{C}$.
\end{proposition}
\begin{remark} \label{r:remark}As defined the energies $E^n(u,v)(t)$ measure not only higher order spatial derivatives for $(u,v)$ but also higher order time derivatives of $(u,v)$. On the other hand at the initial time we want to measure only the size of the Cauchy data, which means at most one time derivative of $(u,v)$. However our a priori bounds (bootstrap assumptions \eqref{boot1.1}-\eqref{boot3}) suffice in order to estimate the Cauchy data to higher order time derivatives of $(u,v)$ in terms of the corresponding bounds for the Cauchy data of $(u,v)$. This is a relatively straightforward  exercise which is left for the reader.
\end{remark}
\begin{proof} The proof heavily uses the energy estimates for the
  linearized system \eqref{general linearization} in the previous
  section. We will use the differentiated equations \eqref{NWKG}
to inductively prove bounds on the differentiated functions $(u^n,v^n)$.
For this we will rely on the energy $E^{quasi}$ and on the bounds in 
Corollary~\ref{c:nonhlinear}.

We begin by discussing the case when $n = 0$. The easy way to handle this case is to relate it to the linearized equation, by simply thinking at \eqref{WKG} as being  written as in Corollary \ref{c:nonhlinear}  where $\mathbf{F}$ and $\mathbf{G}$ are given by
\[
\mathbf{F}:= -\mathbf{N_1}(V,\partial v)-\mathbf{N_2}(U, \partial v) , 
\qquad \mathbf{G}:= -\mathbf{N_1}(V,\partial u)-\mathbf{N_2}(U, \partial u).
\]
The non-homogeneous term $\mathbf{F}$ can be easily bound in $L^1_tL^2_{x}(C_{TS})$ using a priori estimate \eqref{boot3}, while $\mathbf{G}$ is bounded in $L^1_tL^2_{x}(C_{TS})$ using the null structure highlighted in \eqref{rewrite2}, \eqref{rewrite1} as in the proof of Lemma \ref{l:equiv lemma}.

The case $n = 1$ is  a trivial consequence as $(\nabla u, \nabla v)$
exactly solve the linearized equation. So from here on we will assume that $n \geq 2$.

 Since we do not have a way to prove a good energy estimate working only at  fixed
  time, we will focus on proving a good energy estimate on dyadic time
  scales $[T,2T]$. Precisely, arguing by induction on $n$, it suffices to show that
  \begin{equation}
\label{quasin} 
\sup_{t\in [T,2T]} E^{quasi}(u^n,v^n)(t)\leq  E^{quasi}(u^n,v^n)(T)
+ \epsilon C  E^{quasi}(u^{\leq n},v^{\leq n})(T) .
\end{equation}
Just as in Proposition~\ref{p:linear}, to prove this we will use the stronger auxiliary norm $X^T$ in \eqref{XT},
which captures more of the structure of linearized waves. So instead of \eqref{quasin},
we will prove the pair of  bounds
\begin{equation}
\label{quasina} 
\sup_{t\in [T,2T]} E^{quasi}(u^n,v^n)(t)\leq  E^{quasi}(u^n,v^n)(T)
+ \epsilon C \|(u^{\leq n},v^{\leq n})\|_{X^{T}}^2 .
\end{equation}
\begin{equation}
\label{quasinb} 
\|(u^n,v^n)\|_{X^T}^2 \lesssim   E^{quasi}(u^n,v^n)(T)
+ \epsilon C  \|(u^{\leq n},v^{\leq n})\|_{X^{T}}^2 .
\end{equation}

To prove both of these bounds we rely on the results of the Corollary \ref{c:nonhlinear}, which shows that it suffices 
to obtain $L^1_tL^2_x$ bounds for the non-homogeneous contributions  $(\mathbf{F^n}, \mathbf{G^n})$.
Precisely, we will show the bound
\begin{equation}\label{quasin-source}
\| (\mathbf{F^n}, \mathbf{G^n})\|_{L^1 L^2_x}^2 \lesssim C \epsilon  \|(u^{\leq n},v^{\leq n})\|_{X^{T}}^2 ,
\end{equation}
which by Corollary \ref{c:nonhlinear} allows us to recover the estimates \eqref{quasina}, \eqref{quasinb}.

Analogous to the discussion of the terms in \eqref{v-terms}, \eqref{u-terms}, we need to deal with two types of terms:

\medskip

\noindent  \emph{1. The $v$-terms} have the form
\begin{equation}
\label{vn-terms}
\begin{aligned}
\mathbf{N}_1(v^k, \partial v^{n-k})\, , \quad k=\overline{1, n-1},
\end{aligned} 
\end{equation}
which we want to bound in $L^1_tL^2_x$. Here we do not need to use the null structure.
We discuss two possible terms: 

\begin{itemize}
\item[a)] The case when $k=1$. In this case we use the bound \eqref{boot3} for the $\partial^2 v$ term. We also  control  $\Vert \partial v^{n}\Vert_{L^2_x}$ from the $v^n$ energy.  Thus, we get fixed time  bound 
 \[
\Vert \mathbf{N}_1(\partial v, \partial v^{n-1}) \Vert_{L^2_x} \lesssim \epsilon T^{-1}\Vert \nabla v^n\Vert_{L^2_x}.
\]
We obtain the $L^1_tL^2_x$ bound by integrating over the time interval  $[T,2T]$.
\medskip

\item[ b)] The case when $k\geq 2$ and $n-k\geq 1$.  It suffices to estimate at fixed time
\[
\Vert \mathbf{N}_1(v^k, \partial v^{n-k})\Vert_{L^2_x},
\]
which is done by placing $\partial v^k$ in $L^{\frac{2(n-1)}{k-1}}$ and
$\partial^2 v^{n-k}$ in $L^{\frac{2(n-1)}{n-k}}$.  If all derivatives were spatial derivatives then we can bound these terms at fixed time using interpolation by the energy of $v^n$ and the bounds in
\eqref{boot3}
\[
\begin{aligned}
\Vert \mathbf{N}_1(v^k, \partial v^{n-k})\Vert_{L^2_x}&\lesssim  \Vert  \partial v^k \Vert_{L_x^{\frac{2(n-1)}{k-1}}}\Vert \partial^2 v^{n-k}\Vert_{L_x^{\frac{2(n-1)}{n-k}}}\\
&\lesssim \Vert \partial v^n\Vert_{L^2_x} \Vert\partial^2  v\Vert_{L^{\infty}_x}\\
& \lesssim C\epsilon T^{-1}\Vert \partial v^n\Vert_{L^2_x}.
\end{aligned}
\]
If some of these derivatives are time derivatives then the same argument applies with the one difference that on the right we use uniform norms over the interval $[T,2T]$. Integrating in time over the $[T,2T]$ time interval leads to the $L^1_tL^2_x$ bound. 

\end{itemize}

\medskip
\noindent \emph{2. The $u$-terms} are as follows: 
\begin{equation}
\label{un-terms}
\mathbf{N}_2(u^k, \partial v^{n-k}),\quad  \mathbf{N}_1(v^k, \partial u^{n-k}), \quad \mathbf{N}_2(u^k, \partial u^{n-k})  ,  \quad k=\overline{1, n-1},
\end{equation}
and also need to be bounded $L^1_tL^2_x(C_{T})$. The second term is treated as in Case 1 for $k=1$ so we will only consider it for $k=\overline{2,n-1}$. The analysis of the first two terms is almost identical, so we just discuss the first and third terms.
In fact we will estimate them in $L^2_tL^2_x$ and then use H\"older's inequality:
\[
\Vert \mathbf{N}(u^k, \partial w^{n-k}) \Vert_{L^1_tL^2_x}\lesssim  T^{\frac{1}{2}}\Vert \mathbf{N}(u^k, \partial w^{n-k})\Vert_{L^2_tL^2_x}, 
\]
where $ w=u,v,$ and $\mathbf{N}=\mathbf{N}_1$ or $\mathbf{N}=\mathbf{N}_2$ accordingly to \eqref{un-terms}.
We need to estimate the nonlinearity $\mathbf{N}(u^k, \partial w^{n-k})$ in
$L^2_tL^2_x$.  The difficulty here is that the second derivatives of $u$ do
not have $t^{-1}$ uniform decay, instead they decay like
$t^{-\frac{1}{2}}\langle t-r\rangle^{-\frac12-\delta_1}$. 

We begin by noting that in the region $C_T^{out}$ the second derivatives of $u$ do
have $t^{-1}$ uniform decay, so again the argument in Case 1 applies. From here on, we will 
consider the remaining region $C^{in}_T$ which is near the cone and corresponds to dyadic scales $
1 \leq S \ll T$. Here is where we make use of the null structure \eqref{rewrite2} as done in Lemma~\ref{l:equiv lemma}. We successively consider the two terms in  \eqref{rewrite2},
\begin{equation}\label{N_w=u}
\mathbf{N}(u^k, \partial w^{n-k}) = \partial u^k \cdot \mathcal{T} \partial w^{n-k} + \mathcal{T} u^k \cdot \partial^2 w^{n-k} .
\end{equation} 
We consider the $C_{TS}$ partition of the $C_T^{in}$ and will estimate the $L^2_{t,x}(C_{ST})$
norms separately.  As discussed above, we can assume that $S \ll T$;
also we will not distinguish between the $\pm$ (i.e. the interior vs
the exterior of the cone).

We first consider the case where $w=v$.
We estimate the first term in $C_{TS}$ using interpolation restricted to $C_{TS}$:
\[
\begin{aligned}
\Vert \partial u^{k}\cdot \Tau \partial v^{n-k}\Vert_{L^2}&\lesssim \Vert \partial^2 u\Vert_{L^{\infty}} ^{\frac{n-k}{n-1}} \Vert \partial u^{\leq n}\Vert_{L^2}^{\frac{k-1}{n-1}}  \Vert \Tau \partial v\Vert_{L^{\infty}}^{\frac{k-1}{n-1}} \Vert \Tau  v^{\leq n}\Vert_{L^2}^{\frac{n-k}{n-1}} 
\\
&\lesssim C \epsilon  \left( T^{-\frac{1}{2}}S^{-\frac{1}{2}-\delta_1}\right)^{\frac{n-k}{n-1}}   T^{-\frac{1}{2}\frac{k-1}{n-1}}  S^{\frac{1}{2}\frac{n-k}{n-1}} \Vert S^{-\frac{1}{2}}\Tau v^{\le n}\Vert_{L^2} ^{\frac{n-k}{n-1}} \sup_{t\in [T,2T]} \Vert \nabla u^{\le n}(t)\Vert^{\frac{k-1}{n-1}}_{L^2}\\
&\lesssim  C \epsilon T^{-\frac{1}{2}} S^{-\delta_1 \frac{n-k}{n-1}}   \Vert (u^{\leq n}, v^{\le n}) \Vert_{X^T} ,
\end{aligned}
\]
and similarly for the second, where we also use that $S\ll T$ in the region $C_{TS}$:
\[
\begin{aligned}
\Vert \Tau u^k \cdot \partial^2 v^{n-k} \Vert_{L^2}& \lesssim \Vert \Tau \partial u\Vert_{L^{\infty}}^{\frac{n-k}{n-1}} \Vert \Tau  u^{\leq n}\Vert_{L^2}^{\frac{k-1}{n-1}} \Vert \partial^2 v\Vert_{L^{\infty}} ^{\frac{k-1}{n-1}} \Vert \partial ^2v^{\leq n-1}\Vert_{L^2}^{\frac{n-k}{n-1}}
\\
&\lesssim C \epsilon \left(T^{-1}S^{-\delta_1}\right)^{\frac{n-k}{n-1}}  S^{\frac{1}{2}\frac{k-1}{n-1}}   T^{-\frac{k-1}{n-1} + \frac{1}{2}\frac{n-k}{n-1}} \Vert S^{-\frac{1}{2}}\Tau u^{\le n}\Vert_{L^2} ^{\frac{k-1}{n-1}} \sup_{t\in [T,2T]} \Vert \nabla v^{\le n}(t)\Vert^{\frac{n-k}{n-1}}_{L^2}\\
&\lesssim  C \epsilon T^{-\frac{1}{2}} S^{-\delta_1 \frac{n-k}{n-1}}   \Vert (u^{\leq n}, v^{\le n}) \Vert_{X^T} ,
\end{aligned}
\]

Note that commutators between $\Tau$ and derivatives yield  extra  $T^{-1}$ factors and hence give negligible contributions.
The dyadic summation over $S$ is trivial and hence
\[
\Vert \mathbf{N}(u^k, \partial v^{n-k})\Vert_{L^2_t L^2_x}\lesssim C\epsilon T^{-\frac{1}{2}}\Vert (u^{\le n}, v^{\le n})\Vert_{X^T}.
\]

A similar analysis is done on the $L^2_t L^2_x$ norm of \eqref{N_w=u} when $w=u$. Using interpolation restricted to $C_{TS}$ we find for the second term that:
\[
\begin{aligned}
&\Vert \partial^2 u^{n-k}\cdot \Tau u^k\Vert_{L^2}\lesssim  \Vert  \Tau u^k\Vert_{L^{\frac{2(n-1)}{k-1}}} \Vert\partial^2 u^{n-k}\Vert_{L^{\frac{2(n-1)}{n-k}}}
\\
&\lesssim \Vert \Tau \partial u\Vert_{L^{\infty}}^{\frac{n-k}{n-1}} \Vert \Tau  u^{\leq n}\Vert_{L^2}^{\frac{k-1}{n-1}} \Vert \partial^2 u\Vert_{L^{\infty}} ^{\frac{k-1}{n-1}} \Vert \partial ^2u^{\leq n-1}\Vert_{L^2}^{\frac{n-k}{n-1}}
\\
&\lesssim C \epsilon (T^{-1}S^{-\delta_1})^{\frac{n-k}{n-1}}  \left( T^{-\frac{1}{2}}S^{-\frac{1}{2}-\delta_1}\right)^{\frac{k-1}{n-1}} T^{\frac12\frac{n-k}{n-1}}  S^{\frac{1}{2}\frac{k-1}{n-1}}\Vert S^{-\frac{1}{2}}\Tau u^{\le n}\Vert_{L^2} ^{\frac{k-1}{n-1}} \sup_{t\in [T,2T]} \Vert \nabla u^{\le n}(t)\Vert^{\frac{n-k}{n-1}}_{L^2}\\
&\lesssim  C \epsilon T^{-\frac{1}{2}} S^{-\delta_1}   \Vert u^{\leq n} \Vert_{X^T} 
\end{aligned}
\]
The first estimate is treated very similarly and satisfies the same estimate.

The dyadic summation over $S$ is once again trivial and we find 
\[
\Vert \mathbf{N}(u^k, \partial u^{n-k})\Vert_{L^2_t L^2_x}\lesssim C\epsilon T^{-\frac{1}{2}}\Vert (u^{\le n}, v^{\le n})\Vert_{X^T}.
\]
This completes the proof of \eqref{quasin-source} and thus of our proposition. 

\end{proof}

We remark that exactly the same argument also applies to the truncated equation \eqref{WKG-cut}.
Also, implicit in the proof is the fact that we obtain also control over the $X^T$ norm of the solutions.
In addition to that, we will also obtain good control of the localized $L^2$ norms for the 
right hand side of the equation \eqref{WKG-cut}.  To best summarize those, we introduce the 
norms $Y^T$ for functions in $[T,2T] \times \mathbb{R}^2$ by
\begin{equation}
\label{YT}
\|( \mathbf{F},\mathbf{G})\|_{Y^T} = \sup_{1 \leq S \leq T}  T^{\frac12} \| ( \mathbf{F},\mathbf{G})\|_{L^2(C_{TS})}.
\end{equation}
We introduce this norm as a way to measure the RHS of \eqref{WKG-cut},
interpreted as a source term. Such an estimate will be needed later in
the proof of the pointwise estimates derived in Section \ref{s:ks}. 
To formulate the next result we will turn to the set up of the
Proposition \ref{p:energy} and we set $n=2\hh$. Then we have

\begin{proposition}
\label{p:prop}
 Let $n \geq 4$. Assume the solutions $(u,v)$ to either the equations
  \eqref{WKG-cut} or \eqref{WKG} are defined in $\H^n$ in the time interval $[0,
  2T_0]$, and satisfy the bootstrap bounds \eqref{boot1.1}- \eqref{boot3}. Then the following bounds hold:
\begin{equation}
\label{d1-cut}
E^{2\hh} (u, v)(t)\lesssim t^{\tilde{C}\epsilon}E^{2\hh}(u,v)(0), \quad t\in [0,T].
\end{equation}

In addition, for all $1 \leq T \leq T_0$ we have the following estimates in $C_T$: 
\begin{equation}
\label{d1X-cut}
\|\partial^{\leq 2\hh} (u, v)\|_{X^T}^2\lesssim T^{\tilde{C}\epsilon}E^{2\hh}(u,v)(0), 
\end{equation} 
and 
\begin{equation}
\label{d1Y-cut}
\|\partial^{\leq 2\hh-1} (\Box u, (\Box +1) v)\|_{Y^T}^2\lesssim T^{\tilde{C}\epsilon}E^{2\hh}(u,v)(0). 
\end{equation} 
\end{proposition}
\begin{remark} The loss of derivative in the bound \eqref{d1Y-cut} is a consequence of the RHS of \eqref{WKG} or \eqref{WKG-cut} being one derivative higher than what we can control  with the $X^T$ norms.
 \end{remark}

\begin{proof}
The bounds \eqref{d1-cut} and \eqref{d1Y-cut} are direct consequences of Proposition~\ref{p:high-energy}.

The bound \eqref{d1Y-cut} is only partially implicit in the proof, as
one also needs to estimate the first four terms in RHS \eqref{NWKG}.
\end{proof}

We also remark that the result in Proposition \ref{p:prop} proves the 
bounds \eqref{main-energy+} and \eqref{main-RHS+} in Proposition \ref{p:energy}, 
but only when $\ZZ^\gamma$ are regular derivatives.

\section{Vector field energy estimates}
\label{s:fields}

The main goal of this section is to establish energy bounds for the
solution $(u,v)$ to which we have applied a certain number of vector fields
from the family of vector field \eqref{ZZ}. 
Precisely, we want to prove an energy bound for
$(\ZZ^\gamma u, \ZZ^\gamma v)$, where $\gamma$ counts the number of Klainerman vector
fields and spatial and time derivatives applied to the solution
$(u,v)$ of the equation \eqref{WKG}.

All the vector fields $Z$ in \eqref{Z} are related to the geometry of
the problem and they are the generators of the Lorenz transformations of the
Minkowski space $\mathbb{R}^{1+2}$ which preserve the equation
\eqref{WKG-lin}.  Under these circumstances, our aim is to prove an energy inequality for the energy
functional $E^{[2\hh]}$, which we introduced in \eqref{high vf  energy}.

Recall the notations already introduced in the Introduction, where
\[
\ZZ = \{ Z, \partial \}
\]
 is the collection of vector fields we work with. We will use 
multiindices $\alpha$ to count the number of spatial derivatives
and $\beta$ for $Z$ derivatives. We put these together in 
\[
\gamma = (\alpha,\beta),
\]
and set
\[
\ZZ^\gamma = \partial^\alpha Z^\beta .
\]
We use weights to measure the total number of derivatives
\[
|\gamma| = |\alpha| + h |\beta| .
\]
Here we use the parameter $\hh$ to choose a balance between the relative strength of  vector fields versus regular derivatives. This will allow us to work with only two vector fields provided that we have a larger number of regular derivatives, thus enabling  us to use very weak decay assumptions on the initial data. For this we will use the range  $\vert \gamma \vert \leq 2\hh$ which corresponds exactly to two vector fields. Ideally we will want $\hh$ to be as small as possible; its size will be dictated by the Klainerman-Sobolev inequalities in the next section.

One might wish to compare the system satisfied by $(\ZZ^\gamma u, \ZZ^\gamma v)$ with
the linearized system which was studied before in
Section~\ref{s:linearization}. We start by applying $\gamma$ vector fields $\ZZ$ to the equation
\eqref{WKG}. Here the variables that play the role of the linearized
variables $(U,V)$ are $(\ZZ^\gamma u,\ZZ^\gamma v)$. One difference when working with
the spatial rotations or with the Lorentz generators are the commutative
properties of these vector fields with respect to the null structure
nonlinearity.

Thus, applying $\ZZ^\gamma$ vector fields to both hand sides of \eqref{WKG} and denoting $\ZZ^\gamma u =: u^\gamma$
we obtain the inhomogeneous equations
\small
\begin{equation} \label{ZWKG}
\left\{
\begin{aligned}
&\Box u^\gamma (t,x) =  \mathbf{N_1}(v, \partial v^\gamma)  + \mathbf{N_1}(v^\gamma, \partial v) +
\mathbf{N_2}(u, \partial  v^\gamma)  + \mathbf{N_2}(u^\gamma, \partial v)
+
\mathbf{F}^\gamma  \\
&(\Box  + 1)v^\gamma (t,x) =
 \mathbf{N_1}(v, \partial  u^\gamma)  + \mathbf{N_1}(v^\gamma, \partial u) +
\mathbf{N_2}(u, \partial  u^\gamma)  + \mathbf{N_2}(u^\gamma, \partial u)
  + \mathbf{G}^\gamma,
 \\
 \end{aligned} 
\right.
\end{equation}\normalsize
where the source terms $(\mathbf{F}^\gamma,\mathbf{G}^\gamma)$ are of the form
\begin{equation} \label{ZWKG-source}
\mathbf{F}^\gamma := \sum_{\substack{|\gamma_1| + |\gamma_2| \leq |\gamma| \\ |\gamma_1|,|\gamma_2| < |\gamma|  }}
 \mathbf{N}( v^{\gamma_1}, \partial v^{\gamma_2}), \qquad 
\mathbf{G}^\gamma := \sum_{\substack{|\gamma_1| +| \gamma_2| \leq |\gamma| \\ |\gamma_1|,|\gamma_2| < |\gamma|  }}
 \mathbf{N}( v^{\gamma_1}, \partial u^{\gamma_2}).
\end{equation}
Here $\mathbf{N}(\cdot, \cdot)$ is a new linear combination of
quadratic null forms \eqref{null forms} arising from the commutator
terms.

To better understand the structure of the system \eqref{ZWKG} we need a full description of the
quadratic nonlinearities in $(\mathbf{F}^\gamma, \mathbf{G}^\gamma)$.
\begin{lemma} 
All the nonlinear terms in \eqref{ZWKG} are linear combinations of the quadratic null forms \eqref{null forms}.
\end{lemma}
\begin{proof} Simple computations show that this is indeed the case.  Iterations of the following formulas for the Klainerman vector field $\Omega_{0i}$
\begin{equation*}
\left\{
\begin{aligned}
&\Omega_{0i} Q_{12}(\phi,\psi ) = Q_{12}(\Omega_{0i}\phi , \psi) + Q_{12}(\phi, \Omega_{0i}\psi) -(-1)^i Q_{0j}(\phi,\psi), \quad &i,j\in\{1,2\}, i\ne j,\\
&\Omega_{0i} Q_{0j}(\phi,\psi) = Q_{0j}(\Omega_{0i} \phi,\psi) + Q_{0j}(\phi, \Omega_{0i}\psi)+Q_{ij}(\phi, \psi), \quad &i,j\in\{1,2\},\\
&\Omega_{0i} Q_0(\phi,\psi) = Q_0(\Omega_{0i}\phi, \psi) + Q_0(\phi, \Omega_{0i}\psi), \qquad & i=1,2,
\end{aligned}
\right.
\end{equation*}
as well for the $\Omega_{12}$ vector field
\begin{equation*}
\left\{
\begin{aligned}
&\Omega_{12} Q_{12}(\phi,\psi) = Q_{12}(\Omega_{12} \phi,\psi) + Q_{12}(\phi, \Omega_{12}\psi), \\
&\Omega_{12} Q_{0i}(\phi,\psi) = Q_{0i}(\Omega_{12}\phi, \psi) + Q_{0i}(\phi, \Omega_{12}\psi)+(-1)^i Q_{0j}(\phi,\psi), \quad i,j\in \{1,2\}, i\ne j,\\
&\Omega_{12} Q_0(\phi,\psi) = Q_0(\Omega_{12}\phi, \psi)+ Q_0(\phi, \Omega_{12}\psi) ,
\end{aligned}
\right.
\end{equation*}
show that indeed the quadratic nonlinearities have a null structure. Useful in our computations are also the commutator properties of individual vector fields acting on the null structures \eqref{null forms}, which we list below
\begin{equation}
\label{commutators}
\left[\Omega_{0i}, \partial_t \right] = -\partial_i, \quad \! \left[\Omega_{0i}, \partial_j \right] = -\delta_{ij}\partial_t, \quad \! \left[\Omega_{12}, \partial_t\right]=0, \quad \left[\Omega_{12}, \partial_1\right] = \partial_2, \quad \! \left[\Omega_{12}, \partial_2\right] = -\partial_1.
\end{equation}
\end{proof}

Our main vector field energy estimate is as follows:

\begin{proposition}\label{p:vectors} Let $(u,v)$ be solutions to  the equations \eqref{WKG-cut} in the time interval $[0, 2T_0]$, which in
 addition satisfy the corresponding bootstrap bounds \eqref{boot1.1}-\eqref{boot3} Then, for $T'<T_0$, they must also satisfy the bound
\begin{equation}
\label{d2-pre}
E^{[2\hh]}(u, v)(t)\lesssim t^{\tilde{C}\epsilon}E^{[2\hh]}(u,v)(0), \quad t\in [0,T'],
\end{equation} 
where $\tilde{C}$ is a positive constant.
\end{proposition}
The content of Remark \ref{r:remark} remains valid in the context of the above proposition, i.e. we do not distinguish between space and time derivatives in the choice of our vector fields. Observe also that the product $\ZZ^\gamma$ with $|\gamma|\le 2h$ can at most contain two vector fields $Z$ of the family \eqref{Z}.

\begin{proof}
As a preliminary step, we remark that our initial data energy
$E^{[2\hh]}(u,v)(0)$ is equivalent to the square of the norm 
in \eqref{small-data},
\begin{equation}
 E^{[2\hh]}(u,v)(0) \approx \| (u,v)[0]\|_{\mathcal H^{2\hh}}^2 + \| x \partial_x (u,v)[0]\|_{\mathcal H^{\hh}}^2
+ \| x^2 \partial_x^2 (u,v)[0]\|_{\mathcal H^0}^2.
\end{equation}
This is a straightforward elliptic computation 
which is left for the reader. We will take advantage of this observation in order to simplify the analysis in the exterior region $C^{out}$. Indeed, in this region we can directly apply the result of Proposition~\ref{p:out-boot} to obtain the desired bounds on the solution, and we can dispense with the vector field analysis. Precisely, we conclude that we have the exterior fixed
time uniform estimate 
\begin{equation}\label{ext-bd}
E^{[2h]}_{ext}(u,v)(t) \lesssim E^{[2\hh]}(u,v)(0),
\end{equation}
where
\[
E^{[2h]}_{ext}(u,v)(t) = \| (u,v)[t]\|_{\mathcal H^{2\hh}(C^{out})}^2 + \| x \partial_x (u,v)[t]\|_{\mathcal H^{\hh}(C^{out})}^2
+ \| x^2 \partial_x^2 (u,v)[t]\|_{\mathcal H^0(C^{out})}^2.
\]

This bound is stronger than the needed one in \eqref{d2-pre} in the exterior region in two ways: (i) it does not have the $t^{\tilde{C}\epsilon}$ loss, and (ii) it applies to all vector fields 
of size $|x|$ rather than only the $Z$ vector fields.

After these preliminaries, we turn our attention to the evolution of the full vector field energies of $(u,v)$. We begin with several reductions which follow the pattern of previous sections. First we recall that our problem is quasilinear and the energy that can be propagated 
for the derivatives of $(u,v)$ is the quasilinear energy. So denoting
\begin{equation}
E^{[2\hh]}_{quasi}(u,v):=\sum_{|\gamma| \leq 2\hh} 
E^{[2\hh]}_{quasi} (t; \ZZ^\gamma u, \ZZ^\gamma v)
\end{equation}
we will replace the bound \eqref{d2-pre} with the equivalent bound
\begin{equation}
\label{d2}
E^{[2\hh]}_{quasi} (v, u)(t)\lesssim t^{\tilde{C}\epsilon}E^{[2\hh]}_{quasi}(v,u)(0), \quad t\in [0,T'].
\end{equation} 
As before, this reduces to a bound on a dyadic time interval
\begin{equation}
\label{d2-dyadic}
E^{[2\hh]}_{quasi}(u,v)(t)\leq (1+ \tilde C\epsilon) E^{[2\hh]}_{quasi}(u,v)(T), \quad t\in [T,2T]\subset [0, T'].
\end{equation} 
Applying Corollary~\ref{c:nonhlinear}, bound \eqref{d2-dyadic} in turn would follow from a bound for the source terms in the equation \eqref{ZWKG} in the time interval $[T,2T]$,
which we separate into an interior and an exterior part:
\begin{equation}\label{int-source}
\| (\mathbf{F}^\gamma,\mathbf{G}^\gamma) \|_{L^1_t L^2_x(C_T^{int})} \lesssim C \epsilon \| \ZZ^{\leq 2\hh} (u ,v) \|_{X^T},\qquad |\gamma| \leq 2\hh,
\end{equation}
respectively
\begin{equation} \label{out-source}
\| (\mathbf{F}^\gamma,\mathbf{G}^\gamma) \|_{L^1_t L^2_x(C_T^{out})} \lesssim C \epsilon
\| \ZZ^{\leq 2\hh} (u ,v)[0] \|_{\H^0}, \qquad |\gamma| \leq 2\hh .
\end{equation}
 This separation is convenient here because in the exterior region
we have access to the stronger bounds in \eqref{ext-bd} which
simplify matters somewhat. A similar separation could have been implemented in the previous two sections, but there it would have made less
of a difference.

Since $(u,v)$ play symmetric roles in this analysis, we will use 
the notations $w$ and $\ww$ for either $u$ or $v$.
Then we need to estimate 
\[
\| \mathbf{N}(w^{\gamma_1}, \partial \ww^{\gamma_2})\|_{L^1 _tL^2_x} 
\lesssim C \epsilon \|  \ZZ^{\leq 2\hh} (u,v) \|_{X^T},
\]
where $\gamma_1$ and $\gamma_2$ are restricted to the range
\[
|\gamma_1|+|\gamma_2| \leq 2\hh, \qquad |\gamma_1|, |\gamma_2| < 2\hh.
\]
Setting $\gamma_i = (\alpha_i,\beta_i)$, we distinguish three cases:
\begin{itemize}
\item[\textbf{I})] $ \beta_1 = \beta_2 = 0$.  
\item[\textbf{II})] $|\beta_1| = 0$, $|\beta_2| = 1$ or viceversa.
\item[\textbf{III})] $|\beta_1| = |\beta_2| = 1$.
\end{itemize}

The case \textbf{I)}  was already considered in Section \ref{s:higher}. For the case \textbf{II)} we enumerate 
the possibilities:
\[
\mathbf{N}( \partial^{n_1} w, \partial^{n_2} Z \ww) \qquad 1\leq  n_1+ n_2 \leq \hh+1, \ n_2 \leq \hh.
\]
The case $n_1 = 0$ may occur here only if we started with exactly $Z^2$ and one $Z$ was commuted. In this case we get the terms:
\[
\mathbf{N}( w, \partial Z \ww)  
\]
which have not been covered yet. We will still discard this case because it is similar but 
simpler than case \eqref{CZZ} below.
The case $n_1 = 1$ is also identical to the estimates in Section \ref{s:higher}, using directly the 
bootstrap assumptions for the second order derivatives for $u$ and $v$, so we can also discard it.  Thus 
we are left with
\begin{equation}\label{C22}
\mathbf{N}(\partial^{n_1} w, \partial^{n_2} Z\ww)
\qquad 2 \leq n_1 \leq \hh, \ 1\leq n_1+n_2 \leq \hh+1,
\end{equation}
\begin{equation}\label{C40}
\mathbf{N}(\partial^{\hh+1} w,  Z\ww).
\end{equation}

Finally, in case  \textbf{III)} the only terms to consider are
\begin{equation}\label{CZZ}
\mathbf{N}(Zw,\partial Z\ww).
\end{equation}

To bound each of these source terms we follow the same outline as the proofs of the corresponding results discussed in Sections~\ref{s:w-out}, \ref{s:linearization} and \ref{s:higher}. One minor difference is that we will separate the exterior region $C_T^{out}$ from $C_T^{in}$. 
For the main region $C_T^{int}$ we 
refine the analysis even further and prove estimates on the space-time regions $C_{TS}^\pm$. In all cases the dyadic summation in $S$ will be 
trivial. We will repeatedly use the following interpolation Lemma:
\begin{lemma} Assume that $n\geq 0$ and 
\[
\frac{2}{p}=\frac{1}{2}+\frac{1}{q}, \qquad 2\leq q\leq \infty.
\]
Then we have
\begin{equation}
\label{interpolare}
\Vert  \partial^{n+1} Z \phi \Vert_{L^{p}(C_{TS})} \lesssim \|  Z^{\leq 2} \phi  \|_{L^2(C_{TS})}^\frac{1}{2}\left(\| \partial^{\leq 2n} \partial^2 \phi \|^{\frac{1}{2}}_{L^{q}(C_{TS})} + S^{-\frac12}\| \partial \phi \|^{\frac{1}{2}}_{L^{q}(C_{TS})}\right) . 
\end{equation}
The same holds in $C_T^{int}$.
\end{lemma}
\begin{proof} 
Using hyperbolic polar coordinates adapted to $C_{TS}$ (see the next section) this lemma reduces to 
variants of the classical Gagliardo-Nirenberg inequality in a unit sized domain. The details of the proof are included in the appendix.
\end{proof}

\subsection*{A. The null form \texorpdfstring{\eqref{CZZ}}{}.} Here we consider the terms in \eqref{CZZ} and separate the exterior region $C_T^{out}$ and the dyadic interior regions $C_{TS}$. In the latter case, we will not
differentiate between $C_{TS}^+$ and $C_{TS}^-$.

\medskip 

\noindent a) The exterior region $C_T^{out}$. Here we prove the bound
\eqref{out-source} in $C_T^{out}$. 
We neglect the null structure
as well as the vector field structure so that we are left with the
fixed time bound
\[
\| x^2 \partial^2 w  \partial^3 \ww\|_{L^2_x} \lesssim \epsilon T^{-1} (E^{[2h]}_{ext}(\ww))^\frac12.
\]
But this is straightforward as for the first factor we can use 
the $\epsilon r^{-1}$ bound in our bootstrap assumption, and for the second we use the $x^{-2} L^2$ bound in the above exterior norm.

\medskip

\noindent b) The interior region $C_T^{int}$, $v$-terms. Here we consider the terms \eqref{CZZ} for $w=\ww=v$. This is the simpler case and it is instructive to consider it first.  We are considering expressions
of the form
\begin{align}
&\mathbf{N}(Z v, \partial Z v),
\label{vZ-2terms} 
\end{align}
and we want to bound them in $L^1_tL^2_x(C^{int}_{T})$. We actually estimate them in $L^2_t L^2_x(C^{int}_T)$ using that
\[
\Vert \mathbf{N}(Z v, \partial Z v) \Vert_{L^1_t L^2_x (C^{int}_T)}\lesssim T^{\frac{1}{2}} \Vert \mathbf{N}(Z v, \partial Z v) \Vert_{L^2_t L^2_x (C^{int}_T)}.
\]
We neglect the 
null condition and will place
both factors in $L^4(C^{int}_T)$
\[
\Vert \mathbf{N}(Z v, \partial Z v)  \Vert_{L^2(C^{int}_T)} \lesssim 
\Vert \partial Z v \Vert_{L^4(C^{int}_T)} \Vert \partial^2 Z v\Vert_{L^4(C^{int}_T)} ,
\]
where the two terms are estimated using interpolation inequalities as follows:
\begin{equation}
\label{interp1}
\Vert  \partial^2 Z v\Vert_{L^4(C_T^{int})} \lesssim \|  Z^{\leq 2} \partial v \|_{L^2(C_T^{int})}^\frac12 \| \partial^{\leq 2} \partial v \|^{\frac{1}{2}}_{L^\infty(C_T^{int})} ,
\end{equation}
respectively
\begin{equation}
\label{interp2}
\Vert  \partial Z v\Vert_{L^4(C_T^{int})} \lesssim \|  Z^{\leq 2} \partial v \|_{L^2(C_T^{int})}^\frac12 \|\partial  v \|^{\frac{1}{2}}_{L^\infty(C_T^{int})} .
\end{equation}
Here we can use slightly larger sets for the interpolation on the 
right which allows us to localize the estimates.
Combining the two and using our bootstrap assumption we obtain
\[
\Vert \mathbf{N}(Z v, \partial Z v)  \Vert_{L^2_t L^2_x(C_T^{int})} \lesssim C \epsilon T^{-\frac{1}{2}} \Vert Z^{\leq 2} v\Vert_{X^T}
\]
as needed.

\medskip

\noindent c) The $C_{T}^{int}$ region, the $u$ terms. Here we consider 
the nonlinear term \eqref{CZZ} in the case where $w = \ww = u$
\begin{align}
&\mathbf{N}(Z u, \partial Z u),
\label{uZ-2terms} 
\end{align}
which we want to bound them in $L^1_tL^2_x(C^{int}_{T})$. We consider 
separately each of the $C_{TS}$ regions.
This is simpler if $S \approx T$, as the null condition is not needed there and the argument in case (b) applies.

It remains to separately consider the regions $C_{TS}$ with  $1 \leq  S\ll T$, in which case we use the null form structure via the representation \eqref{rewrite1} 
 \begin{equation}
 \label{greu}
\mathbf{N}(Zu, \partial Z u) =   \partial Zu \cdot \mathcal{T} \partial Z u+ \mathcal{T}Zu\cdot  \partial^2 Z u.
 \end{equation}

Again we interpolate between $L^2$ and $L^\infty$ to get $L^4$
but this time in $C_{TS}$. This is acceptable because the $Z$
vector fields are well adapted to $C_{TS}$. We harmlessly commute 
$\Tau$ and $\partial$ with $Z$ at the expense of much better error terms. The bootstrap assumptions \eqref{boot1.1} and \eqref{boot2.1} yield that
\[
|\Tau u|\lesssim \frac{1}{t}|Zu| + \frac{t-r}{t}|\partial u |\lesssim C\epsilon t^{-1}\langle t-r\rangle^\delta
\]
hence for the second term in \eqref{greu} we have
\[
 \begin{aligned}
 \Vert Z \Tau u \cdot \partial^2 Z u  \Vert_{L^2(C_{TS})} &\lesssim  \Vert  Z \Tau  u  \Vert_{L^4(C_{TS})}\Vert \partial^2 Z u \Vert_{L^4(C_{TS})}\\
 &\lesssim  \|  Z^{\leq 2} \Tau u \|_{L^2(C_{TS})}^\frac12 \|   \Tau u \|^{\frac{1}{2}}_{L^\infty} \| Z^{\leq 2}  \partial u \|_{L^2(C_{TS})}^\frac12 (\|  \partial^3  u \|^{\frac{1}{2}}_{L^\infty}
 + S^{-\frac12} \|  \partial  u \|^{\frac{1}{2}}_{L^\infty})
 \\
 &\lesssim \epsilon C  S^{\frac{1}{4}} T^{-\frac{1}{4}} S^{-\frac{1}{4}-\frac12(\delta_1-\delta)} T^{-\frac{1}{4}}\Vert Z^{\leq 2} (u,v)\Vert_{X^T},
 \end{aligned}
\]
and the $S$ summation is straightforward since $\delta\ll \delta_1$.
For the first term in \eqref{greu} we have
\[
\begin{aligned}
\Vert  Z \partial u \cdot \partial Z  \Tau u \Vert_{L^2(C_{TS})} & \lesssim \Vert  Z \partial u\Vert_{L^4(C_{TS})} \Vert\partial Z  \Tau u \Vert_{L^4(C_{TS})} \\
& \lesssim \Vert Z^{\le 2}\partial u\Vert^{\frac{1}{2}}_{L^2(C_{TS})}\Vert \partial u\Vert^{\frac{1}{2}}_{L^\infty(C_{TS})} \Vert Z^{\le 2} \Tau u \Vert^{\frac{1}{2}}_{L^2(C_{TS})} \\
&   \qquad \qquad \qquad \qquad \cdot\left(\Vert \partial^{2} \Tau u \Vert^{\frac{1}{2}}_{L^\infty(C_{TS})} + S^{-\frac12} \Vert \partial \Tau u \Vert^{\frac{1}{2}}_{L^\infty(C_{TS})}\right) \\
& \lesssim C\epsilon T^{\frac{1}{4}}S^{\frac{1}{4}}T^{-\frac{
1}{4}}S^{-\frac{1}{4}} T^{-\frac{1}{2}}S^{-\frac{\delta_1}{2}}\Vert Z^{\leq 2} (u,v)\Vert_{X^T} \\
&\lesssim C\epsilon T^{-\frac{1}{2}}S^{-\frac{\delta_1}{2}} \Vert Z^{\leq 2} (u,v)\Vert_{X^T},
\end{aligned}
\]
where for the last factors we used the bootstrap assumptions \eqref{boot1.2} and \eqref{boot2} which imply 
\[
|\partial^j \Tau u|\lesssim \frac{1}{t}|\partial^j Zu| + \frac{t-r}{t}|\partial^{j+1}u|\lesssim C\epsilon t^{-1}\langle t-r\rangle^{-\delta}, \quad j=\overline{1,2}
\]

\noindent (d) The $C^{int}_T$ region, the mixed $u, v$ terms. Consider first the quadratic term
\[
\textbf{N}(Zu, \partial Zv)
\]
which we need to estimate separately in each $C_{TS}$ region. Again, we only consider the case $1\le S\ll T$, the one where $S \approx T$ being simpler, and use the null structure via the representation \eqref{rewrite2}
\[
\mathbf{N}(Zu, \partial Z v) =   \partial Zu \cdot \mathcal{T} \partial Z v+ \mathcal{T}Zu\cdot  \partial^2 Z v.
\]
For the  first term we use interpolation inequalities 
and our bootstrap assumptions \eqref{boot2.1} and \eqref{boot3} to obtain
\[
\begin{aligned}
\Vert   Z \partial u \cdot  \partial Z \mathcal{T} v \Vert_{L^2(C_{TS})} & \lesssim \Vert  Z \partial  u \Vert_{L^4(C_{TS})} \Vert \partial Z \mathcal{T} v \Vert_{L^4(C_{TS})} \\
& \lesssim \Vert Z^{\le 2}\partial u\Vert_{L^2(C_{TS})}^{\frac{1}{2}}\Vert \partial u\Vert^{\frac{1}{2}}_{L^\infty(C_{TS})}\Vert Z^{\le 2}\Tau v\Vert^{\frac{1}{2}}_{L^2(C_{TS})} \\
& \qquad \qquad \qquad \qquad \qquad \cdot \left(\Vert \partial^2 \Tau v\Vert^{\frac{1}{2}}_{L^\infty(C_{TS})}+ S^{-\frac12}\Vert \partial \Tau v\Vert^{\frac{1}{2}}_{L^\infty(C_{TS})}\right) \\
& \lesssim C\epsilon T^{\frac{1}{4}}T^{-\frac{1}{4}}S^{-\frac{1}{4}}S^{\frac{1}{4}}T^{-\frac{1}{2}-\frac{\delta_1}{2}}S^{\frac{\delta_1}{2}} \Vert Z^{\leq 2} (u,v)\Vert_{X^T} \\
& \lesssim C\epsilon T^{-\frac{1}{2}-\frac{\delta_1}{2}}S^{\frac{\delta_1}{2}} \Vert Z^{\leq 2} (u,v)\Vert_{X^T},
\end{aligned}
\]
where the $S$ summation is straightforward.
The second quadratic term is similar.

The other mixed quadratic term to consider in this scenario is
\[
\mathbf{N}(Zv, \partial Zu)
\]
which is easy to treat as shown below
\[
\begin{aligned}
\Vert \mathbf{N}(Zv, \partial Zu) \Vert_{L^2(C_{TS})}  & \lesssim \Vert \partial Zv \cdot \partial^2 Zu\Vert_{L^2(C_{TS})} \lesssim \Vert \partial Zv\Vert_{L^4(C_{TS})} \Vert  \partial^2 Zu\Vert_{L^4(C_{TS})} \\
& \lesssim \Vert Z^{\le 2} v\Vert^\frac{1}{2}_{L^2(C_{TS})} (\Vert \partial^2 v\Vert^\frac{1}{2}_{L^\infty(C_{TS})}
+ S^{-\frac12}\Vert \partial v\Vert^\frac{1}{2}_{L^\infty(C_{TS})})
\Vert Z^{\le 2}\partial u\Vert^\frac{1}{2}_{L^2(C_{TS})} \\
&  \qquad \qquad  \qquad \qquad \qquad \qquad \qquad \cdot \left(\Vert \partial^3 u\Vert^\frac{1}{2}_{L^2(C_{TS})}
+ S^{-\frac12}\Vert  \partial u\Vert^\frac{1}{2}_{L^2(C_{TS})}\right)
\\
& \lesssim C\epsilon S^\frac{1}{4}T^{-\frac{1}{2}} T^\frac{1}{4} T^{-\frac{1}{4}}S^{-\frac{1}{4}-\frac{\delta_1}{2}} \Vert Z^{\leq 2} (u,v)\Vert_{X^T} \\
& \lesssim C\epsilon T^{-\frac{1}{2}}S^{-\frac{\delta_1}{2}} \Vert Z^{\leq 2} (u,v)\Vert_{X^T}.
\end{aligned}
\]

\medskip

\subsection*{B. The null form  \texorpdfstring{\eqref{C40}}{}}
The arguments here are similar to the ones above.
In the exterior region $C_T^{ext}$ we have \[
|\mathbf{N}(\partial^{\hh+1} w, Z \ww)| \lesssim |x| |\partial^{h+2} \w|
|\partial^2 \ww| ,
\]
and we can bound the first factor in $L^2$,
\[
\||x| \partial^{h+2} w\|_{L^2(C_T^{out})} \lesssim (E^{[2h]}_{ext}(w))^\frac12, 
\]
and the second factor pointwise by $C \epsilon |x|^{-1}$.

The argument in the region $C^{int}_{T}$ is also similar. The bounds for $\mathbf{N}( \partial^{h+1}w,Z\ww )$ in each  $C_{TS}$ are obtained in the same way as the bounds for $\mathbf{N}(Zw, \partial Z\ww)$, 
 replacing the inequalities for $\partial Zw$ and $\Tau Zw$
with
the following interpolation inequality 
\begin{equation}
\label{estC40}
\| \partial^{\hh+1} \psi\|_{L^4} \lesssim \| \partial^{ \leq 2\hh-2} \partial^2 \psi\|_{L^2}^{\frac{1}{2}} \| \partial^2  \psi\|_{L^\infty}^{\frac{1}{2}}, \quad \psi= \partial w, \Tau w
\end{equation}
which holds in each region  $C_{TS}$.

\medskip

\subsection*{C. The null form  \texorpdfstring{\eqref{C22}}{}} 
In this case we follow the same strategy as above, but with the
difference that we can no longer rely on $L^4$ interpolation and
instead we must use other $L^p$ norms.  To fix the notations we simply
consider the worst case when $n_1+n_2=\hh+1$.

We start with the exterior region, where 
we estimate the terms in \eqref{C22} as follows:
\[
|\mathbf{N}(\partial^{n_1} w, \partial^{n_2} Z\ww)| \lesssim |x| |\partial^{n_1+1} w| |\partial^{n_2+2} \ww| .
\]
 We chose exponents $p_1$, $p_2$ so that
\[
p_1=\frac{2(\hh -1)}{n_1-2}, \quad p_{2}=\frac{2(\hh -1)}{n_2}, \quad \frac{1}{p_1}+\frac{1}{p_2}=\frac{1}{2},
\]
and will place  the two factors in $L^{p_1}$ respectively in $L^{p_2}$. At a fixed time $t \in [T,2T]$ we 
use Holder's inequality and interpolation to get the exterior bound
\[
\begin{aligned}
&\| |x| \mathbf{N}(\partial^{n_1} w, \partial^{n_2} Z\ww)\|_{L^2(C_t^{out})} \lesssim \Vert |x| \partial^{n_1+1} w\Vert_{L^{p_1}(C_t^{out})}\Vert  \partial^{n_2+2} \ww \Vert_{L^{p_2}(C_t^{out})}
\\
&\qquad  \lesssim \Vert |x| \partial^{2} w\Vert^{\frac{n_2}{h-1}}_{L^{\infty}(C_t^{out})}\Vert  \vert x \vert \partial^{\leq \hh}  \partial^2 w \Vert^{\frac{n_1-2}{h-1}}_{L^{2}(C_t^{out})} \Vert \partial^2 \ww\Vert^{\frac{n_1-2}{h-1}}_{L^\infty(C^{out}_T)}\Vert \partial^{\le h}\partial^2 \ww \Vert^{\frac{n_2}{h-1}}_{L^2(C^{out}_T)},
\end{aligned}
\]
where the $L^{\infty}$ norms are estimated  using the bootstrap assumption and the $L^2$ norms are estimated using  the outer energy bound, in particular that
\[
\Vert \partial^{\le h}\partial^2 \ww \Vert_{L^2(C^{out}_T)}\lesssim T^{-1}(E^{[2h]}_{ext}(\ww))^\frac12 .
\]

For the interior region $C_{T}^{int}$ we consider separately the sets $C_{TS}$ as before and use 
the null form representation \eqref{rewrite1} in the case where $(w, \ww)=\{(u,u), (u,v), (v,u)\}$. Then we need to estimate the expressions
\begin{equation}
 \label{greu1}
\mathbf{N}(\partial^{n_1} w, \partial^{n_2} Z \ww) =   \partial^{n_1+1}  w \cdot \mathcal{T} \partial^{n_2} Z \ww
+ \mathcal{T} \partial^{n_1} \ww\cdot  \partial^{n_2+1} Z \ww
 \end{equation}
in $L^2(C_{TS})$ for $1 \leq S \lesssim T$.
When $n_2=0$ we use again the $L^4-L^4$ interpolation argument, so we focus here on the case $n_2\ge 1$.
Both terms are treated in a similar manner. For simplicity 
we consider the latter one, in order to facilitate comparison with case (a). We begin with two exponents $p_1$
and $p_2$ given by 
\[
p_1=\frac{4(\hh -1)}{n_1-1}, \quad p_{2}=\frac{4(\hh -1)}{\hh+n_2-2} , \quad \frac{1}{p_1} + \frac{1}{p_2} = \frac12 ,
\]
and estimate
\[
\| \mathcal{T} \partial^{n_1} w \cdot   \partial^{n_2+1} Z  \ww \|_{L^2(C_{TS})} \lesssim 
\| \mathcal{T} \partial^{n_1} w \|_{L^{p_1}(C_{TS})} \|   \partial^{n_2+1} Z  \ww \|_{L^{p_2}(C_{TS})},
\]
where the two terms are estimated using interpolation inequalities as follows:
\begin{equation}
\label{interp11}
\Vert  \Tau \partial^{n_1} w\Vert_{L^{p_1}(C_{TS})} \lesssim \| \partial^{\leq 2\hh-2} \Tau \partial w \|_{L^2(C_{TS})}^\frac{2}{p_1} \| \Tau \partial w \|^{1- \frac{2}{p_1}}_{L^\infty (C_{TS})} ,
\end{equation}
respectively 
\begin{equation}
\label{interp22}
\Vert  \partial^{n_2+1} Z \ww\Vert_{L^{p_2}(C_{TS})} \lesssim \|  Z^{\leq 2} \partial \ww \|_{L^2(C_{TS})}^\frac{1}{2} 
(\| \partial^{\leq 2(n_2-1)}  \partial^3  \ww \|^{\frac{1}{2}}_{L^{p_3}(C_{TS})} +
S^{-\frac12}\|  \partial^2  \ww \|^{\frac{1}{2}}_{L^{p_3}(C_{TS})} ) ,
\end{equation}
where $p_3$ is given by
\[
\frac{1}{p_3}+\frac{1}{2}=\frac{2}{p_2}.
\]
Finally the last terms in \eqref{interp22} are interpolated again as
\[
 \| \partial^{\leq 2(n_2-1)} \partial^2  \ww \|_{L^{p_3}(C_{TS})} \lesssim \| \partial^{\leq 2(\hh-1)} \partial^3  \ww \|_{L^2(C_{TS})}^{\frac{4}{p_2}-1} \|  \partial^3 \ww \|^{2-\frac{4}{p_2}}_{L^{\infty}(C_{TS})},
\]
\[
\|\partial^2\ww\|_{L^{p_3}(C_{TS})}\lesssim \|\partial^2 w\|_{L^2(C_{TS})}^{\frac{4}{p_2}-1}\|\partial^2\ww\|_{L^{\infty}(C_{TS})}^{2-\frac4{p_2}}.
\]
Combining the last five relations and using our
bootstrap assumptions we obtain 
\[
\| \mathcal{T} \partial^{n_1} w \cdot   \partial^{n_2+1} Z  \ww \|_{L^2(C_{TS})}  \lesssim C \epsilon T^{-\frac{1}{2}} S^{-{c\delta_1}}
 \Vert Z^{\leq 2} \ww \Vert_{X^T}^\frac12  \Vert \partial^{\leq 2\hh} (w,\omega) \Vert_{X^T}^\frac12
\]
as needed, where 
\[
c = 
\begin{cases}
1, \quad & (w,\ww) = (u,u)\\ 
\frac{2}{p_2}, \quad & (w, \ww) = (u,v)\\
\frac{2}{p_1}, \quad & (w, \ww) = (v,u).
\end{cases}
\]

The term $\mathbf{N}(\partial^{n_1}v, \partial^{n_2}Zv)$, on the other hand,  does not use the null structure and can be bounded using the same interpolation argument as the one employed above for the product $ \mathcal{T} \partial^{n_1} w \cdot   \partial^{n_2+1} Z $ but carried in the whole interior region $C^{int}_T$ rather than in each $C_{TS}$. From the pointwise bound \eqref{boot3} we get
\[
\begin{aligned}
& \left\| \mathbf{N}(\partial^{n_1}v, \partial^{n_2}Zv) \right\|_{L^2(C^{int}_T)} \lesssim \left\| \partial^{n_1+1}v\cdot \partial^{n_2+1}Zv\right\|_{L^2(C^{int}_T)} \\
& \hspace{10pt} \lesssim \left\|\partial^{\le 2(h-1)}\partial^2 v \right\|^{\frac{2}{p_1}}_{L^2(C^{int}_T)} \left\| \partial^2 v\right\|^{1-\frac{2}{p_1}}_{L^\infty(C^{int}_T)}\left\| Z^{\le 2}\partial v\right\|^\frac{1}{2}_{L^2(C^{int}_T)} \left\|\partial^{\le 2(h-1)}\partial^3 v \right\|^{\frac{2}{p_2}-\frac{1}{2}}_{L^2(C^{int}_T)} \left\| \partial^3v\right\|^{1-\frac{2}{p_2}}_{L^\infty(C^{int}_T)} \\
& \hspace{10pt}\lesssim \epsilon T^{-1} \left\|\partial^{\le 2h}v \right\|^\frac{1}{2}_{X^T} \left\| Z^{\le 2}v\right\|^{\frac{1}{2}}_{X^T}.
\end{aligned}
\]

\end{proof}

Now we are able to finish the proof of Proposition~\ref{p:energy}. The
proof of Proposition \ref{p:vectors} already gives us the $X^T$ bound
of $(u,v)$. It remains to consider the $Y^T$ bound.  The $Y^T$ bounds
without any $Z$ vector fields were already discussed in Section
\ref{s:higher} so we are left with the single $Y^T$ bound that
involves the $Z$ vector field
\begin{equation}
\label{impyt}
\Vert Z(\Box u, (\Box +1)v)\Vert_{Y^T}\lesssim \epsilon CT^{\epsilon \tilde{C}}.
\end{equation}
This is the same as
\begin{equation}
\label{impyt-re}
\Vert \mathbf{N}(Zw, \partial \ww)\Vert_{Y^T} + \Vert \mathbf{N}(w, Z\partial \ww)\Vert_{Y^T} 
\lesssim \epsilon CT^{\epsilon \tilde{C}}.
\end{equation}
These bounds have already been proved in the proof of Proposition
\ref{p:vectors} above in the worst case scenario, which is that of the
$u$-terms, see \eqref{rewrite2}.

\section{Klainerman-Sobolev inequalities}
\label{s:ks}

To recover the bootstrap bounds \eqref{boot1.1} to \eqref{boot3} on the almost global time scale we 
need appropriate Klainerman-Sobolev inequalities, where the aim is to obtain pointwise bounds from the integral bounds. 
Our main result here is a linear result. Unfortunately  we cannot work at fixed time, so instead we work in a dyadic 
time region $C_T$ with $T \geq 1$. The pointwise bounds in the exterior region $C_t^{out}$ were already established in Proposition~\ref{p:out}, so here it remains to concentrate on the region $C_T^{in}$.

\begin{theorem}
\label{p:vreau}
Let $h\geq 8$. Assume that the functions $(u,v)$ in $C^{in}_T$ satisfy the bounds 
\begin{equation}
\label{Zineq3}
\| \ZZ^\gamma (u,v)\|_{X^T} \leq 1, \qquad |\gamma| \leq 2\hh ,
\end{equation}
as well as 
\begin{equation}
\label{boxZineq3}
\|\ZZ^\gamma (\Box u, (\Box+1)v)\|_{Y^T} \leq 1, \qquad |\gamma| \leq  \hh .
\end{equation}
Then they also satisfy the pointwise bounds
\begin{equation}
\label{boot11bis}
\vert Z u \vert  \lesssim 1,
\end{equation}
\begin{equation}
\label{boot11}
| Z \partial^j u| \lesssim  \langle t-r \rangle^{-\delta} \qquad j = \overline{1,2},
\end{equation}
\begin{equation}
\label{boot12bis}
\vert \partial u \vert  \lesssim \langle t\rangle^{-\frac{1}{2}} \langle t-r \rangle ^{-\frac{1}{2}},
\end{equation}
\begin{equation}
\label{boot12}
\vert \partial^{j} u \vert  \lesssim \langle t \rangle ^{-\frac{1}{2}} \langle t-r \rangle ^{-\frac{1}{2}-\delta}, \qquad  j = \overline{2,3}
\end{equation}
\begin{equation}
\label{boot13}
 \vert \partial^j v\vert  \lesssim \langle t \rangle ^{-1-\delta} \langle t-r\rangle^{\delta}, \quad j = \overline{0,3}.
\end{equation}
\end{theorem}
Here $\delta > 0$ is a fixed small constant.
This suffices to prove the almost global well-posedness result. 
The remainder of this Section is devoted to the proof of Theorem~\ref{p:vreau}.

\bigskip

To prove the theorem we  divide the forward half-space $\mathbb{R}^+\times \mathbb{R}^2$
in two regions: $\mathcal{C}^+$ inside the cone, and  $\mathcal{C}^{-}$ outside the cone, and solve the problem separately in each of the regions.  
Here we will allow for a small ambiguity in that the region at distance at most one from the cone can be treated both ways; this corresponds to the fact that the bulk of the $X^T$ norm is invariant with respect to unit (space and time) translations. This is related to the fact that the main result of this paper is also invariant with respect to unit translations. To a large extent, we will think of the bounds of this Theorem as consequences of Sobolev type embeddings or Bernstein  type inequalities. Since our vector fields include the $Z$ vector fields, in order to be able to interpret the bounds in Theorem \ref{p:vreau} we need to work in coordinates which are adapted to the vector fields $Z$. In practice, this means working in hyperbolic coordinates, both in $\mathcal{C}^+$ and in $\mathcal{C}^{-}$. Because of this, in what follows we first introduce the hyperbolic coordinates inside the cone and prove our bounds there, and then introduce the hyperbolic coordinates for the $\mathcal{C}^{-}$ region and again prove the corresponding estimates. Fortunately there will be many similarities between the two regions and some proofs will actually be completely identical.

\subsection{Normalized coordinates inside the cone} 

For the analysis inside the cone it is very convenient to work in the so called \emph{spherical hyperbolic coordinates} in $\mathbb{H}^2\times \mathbb{R}$:
\begin{equation}
\label{sphericalc}
\left\{
\begin{aligned}
&t= e^{\sigma}\cosh (\phi),\\
&x_1=e^{\sigma}\sinh (\phi) \sin (\theta), \\
&x_2=e^{\sigma}\sinh (\phi) \cos(\theta),
\end{aligned}
\right.
\end{equation} 
where $\theta$ and $\phi$ are the polar coordinates in the hyperbolic
space; $\phi$ measures the distance from a point on the hyperboloid to
the origin $(1, 0, 0)$ (the origin is also called the \emph{pole}),
and $\theta$ is the angle from a reference direction. Finally,
$\sigma$ represents the time in the hyperbolic coordinates.  The
Jacobian of the change of variable is
\[
J(\sigma, \phi, \theta) =e^{3\sigma}\sinh(\phi).
\]
We will also need to express the wave operator into this new variables, as it will later become important in our analysis
\begin{equation}
\label{box}
-\Box =e^{-2\sigma}\left( -\partial_{\sigma}^{2}+ \partial_{\phi}^2 +\frac{1}{\sinh ^2(\phi)}\partial_{\theta}^2-\partial_{\sigma} +\frac{\cosh(\phi)}{\sinh (\phi)}\partial_{\phi}\right).
\end{equation}
One can verify very easily the formula in \eqref{box} as well as the following correspondence in between the derivatives relative to the Euclidean or to the hyperbolic coordinates:
\begin{equation}
\label{deriv1}
\left\{
\begin{aligned}
&\partial_{\sigma}= t\partial_t+x_1\partial_{x_1} +x_2\partial_{x_2}=t\partial_t+r\partial_r,\\
&\partial_{\phi} =r\partial_t+\frac{t}{r} \left[  x_1\partial_{x_1} +x_2\partial_{x_2}\right]=r\partial_t+t\partial_r,\\
&\partial_{\theta} =-x_2\partial_{x_1}+x_1\partial_{x_2},
\end{aligned}
\right.
\end{equation}
and respectively
\begin{equation}
\left\{
\begin{aligned}
&\partial_t=e^{-\sigma} \cosh (\phi) \partial_{\sigma}-e^{-\sigma} \sinh (\phi) \partial_{\phi},\\
&\partial_{x_1}= -e^{-\sigma}\sinh (\phi) \sin (\theta)\partial_{\sigma} +e^{-\sigma} \cosh (\phi) \sin (\theta) \partial_{\phi} +\frac{e^{-\sigma} \cos (\theta)}{ \sinh (\phi)} \partial_{\theta},\\
&\partial_{x_2}= -e^{-\sigma}\sinh (\phi) \cos (\theta)\partial_{\sigma} +e^{-\sigma} \cosh (\phi) \cos (\theta) \partial_{\phi} - \frac{e^{-\sigma} \sin (\theta)}{ \sinh (\phi)} \partial_{\theta}.
\end{aligned}
\right.
\end{equation}
Once in these coordinates, the regions $C^+_{TS} $ become essentially rectangular regions of size $1$. Precisely, in spherical hyperbolic coordinates the regions $C^+_{TS}$ are represented as follows
\[
C^+_{TS} \quad \longrightarrow \quad D:= \left\{ (\sigma, \phi, \theta )\; : \; (\sigma, \phi, \theta)\in  I_{\sigma}\times I_{\phi} \times [0, 2\pi]\right\},
\]
where $I_\sigma$ and $I_\phi$ are intervals of size $1$. As discussed earlier we are assuming we are at distance at least one from the cone, which corresponds to $S\geq 1$. Furthermore, in the region $C^+_{TS}$ we have 
\begin{equation}
\label{relations cts-phi}
e^{2\sigma}=t^2-\vert x\vert ^2\approx ST, \qquad e^{\sigma}\cosh (\phi) \approx T, \qquad J\approx\footnote{At least when $\phi >1$ which corresponds to $S\ll T$. A small variation of this computation is needed in the interior region $C_{TT}^+$ .} ST^2 .
\end{equation}

\noindent \textbf{Observation:} Please note  the connection between the new coordinates and the vector fields associated to our problem: 
\begin{itemize} 
\item the scaling vector field \eqref{scal} and the derivative with respect to the time-like variable $\sigma$:
\[
\partial_{\sigma}=\mathscr{S}.
\]
\item the rotations $\Omega_{12}$ and $\Omega_{21}$ are nothing more than the derivative in the $\theta$ direction:
\[
\partial_{\theta}=\Omega_{12}=-\Omega_{21}.
\]
\item  the Lorentz boosts $\Omega_{01}$ and $\Omega_{02}$ are closely connected to the derivative in the $\phi$ direction, which we will denoted by $\Omega_{0r}$: 
\[
\begin{aligned}
&\Omega_{0r}:= \partial_{\phi} =r\partial_t+t\partial_r=\frac{x_1}{r}\Omega_{01}+\frac{x_2}{r}\Omega_{02}.
\end{aligned}
\]
A more useful relation is given by
\[
\left\{
\begin{aligned}
& \Omega_{01}=\cos \theta \partial_{\phi}-\sin \theta \frac{\cosh \phi}{\sinh \phi}\partial_{\theta},\\
&\Omega_{02}=\cos \theta \frac{\cosh \phi}{\sinh \phi}\partial_{\theta} + \sin \theta \partial_{\phi}.
\end{aligned}
\right.
\]

\end{itemize}

Many of our estimates involve the regular gradient, which is very simple in the standard Minkowski coordinates where it can be expressed in the basis $\left\{ \partial_t, \partial_{x_1}, \partial_{x_2}\right\}$, but not as simple when expressed in 
spherical hyperbolic coordinates. Because of this it is very useful to have an alternative basis to measure the size of the gradient which is well-adapted to the geometry of the hyperboloids. It is natural to choose two vector fields which are tangent to the hyperboloids, but then it is not obvious how to complete this to a basis with a third vector field. For simplicity, let us restrict ourselves to the region where $\phi \geq 1$. Then two natural vectors which are tangent to the hyperboloids are $\partial_{\phi}$ and $\partial_{\theta}$, both of which have Euclidean length $T$, and we can choose the first two elements of our new basis to be $T^{-1}\partial_{\phi}$ and $T^{-1}\partial_{\theta}$. For the third vector field we cannot chose $\partial_{\sigma}$ because it is too close in direction to $\partial_{\phi}$. Instead we could choose $\partial_r$ which we can rewrite in the form 

\begin{equation}\label{d-equiv}
2\partial_r=-\frac{1}{t-r}\left( \partial_{\sigma}-\partial_{\phi}\right)+  \frac{1}{t+r}\left( \partial_{\sigma}+\partial_{\phi}\right) .
\end{equation}

Thus, in the $C_{TS}^+$ regions we can use the following three vector fields as a substitute for the gradient:
\begin{equation}
\label{cool}
\nabla_{t,x}\approx \left\{ \,  \frac{1}{T}\partial_{\theta}, \  \frac{1}{T}\partial_{\phi}, \ \frac{1}{S}\left( \partial_{\sigma}-\partial_{\phi}\right) \, \right\}.
\end{equation} 
In the region $\phi <1$, which corresponds to $S\approx T$, the matters are simpler because we can simple use $T^{-1}\partial_{\sigma}$ for the third vector field.

\bigskip

From here we split the analysis in two components: one that deals with the Klein-Gordon pointwise estimates and one that establishes the wave pointwise bounds. 

\subsection{Pointwise bounds for the Klein-Gordon component
inside the cone} 
We work in a dyadic region $C_{TS}^+$, which is foliated by hyperboloids $H_\sigma$ with
\[
e^{2\sigma} \approx ST.
\]

We begin by recalling the components of the $X^T$ norms in $C_{TS}^+$. From the localized energies we have
\begin{equation}
\label{kgneed2}  
\Vert \ZZ^\gamma v\Vert_{L^2_tL^2_x} +  \Vert \ZZ^\gamma \Tau v\Vert_{L^2_tL^2_x}  \lesssim S^{\frac{1}{2}}, \qquad |\gamma| \leq 2\hh ,
\end{equation}
and
\begin{equation}
\label{kgneed2a}  
  \Vert \ZZ^\gamma \nabla  v\Vert_{L^2_tL^2_x}  \lesssim T^{\frac{1}{2}}, \qquad |\gamma| \leq 2\hh .
\end{equation}
Correspondingly we have the stronger $L^2$ bounds on the hyperboloids
\begin{equation}
\label{kgneed2+}  
\Vert \ZZ^\gamma v\Vert_{L^2(H)} +  \Vert \ZZ^\gamma \Tau v\Vert_{L^2(H)}  \lesssim 1, \qquad |\gamma| \leq 2\hh ,
\end{equation}
and
\begin{equation}
\label{kgneed2a+}  
  \Vert \ZZ^\gamma \nabla  v\Vert_{L^2(H)}  \lesssim S^{-\frac12} T^{\frac{1}{2}}, \qquad |\gamma| \leq 2\hh .
\end{equation}

On the other hand for $(\Box+1) v$ we have
\begin{equation}
\label{kgneed2a++}  
  \Vert \ZZ^\gamma (\Box+1) v\Vert_{L^2_tL^2_x}  \lesssim  T^{-\frac{1}{2}}, \qquad |\gamma| \leq \hh .
\end{equation}

\medskip

The next step is to translate the above estimates  in the new coordinates. We will use the subscript $h$ to indicate norms evaluated in the spherical hyperbolic coordinates:
\begin{equation}
\label{kgneed22}  
\Vert \ZZ^\gamma v\Vert_{L^2_h} +  \Vert \ZZ^\gamma \Tau v\Vert_{L^2_h}  \lesssim T^{-1}, \qquad |\gamma| \leq 2\hh,
\end{equation}
and
\begin{equation}
\label{kgneed22a}  
  \Vert \ZZ^\gamma \nabla  v\Vert_{L^2_h}  \lesssim S^{-\frac12} T^{-\frac{1}{2}}, \qquad |\gamma| \leq 2\hh,
\end{equation}
as well as the  $L^2$ bounds on the hyperboloids
\begin{equation}
\label{kgneed22+}  
\Vert \ZZ^\gamma v\Vert_{L^2_h(H)} +  \Vert \ZZ^\gamma \Tau v\Vert_{L^2_h(H)}  \lesssim T^{-1}, \qquad |\gamma| \leq 2\hh,
\end{equation}
and
\begin{equation}
\label{kgneed22a+}  
  \Vert \ZZ^\gamma \nabla  v\Vert_{L^2_h(H)}  \lesssim S^{-\frac12} T^{-\frac{1}{2}}, \qquad |\gamma| \leq 2\hh, 
\end{equation}
while  for $(\Box+1) v$ we have
\begin{equation}
\label{kgneed22a++}  
  \Vert \ZZ^\gamma e^{2\sigma}  (\Box+1) v\Vert_{L^2_h}  \lesssim  S^\frac12 T^{-\frac{1}{2}}, \qquad |\gamma| \leq \hh .
\end{equation}

We  use \eqref{box} to rewrite  
\[
-e^{2\sigma} (\Box +1) =-e^{2\sigma} -\left( \partial_{\sigma}+\frac{1}{2}\right)^{2} +\frac{1}{4} +\partial^2_{\phi} +\frac{1}{\sinh^2 \phi}\partial^2_{\theta}+\frac{\cosh \phi}{\sinh \phi}\partial_{\phi},
\]
and observe that the bounds \eqref{kgneed22+} and \eqref{kgneed22a++}  imply the bound
\begin{equation}
\label{m1}
\left\|  \ZZ^\gamma 
\left( e^{2\sigma} +\left( \partial_{\sigma}+\frac{1}{2}\right)^{2}- \partial_\phi^2 \right) v \right\| _{L^2_{h}} \lesssim  
S^\frac12 T^{-\frac{1}{2}}, \qquad |\gamma| \leq \hh .
\end{equation}
This last bound is strictly speaking not needed here but we have added for completeness; however its counterpart in the exterior region $\mathcal{C}^-$ will be essential.

Our goal is to estimate $v$ and its derivatives pointwise in $C_{TS}^+$. The advantage of working in hyperbolic 
coordinates  is that $C_{TS}^+$ is a region of size $1$. At this point we have two choices: (i) to use Sobolev embeddings
on the hyperboloids or (ii)  to use Sobolev embeddings in the full region. Both strategies work, however
\begin{itemize}
\item[ (i)] is more efficient in terms of derivative counting (choice of $\hh$);

\item[(ii)] also applies in the regions $C_{TS}^-$ exterior to the cone.
\end{itemize}
For these reasons we will alternate between the two strategies from case to case. In the case of the Klein-Gordon problem it will be convenient to use strategy (i) inside the cone and strategy (ii) outside.

The hyperboloids have dimension two so the simplest Sobolev pointwise inequality is
\[
 \Vert v\Vert_{L_h^{\infty}(H)} \lesssim \Vert Z^{\leq 2} v\Vert_{L^2_h(H)}.
\]
One can write this more efficiently as an interpolation inequality
\begin{equation}\label{easy-interp}
 \Vert v\Vert_{L_h^{\infty}(H)} \lesssim \Vert Z^{\leq 2} v\Vert_{L_h^{2}(H)}^\frac12  \Vert v\Vert_{L_h^{2}(H)}^\frac12 .
\end{equation}
For higher derivatives we need a similar bound for derivatives of $v$:

\begin{lemma}
The following interpolation estimate holds in $C_{TS}^+$:
\begin{equation}\label{better-interp}
\sup_{\rho} \Vert \partial^{\leq \hh} w\Vert_{L_h^{\infty}(H)} \lesssim \left(\sup_{\rho}\Vert Z^{\leq 2} w\Vert_{L_h^{2}(H)}\right)^\frac12 
\left( \sup_\rho \Vert \partial ^{\leq 2\hh} w\Vert_{L_h^{2}(H)}\right)^\frac12 .
\end{equation}
\end{lemma}
\begin{proof}
Here we cannot argue on a single hyperboloid because $D$ also contains transversal derivatives. Expressed in spherical hyperbolic coordinates, using the substitute \eqref{cool} for the space-time gradient, this becomes a standard interpolation inequality. The proof is omitted, as it uses the same coordinates and the same principles as the proof of Lemma~\ref{l:interpolation} in the appendix.

\end{proof}

The desired pointwise bounds for $v$ and its derivatives follow directly from this interpolation inequality provided that  $\hh \geq 3$.

\subsubsection{ Extra gain near the cone.}
Our goal here is to show that near the cone the pointwise bounds for $v$ and its derivatives 
improve to 
\begin{equation}\label{est_partialjv}
| \partial^j v| \leq T^{-1-\delta} \langle t-r \rangle^{\delta}, \qquad \delta > 0, \qquad j = 0,1,2,3.
\end{equation}

Assuming we have no gain when $j = 4$, by interpolation it suffices to have a gain in the $j = 0$ bound. By the interpolation estimate \eqref{easy-interp}
it suffices to improve the $L^2$ bound
on hyperboloids in \eqref{kgneed22+} to
\begin{equation}
\label{need-2H}  
\Vert  v\Vert_{L^2_h(H)}   \lesssim S^{\delta}T^{-1-\delta}.
\end{equation}

\medskip

As a first  starting point we begin with the similar bound without a gain in \eqref{kgneed22+}, namely 
\begin{equation}
\label{got1}  
\Vert  v\Vert_{L^2_h(H)}   \lesssim T^{-1}.
\end{equation}
A second starting point is obtained 
from \eqref{kgneed22+} with exactly one derivative via the relation \eqref{d-equiv}, where $\partial_\sigma v$ is a vector field bound and thus better. This yields 
\begin{equation}
\label{got2}  
\Vert  \partial_\sigma v\Vert_{L^2_h(H)}   \lesssim ST^{-1} .
\end{equation}

Finally, our third starting point 
comes from \eqref{m1} with $\gamma = 0$. In that case we can use \eqref{kgneed22+} with $Z^\gamma = \partial_\phi^2$ as well as \eqref{got1} and \eqref{got2}
to simplify \eqref{m1} to 
\begin{equation}
\label{got3}
\left\|  
\left( e^{2\sigma} + \partial_{\sigma}^{2}
 \right) v \right\| _{L^2_{h}} \lesssim  
S^\frac12 T^{-\frac{1}{2}}.
\end{equation}
It remains to show that the bounds \eqref{got1}, \eqref{got2} and \eqref{got3}
imply \eqref{need-2H}. Here we recall that 
the set $C_{TS}$ corresponds to an unit interval in $\sigma$.

\medskip

We will interpret \eqref{need-2H}
as coming from an energy estimate on hyperboloids, which in the new coordinates are the sets $\sigma = const$.
The natural energy associated to the operator in \eqref{got3} is 
\[
E(v)  = |\partial_\sigma v|^2 + e^{2\sigma} v^2 .
\]
Here we omit the $\phi$ and $\theta$ integration as it plays no role at all in the proof of the estimate \eqref{need-2H}.
Since in $C_{TS}$ we have $e^{2\sigma} 
\approx ST$, in order to prove \eqref{need-2H} it suffices to show 
that 
\begin{equation}
\label{need-2Ha}  
E(v) \lesssim S^{1+\delta}T^{-1-\delta}.
\end{equation}
For this we will use an energy estimate. We  compute
\[
\frac{d}{d\sigma} E(v) 
= 2e^{2\sigma} v^2 + 2 \partial_{\sigma}v 
\left( e^{2\sigma} + \partial_{\sigma}^{2}
 \right) v = O(E(v)) + 2 \partial_{\sigma}v 
\left( e^{2\sigma} + \partial_{\sigma}^{2}
 \right) v .
\]
Initializing at some point $\sigma = \sigma_0$, we use \eqref{got2} and \eqref{got3} to estimate the last term on the right as 
\[
\| \partial_{\sigma}v 
\left( e^{2\sigma} + \partial_{\sigma}^{2}
 \right) v \|_{L^1} \lesssim S^{\frac32} T^{-\frac32} ,
\]

and then Gronwall's inequality to arrive 
at 
\[
\sup_{\sigma} E(v)(\sigma) \lesssim E(v)(\sigma_0) + S^{\frac32} T^{-\frac32}.
\]
This suffices provided we find a good initialization. By \eqref{got2} any $\sigma_0$ is good for the first, kinetic  component of the energy, but we have a problem with the second, namely the potential energy. To address this difficulty 
we use the principle that the two components 
should be comparable in average. To prove this, we use a compactly supported nonnegative cutoff $\chi(\sigma)$ and
integrate by parts
\[
\int \chi(\sigma)(|\partial_\sigma v|^2 - e^{2\sigma} v^2) \, d\sigma = - \int \chi(\sigma) v \left( e^{2\sigma} + \partial_{\sigma}^{2}
 \right) v\, d\sigma + \int \frac12 \chi''(\sigma) v^2 \,d\sigma = O(S^{\frac12} T^{-\frac32}),
\]
where the integrals on the right were estimated using \eqref{got1} and \eqref{got3}. This implies that  
we also have 
\[
\int \chi(\sigma)e^{2\sigma} v^2 \, d\sigma \lesssim S^{\frac32} T^{-\frac32},
\]
and in turn allows us to find a good initialization $\sigma_0$ so that 
\[
E(v)(\sigma_0) \lesssim S^{\frac32} T^{-\frac32} .
\]
Then \eqref{need-2Ha} follows with $\delta = \frac12$, hence showing that \eqref{need-2H} holds true with $\delta=\frac14$ and \eqref{est_partialjv} with $\delta=\frac18$.

\subsection{Pointwise bounds for the wave equation inside the cone} 
In the analysis for the wave component we want to obtain pointwise
bounds on $C_{TS}^+$ for $\nabla u$  and its derivatives (first and second) and then for $\Tau u $ (tangential
derivatives) as well as its derivatives (first and second).

Applying \eqref{Zineq3} we will  control 
\begin{equation}
\label{wneed3}  
\Vert \ZZ^\gamma \Tau u\Vert_{L^2_tL^2_x}\lesssim S^{\frac{1}{2}}, \qquad |\gamma| \leq 2\hh 
\end{equation}
as well as the gradient of $u$
\begin{equation}
\label{wneed2}  
\Vert \ZZ^\gamma \nabla u\Vert_{L^2_t L^2_x}\lesssim T^{\frac{1}{2}}, \qquad |\gamma| \leq 2\hh .
\end{equation}

On hyperboloids we have
\begin{equation}
\label{wneed3+}  
\Vert \ZZ^\gamma \Tau u\Vert_{L^2(H)}\lesssim 1, \qquad |\gamma| \leq 2\hh ,
\end{equation}
\begin{equation}
\label{wneed2+}  
\Vert \ZZ^\gamma \nabla u\Vert_{L^2(H)}\lesssim S^{-\frac12} T^{\frac{1}{2}}, \qquad |\gamma| \leq 2\hh .
\end{equation}

Finally from the wave equation we get
\begin{equation}
\label{w}  
  \Vert \ZZ^\gamma  \Box u \Vert_{L^2_t L^2_x}  \lesssim  T^{-\frac{1}{2}}, \qquad |\gamma| \leq \hh .
\end{equation}

We  translate all the estimates above in spherical hyperbolic coordinates
\begin{equation}
\label{wneed2c}  
\Vert \ZZ^\gamma \Tau u\Vert_{L^2_{h}}\lesssim  T^{-1}, \quad |\gamma|\le 2h
\end{equation}
and
\begin{equation}
\label{wneed2c2}  
\Vert \ZZ^\gamma \nabla u\Vert_{L^2_{h}}\lesssim  (ST)^{-\frac12}, \quad |\gamma|\le 2h.
\end{equation}
On the hyperboloids:
\begin{equation}
\label{wneed2c+}  
\Vert \ZZ^\gamma \Tau u\Vert_{L_h^2(H)}\lesssim  T^{-1}, \quad |\gamma|\le 2h,
\end{equation} 
and
\begin{equation}
\label{wneed2c2+}  
\Vert \ZZ^\gamma \nabla u\Vert_{L^2_h(H)}\lesssim  (ST)^{-\frac12}, \quad |\gamma|\le 2h.
\end{equation}

Finally for the wave equation we have
\begin{equation}
\label{wneed1c}  
\Vert  \ZZ^\gamma e^{2\sigma} \Box u \Vert_{L^2_h} \lesssim S^{\frac{1}{2}} T^{-\frac12}, \qquad |\gamma| \leq \hh.
\end{equation}
Combining this with \eqref{box} we obtain 
\begin{equation}
\label{wm1}
\left\|  \ZZ^{\gamma} (\partial_{\sigma}-\partial_{\phi}) (\partial_{\sigma}+\partial_{\phi}+1) u \right\| _{L^2_{h}}
\lesssim S^{\frac12}T^{-\frac12}  \qquad |\gamma| \leq \hh .
\end{equation}

At this point we will show how to use the bounds on hyperboloids to prove the estimates \eqref{boot11}, \eqref{boot12bis} as well as \eqref{boot12} but with $\delta=0$.

From \eqref{wneed2c} and \eqref{wneed2c2} we deduce
\begin{equation}
\label{wneed2c2g}  
\Vert \ZZ^\gamma Z  u\Vert_{L_h^2(H)}\lesssim  1, \qquad |\gamma| \leq 2\hh,
\end{equation}
as well as
\begin{equation}
\label{wneecomp}  
\Vert \ZZ^\gamma (\partial_{\sigma} -\partial_{\phi}) u\Vert_{L_h^2(H)}\lesssim  S^\frac12 T^{-\frac12}\qquad |\gamma| \leq 2\hh .
\end{equation}
The interpolation bound \eqref{better-interp} yields
\[
|\partial^{\leq \hh} Zu| \lesssim  1,
\]
and
\[
|\partial^{\leq \hh}(\partial_{\sigma} -\partial_{\phi})  u| \lesssim  S^\frac12 T^{-\frac12}.
\]
Assuming that $\hh\geq 2$, these last two bounds translated back in terms of regular derivatives give exactly the estimates \eqref{boot11}, \eqref{boot12bis} as well as \eqref{boot12} with $\delta=0$. In order to obtain the $\delta$ gain we need however to return to the $L^2$ bounds in $C^+_{TS}$ rather than work on hyperboloids. Since this situation is identical to the one we will encounter in the exterior region $\mathcal{C}^-$, we postpone this proof to Section \ref{s:w-out}.

\subsection{Normalized coordinates outside the cone} For the analysis outside the cone it is very convenient to work with the natural counterpart of the  \emph{spherical hyperbolic coordinates} used in the interior. These are:
\begin{equation}
\label{sphericalc-out}
\left\{
\begin{aligned}
&t= e^{\sigma}\sinh (\phi),\\
&x_1=e^{\sigma}\cosh (\phi) \sin (\theta), \\
&x_2=e^{\sigma}\cosh (\phi) \cos(\theta),
\end{aligned}
\right.
\end{equation} 
where $\theta$ and $\phi$ are the polar coordinates on  the single sheeted hyperboloid; $\phi$ plays the role of the radial coordinate and provides the arch length parametrization of the radial time-like geodesics on the hyperboloid 
and $\theta$ is the angle from a reference direction. Finally,
$e^{2\sigma}$ represents the Minkowski distance to the origin. The Jacobian of the change of variable is
\[
J(\sigma, \phi, \theta) =e^{3\sigma}\cosh(\phi).
\]
The wave operator in these new variables has the form 
\begin{equation}
\label{box-out}
-\Box =e^{-2\sigma}\left( \partial_{\sigma}^{2}- \partial_{\phi}^2 +\frac{1}{\cosh^2(\phi)}\partial_{\theta}^2-\partial_{\sigma} +\frac{\sinh(\phi)}{\cosh (\phi)}\partial_{\phi}\right).
\end{equation}

In these coordinates, the regions $C^-_{TS} $ become essentially rectangular regions of size $1$. Precisely, in spherical hyperbolic coordinates the regions $C^-_{TS}$ are represented as follows
\[
C^-_{TS} \quad \longrightarrow \quad D:= \left\{ (\sigma, \phi, \theta )\; : \; (\sigma, \phi, \theta)\in  I_{\sigma}\times I_{\phi} \times [0, 2\pi]\right\},
\]
where $I_\sigma$ and $I_\phi$ are intervals of size $1$. As discussed earlier we are assuming we are at distance at least one from the cone, which corresponds to $S\geq 1$. Here we work under the assumption that $S\lesssim T$, which selects a conical neighbourhood of the cone, as the analysis in the outer part is much simpler (see Proposition~\ref{p:out}). Therefore, in such region $C^-_{TS}$ we have 
\begin{equation}
\label{relations cts-phi-out}
e^{2\sigma}=\vert x\vert ^2-t^2 \approx ST, \qquad e^{\sigma}\cosh (\phi) \approx T, \qquad J\approx ST^2 .
\end{equation}

\subsection{Pointwise bounds for the Klein-Gordon component
outside the cone} Here we prove the pointwise bounds for the Klein-Gordon equation in the exterior region. We harmlessly assume that $S \leq T$, which is the interesting region near the cone. 
Arguing in the same way as for the interior region, we move the bounds \eqref{Zineq3}, \eqref{boxZineq3} in the hypothesis of the theorem  to the spherical hyperbolic coordinates; here there are no hyperboloid bounds. We get
\begin{equation}
\label{kgneed22-out}  
\Vert \ZZ^\gamma v\Vert_{L^2_h} +  \Vert \ZZ^\gamma \Tau v\Vert_{L^2_h}  \lesssim T^{-1}, \qquad |\gamma| \leq 2\hh,
\end{equation}
and
\begin{equation}
\label{kgneed22a-out}  
  \Vert \ZZ^\gamma \nabla  v\Vert_{L^2_h}  \lesssim S^{-\frac12} T^{-\frac{1}{2}}, \qquad |\gamma| \leq 2\hh,
\end{equation}
while  for $(\Box+1) v$ we have
\begin{equation}
\label{kgneed22a++out}  
  \Vert \ZZ^\gamma e^{2\sigma}  (\Box+1) v\Vert_{L^2_h}  \lesssim  S^\frac12 T^{-\frac{1}{2}}, \qquad |\gamma| \leq \hh .
\end{equation}

We  use \eqref{box-out} to rewrite  
\[
-e^{2\sigma} (\Box +1) =-e^{2\sigma} +\left( \partial_{\sigma}-\frac{1}{2}\right)^{2} -\frac{1}{4} -\partial^2_{\phi} +\frac{1}{\cosh^2 \phi}\partial^2_{\theta}+\frac{\sinh \phi}{\cosh \phi}\partial_{\phi},
\]
and observe that the bounds \eqref{kgneed22a-out} and \eqref{kgneed22a++out}  give
\begin{equation}
\label{m1-out}
\left\|  \ZZ^\gamma 
\left( e^{2\sigma} -\left( \partial_{\sigma}-\frac{1}{2}\right)^{2}+ \partial_\phi^2 \right) v \right\| _{L^2_{h}} \lesssim  
S^\frac12 T^{-\frac{1}{2}}, \qquad |\gamma| \leq \hh.
\end{equation}

Our goal is to estimate $v$ and its derivatives pointwise in $C_{TS}^-$.  There we can set (see formula \eqref{cool} which still applies)
\begin{equation}
\label{cool2}
Z = \left\{ \partial_\theta, \partial_\phi \right\} , \qquad \Tau = T^{-1} Z, \qquad  \nabla_{t,x} = \left\{ T^{-1} Z, S^{-1}(\partial_\phi - \partial_{\sigma} )\right\},
\end{equation}
so that everything is constant coefficients in $\theta$ and $\phi$ within $C^-_{TS}$. At this point we can harmlessly localize on the unit scale in $\phi$, then freeze the constants and finally  forget about about the $\phi$ localization and assume that $\phi$ is either on $\mathbb{R}$ or on the circle.

We localize to a frequency $\lambda$  in $\theta$, $\phi$. Based on the symbol of the operator in \eqref{m1-out}, the interesting threshold for $\lambda$ is 
$\sqrt{ST}$. For smaller $\lambda$, the $\partial_\phi^2$ component is  controlled by $e^{2\sigma}$ and thus perturbative  in \eqref{m1-out}, which can then 
be treated as an elliptic bound.

Based on this we distinguish two cases:

\bigskip

(i) Large $\lambda$, namely $ \lambda \gtrsim \sqrt{ST}$. Then we disregard \eqref{m1-out} and work only with 
\eqref{kgneed22-out} and \eqref{kgneed22a-out}. There we have a second interesting frequency, namely the one for 
$\partial_\sigma-\partial_\phi$. We localize this dyadically to the frequency $\mu$. Then from \eqref{cool2}
\[
\nabla_{t,x} \approx T^{-1} \lambda + S^{-1} \mu.
\]  
Hence  from  \eqref{kgneed22-out} and \eqref{kgneed22a-out} we have the following $L^2$ bound for the corresponding component $v_{\lambda \mu}$ of $v$: 
\[
\| v_{\lambda \mu}\|_{L^2} \lesssim T^{-1} \frac{1}{\left[  \lambda^2 + (T^{-1} \lambda + S^{-1} \mu)^{2\hh} \right] \left[1+ (T^{-1} \lambda + S^{-1} \mu) S^{\frac12}T^{-\frac12}\right]}\, .
\]
Applying Bernstein we arrive at the $L^\infty$ bound
\[
\| \nabla^j v_{\lambda \mu}\|_{L^\infty} \lesssim T^{-1} \frac{\lambda \mu^{\frac12}  (T^{-1} \lambda + S^{-1} \mu)^{j}
}{ \left[ \lambda^2 + (T^{-1} \lambda + S^{-1} \mu)^{2\hh} \right]\left[1+ (T^{-1} \lambda + S^{-1} \mu) S^{\frac12}T^{-\frac12}\right]} \, .
\]
It remains to maximize the right hand side with respect to $\lambda$ and $\mu$ subject to the constraint $\lambda \geq \sqrt{ST}$.

On $\mu$ we have no constraint, and the maximum is attained when its contribution  in the first factor in the denominator 
 balances that of $\lambda$. Depending on which $\lambda$ term gets balanced, we have two scenarios:

\medskip

(a) ($S^{-1} \mu)^{2\hh} = \lambda^2 \gtrsim (T^{-1} \lambda)^{2\hh}$. Then the above expression becomes
\[
T^{-1} S^\frac12  \lambda^{-1 + \frac{j}{h}+\frac{1}{2h}} \left( 1+\lambda^{\frac{1}{h}} S^{\frac12} T^{-\frac12}\right)^{-1} ,
\]
which is decreasing in $\lambda$ and hence maximized at $\lambda = \sqrt{ST}$. We get a maximum of 
\[
T^{-1} S^\frac12  (ST)^{\frac12(-1 + \frac{j}{h}+\frac{1}{2h})} \left( 1+(ST)^{\frac{1}{2h}}S^{\frac12}T^{-\frac{1}{2}}\right)^{-1}
\lesssim T^{-1} S^\frac12  (ST)^{\frac12(-1 + \frac{j}{h}+\frac{1}{2h})} .
\]
This is favorable (i.e. $\leq T^{-1-\delta}$) if  $h > 2j +1$. Since we need $j \leq 3$, we should have $\hh \geq 8$.

\medskip

b) $(S^{-1} \mu)^{2\hh} = (T^{-1} \lambda)^{2\hh} \gtrsim \lambda^2 $. We get again a negative power of $\lambda$ 
which is maximized if $\lambda$ is as small as possible within these constraints.
But this cannot happen 
at $\lambda = \sqrt{ST}$, as this is inconsistent with the above constraint. Therefore it happens when the two cases (b) and (a) 
are in balance. So the maximum is always attained in case (a).

\bigskip 

(ii) Small $\lambda$, namely $\lambda \ll \sqrt{ST}$. The contribution of $\partial_\phi^2$   in \eqref{m1-out}
can be estimated by
\[
\|\ZZ^\gamma \partial^2_\phi v\|_{L^2_h}\ll  ST 
\|\ZZ^\gamma v\|_{L^2_h},
\quad |\gamma|\le h ,
\]
while the remaining operator $e^{2\sigma} -\partial^2_\sigma$ is elliptic so we have
\[
ST \|\ZZ^\gamma v\|_{L^2_h} + \| \partial_\sigma^2 \ZZ^\gamma v \|_{L^2_h}  \lesssim  \| (e^{2\sigma} -\partial^2_\sigma) v\|_{L^2_h}
, \quad |\gamma|\le h .
\]
Combining these two estimates with \eqref{m1-out} we obtain the elliptic bound
\begin{equation}
\label{v-ell}
ST \|\ZZ^\gamma v\|_{L^2_h} + \| \partial_\sigma^2 \ZZ^\gamma v \|_{L^2_h}  \lesssim  
S^\frac12 T^{-\frac{1}{2}}, \qquad |\gamma| \leq \hh.
\end{equation}
We keep the same $\mu$ notation but introduce a third dyadic scale $\gamma$ for the $\sigma$ frequency.

 Combining the $L^2$ bounds \eqref{kgneed22-out} and \eqref{v-ell}, as above we get the $L^\infty$ bound 
\[
\|\nabla^j v_{\lambda \mu \gamma}\|_{L^\infty} \lesssim  T^{-1}
 \frac{\lambda \min\{ \mu,\gamma\}^{\frac12}  (S^{-1} \mu+T^{-1} \lambda)^{j}
}{ \left( \lambda^2 + (S^{-1} \mu + T^{-1} \lambda )^{2\hh}\right) + (ST)^{-\frac12} ( ST + \gamma^2)(\lambda +   (S^{-1} \mu+ T^{-1} \lambda)^{\hh})    }.
\]

Now we harmlessly replace $(S^{-1} \mu+T^{-1} \lambda)^{j}$
by $\lambda^{\frac{j}h} + (S^{-1} \mu+T^{-1} \lambda)^{j}$
at the numerator. Then we
can drop the $T^{-1} \lambda $ term to get
\[
\|\nabla^j v_{\lambda \mu \gamma}\|_{L^\infty} \lesssim  T^{-1}
 \frac{\lambda \min\{ \mu,\gamma\}^{\frac12}  (\lambda^{\frac{j}h}+ (S^{-1} \mu)^{j})
}{ \left( \lambda^2 + (S^{-1} \mu)^{2\hh}\right) + (ST)^{-\frac12} ( ST + \gamma^2)(\lambda +   (S^{-1} \mu)^{\hh})    }.
\]

We need to maximize the right hand side. The $\mu$ maximum is attained when $\lambda = (S^{-1} \mu)^\hh$,
in which case the above expression simplifies to
\[
T^{-1}
 \frac{ \min\{ \mu,\gamma\}^{\frac12}  \lambda ^{\frac{j}{\hh}}
}{  \lambda  + (ST)^{-\frac12} ( ST + \gamma^2) }.
\]
Since $\lambda < \sqrt{ST}$ we can drop it from the first term from the denominator and then maximize it at the numerator.
Finally, the $\gamma$ maximum is attained if $\gamma = \sqrt{ST}$. We get a maximum of 
\[
T^{-1}  \frac{ \min\{ S^\frac12 (ST )^{\frac{1}{4\hh}},(ST)^{\frac14}\}   (ST)^{\frac{j}{2\hh}}
}{ (ST)^{\frac12} }.
\]
Using the second term in the min,
this is favorable if $2j<h$ i.e. the same as in case (i).

\begin{remark} Here we gain a better bound of $T^{-1-\delta}$ which shows that outside the cone we have better decay for the Klein-Gordon component.
\end{remark}

\subsection{Pointwise bounds for the wave component outside the cone}
\label{s:w-out}
Here we prove the pointwise bounds for the wave equation in the
exterior region. As before we harmlessly assume that $S \leq T$.
Arguing in the same way as for the interior region, we move the bounds
\eqref{Zineq3}, \eqref{boxZineq3} in the hypothesis of the theorem to
the spherical hyperbolic coordinates; here there are no hyperboloid
bounds. In each region $C^-_{TS}$ we get
\begin{equation}
\label{wneed2c2g+}  
\Vert \ZZ^\gamma Z u\Vert_{L^2_{h}}\lesssim  1, \qquad |\gamma| \leq 2\hh ,
\end{equation}
and
\begin{equation}
\label{wneecomp+}  
\Vert \ZZ^\gamma (\partial_{\sigma} -\partial_{\phi}) u\Vert_{L^2_{h}}\lesssim  S^\frac12 T^{-\frac12}, 
\qquad |\gamma| \leq 2\hh ,
\end{equation}
as well as the  bounds
\begin{equation}
\label{m1+}
\left\| \ZZ^\gamma  (\partial_{\sigma}-\partial_{\phi}) (\partial_{\sigma}+\partial_{\phi}+1) 
u \right\| _{L^2_{h}}\lesssim S^\frac12 T^{-\frac12} , \qquad |\gamma| \leq \hh .
\end{equation}
All the analysis below applies equally not only to $C^{-}_{TS}$ but
also to $C^+_{TS}$, providing an alternative approach in the latter
case. For this reason we drop the $\pm$ superscripts below.  We are
allowed to freely localize in $\phi$ on the unit scale in all of the
$C_{TS}$, as well as localize on the unit scale in $\sigma$ if $S
\approx T$.  For the purpose of the proofs below we assume these
localizations. 

Relative to the variable $\sigma$, we are not allowed to localize 
directly on the unit scale in the full set of inequalities 
\eqref{wneed2c2g+}, \eqref{wneecomp+}  and \eqref{m1+}.
However we can finesse this minor difficulty if we agree to use 
only the pair of bounds \eqref{wneed2c2g+}, \eqref{wneecomp+}
\footnote{with a weaker bound of $1$ in the second case} to prove 
\eqref{boot11bis} and \eqref{boot11}, and the pair of bounds \eqref{wneecomp+}  and \eqref{m1+}
to prove \eqref{boot12} and \eqref{boot12bis} (modulo \eqref{boot11bis},
can be recast as bounds for $(\partial_\sigma - \partial_\phi) u$). 
Then in the first case we can localize the function $u$ 
on the unit scale in $\sigma$, whereas in the second case we can instead localize
the function $(\partial_\sigma - \partial_\phi) u$ on the unit scale in $\sigma$.

\bigskip

Our main tool will be the following Sobolev pointwise  inequality:

\begin{lemma}
For functions $w$ compactly supported in $C_{TS}$ we have the following pointwise bound
(interpolation inequality):
\begin{equation}
\label{interpolate-2h}
\Vert \partial^{\leq \hh/2} w\Vert_{L^{\infty}(C_{TS})} \lesssim \Vert \ZZ^{\leq 2\hh} w\Vert_{L^{2}_{h} (C_{TS})} 
+ \left\|  \ZZ^{\leq \hh} \left( \partial_{\sigma}\pm \partial_{\phi}\right) w \right\| _{L^{2}_{h} (C_{TS})}.
\end{equation}
\end{lemma}
As in previously discussed interpolation inequalities, in hyperbolic coordinates 
this can be viewed as a standard constant coefficient bound which is obtained from Bernstein type
inequalities.  For instance if $h = 0$, the inequality becomes
\[
\| w\|_{L^\infty_h} \lesssim \| \partial_{\theta,\phi}^{\leq 2}
 w\|_{L^2_h} + \| \partial_{\theta,\phi}^{\leq 1}  \left( \partial_{\sigma}\pm \partial_{\phi}\right)
 w\|_{L^2_h} .
\]
Denoting by $\lambda$ the $\partial_{\phi,\theta}$ frequency and by $\mu$ the $\partial_{\sigma}\pm \partial_{\phi}$ frequency,
by Bernstein's inequality the above bound reduces to
\[
\lambda \mu^\frac12 \lesssim (1+ \lambda^2) + \mu (1+\lambda),
\]
which is straightforward, also with room for dyadic summation.
The case $h > 0$ is similar but with more cases, interpreting 
regular derivatives as
$\partial_{x,t} \approx (T^{-1} \partial_{\theta,\phi}, S^{-1} (\partial_\sigma \pm \partial_\phi))$.

We will use these Sobolev embeddings to estimate in
$L^{\infty}$ the following two functions, namely $Zu$ ($Z$ is either
$\partial_{\theta}$ or $\partial_{\phi}$ or any combinations of them)
and $(\partial_{\sigma}-\partial_{\phi})u $, as well as their derivatives.

To estimate $Zu$,  $w$ will be replaced with $Z u$. Then we can use 
the estimates \eqref{wneed2c2g+} and \eqref{wneecomp+} to bound the ``-" version of the right hand side in \eqref{interpolate-2h} by $1$.
This implies that in $C_{TS}$ we have the pointwise bound for $Zu$ 
\[
\Vert \partial^{\le h/2}   Z u \Vert_{L^{\infty}(C_{TS})} \lesssim 1.
\]

For $\left( \partial_{\sigma}-\partial_{\phi}\right) u$ we replace $w$ with $\left( \partial_{\sigma}-\partial_{\phi}\right) u$.
Then we can use 
the estimates \eqref{wneecomp+} and \eqref{m1+}  to bound the ``+" version of the right hand side in \eqref{interpolate-2h} by $S^\frac12 T^{-\frac12}$.

This implies that in $C_{TS}$ we have the pointwise bound for $(\partial_{\sigma}-\partial_{\phi}) u$ 
\[
\Vert \partial^{\leq \hh/2} (\partial_{\sigma}-\partial_{\phi}) u \Vert_{L^{\infty}(C_{TS})} \lesssim S^{\frac12} T^{-\frac12}.
\]

At this point we can rephrase the last two bounds as  bounds for $\nabla u$ and $\Tau u$:
\[
\begin{aligned}
&\Vert \partial^{\leq \hh/2} \Tau u \Vert_{L^{\infty}(C_{TS})}\lesssim T^{-1}, \\
& \Vert \partial^{\leq \hh/2}  \nabla u \Vert_{L^{\infty}(C_{TS})}\lesssim S^{-\frac12} T^{-\frac12}.
\end{aligned}
\]
This suffices for the bounds  \eqref{boot11}, \eqref{boot12bis} and \eqref{boot12} with $\delta = 0$ in 
our theorem provided that $\hh/ 2 > 2$ i.e. $\hh \geq 5$.

\bigskip
\textbf{Extra gain away from the cone.}
The remaining step in the proof of the theorem is to obtain the $\delta$ improvement in 
the bound \eqref{boot12}, which is needed  only for derivatives of $u$ of second and third 
order. As a byproduct, we will also obtain a similar improvement in \eqref{boot11}.
Precisely, we will prove that in $C_{TS}$ we have
\begin{equation} \label{delta-d}
| \partial^j \nabla u | \lesssim T^{-\frac12} S^{-\frac12-\delta}, \qquad j = 1,2
\end{equation}
respectively 
\begin{equation}\label{delta-tau}
| \partial^j \Tau u | \lesssim T^{-1} S^{-\delta} \qquad j = 1,2.
\end{equation}
For this we need an improvement in \eqref{interpolate-2h} when on the left we put
$\partial^j w$ with $0 < j < \hh/2$ namely, with some $\delta > 0$,
\begin{equation}
\label{interpolate-2h+}
\Vert \partial^j w\Vert_{L^{\infty}(C_{TS})} \lesssim S^{-\delta} \left(\Vert \ZZ^{\leq 2\hh} w\Vert_{L^{2}_{h} (C_{TS})} 
+ \left\|  \ZZ^{\leq \hh} \left( \partial_{\sigma}\pm \partial_{\phi}\right) w \right\| _{L^{2}_{h} (C_{TS})}\right).
\end{equation}

To prove \eqref{interpolate-2h+} we separate into two cases:

\bigskip

{\bf Case I.}  We consider first the slightly simpler case of the $-$ sign in \eqref{interpolate-2h}, which
corresponds to \eqref{delta-d}.  To obtain a better that $T^{-\frac12}
S^{-\frac12}$ bound for $\partial^j\nabla  u$ in \eqref{delta-d} or equivalently a better than $1$ bound for $\partial^j w$ in \eqref{interpolate-2h}  we consider the
balance of frequencies there.  Taking a Littlewood-Paley
decomposition, denote by $\lambda$ the $Z$ frequency and by $\mu$ the
$\partial_\sigma - \partial_\phi$ frequency. We can take both $\lambda,
\mu \geq 1$ since we work in a unit size region; all frequencies below
$1$ can be combined in the frequency $1$ case.  As before, for the
gradient we can think of the vector fields
\[
\nabla = \left\{T^{-1} Z, S^{-1} (\partial_\sigma - \partial_\phi)\right\}, \qquad \Tau = T^{-1} Z.
\]
The $L^2$ bound for $w_{\lambda\mu}$ given by the right hand side of \eqref{interpolate-2h} is 
\[
\|w_{\lambda \mu}\|_{L^2} \lesssim  \left( \lambda^2 + (T^{-1} \lambda +
S^{-1} \mu)^{2\hh} + \mu (\lambda +(T^{-1} \lambda + S^{-1} \mu)^{\hh})\right)^{-1}.
\]
To estimate $\partial^j w_{\lambda\mu}$ in $L^\infty$ we use Bernstein's inequality,
\[
\| \partial^j  w_{\lambda \mu}\|_{L^\infty} \lesssim \frac{   \lambda \mu^\frac12 (T^{-1} \lambda +
S^{-1} \mu)^{j}}{   \lambda^2 + (T^{-1} \lambda +
S^{-1} \mu)^{2\hh} + \mu ( \lambda + (T^{-1} \lambda + S^{-1} \mu)^{\hh})}.
\]
Now we maximize over $\lambda$ and $\mu$ on the right. One also needs to  sum 
with respect to $\lambda$ and $\mu$ but this is straightforward as there is 
dyadic exponential decay away from the maximum.
We first consider homogeneous variations in $\lambda,\mu$ where we keep the
 ratio fixed but vary the size. The maximum will be attained exactly when\footnote{Strictly speaking 
one should also separately consider the cases  when $\mu = 1$ or $\lambda = 1$. These are simpler 
and are omitted.}
\begin{equation}\label{max point}
\lambda = (T^{-1} \lambda + S^{-1} \mu)^{\hh}.
\end{equation}
Below that we have a positive power of the size parameter, and above
that a negative one.  Here it is important that $0 < j < 2\hh
-\frac32$. Otherwise, if $j = 0$ the power in the denominator always dominates
and the maximum is exactly at $\lambda = \mu = 1$. If $j$ is too large
then the above expression is unbounded.  The above bound from above
is not too important, as it gets tighter later on.

Assuming \eqref{max point}, the expression above simplifies to
\[
\frac{  \lambda^{\frac{j}{\hh}} \mu^\frac12 }{   \lambda + \mu}.
\]

We distinguish two cases:

\bigskip

(i) Large $\lambda$, 
\[
\lambda = (T^{-1} \lambda)^{\hh}, \qquad T^{-1} \lambda \geq S^{-1} \mu.
\]
Here we have $\lambda > \mu$ so the denominator becomes $\lambda$. Then the $\mu$ in the numerator
must be maximal. This leads to
\[
\lambda= T^{1+\frac{1}{\hh-1}}, \qquad \mu = \frac{S \lambda}{T},
\]
and the above expression becomes
\[
 [ S^\frac12 T^{-\frac12} \lambda^{\frac{j}\hh - \frac12}].
\]
which gains a power of $T$ provided that $2j < h$.

\bigskip

(ii) Large $\mu$, 
\[
\lambda = (S^{-1} \mu)^{\hh}, \qquad T^{-1} \lambda \leq S^{-1} \mu.
\]
We substitute this expression for $\lambda$ to get
\[
 \frac{S^{-j} \mu^{j+\frac12}} {(S^{-1} \mu)^{\hh}+ \mu}.
\]
This balances when the two terms in the denominator are equal. Then
\[
\mu  = S^{1+\frac{1}{\hh-1}},
\]
and we get 
\[
 S^{- j+ (j-\frac12)(1+\frac{1}{\hh-1})} =  S^{ (\frac{j}{h} -\frac12)(1+\frac{1}{\hh-1})},
\]
which gains a power of $S$ provided again that $2j < h$.

\bigskip

{\bf Case II.}  Here we consider  the case of the $+$ sign in \eqref{interpolate-2h}, which
corresponds to \eqref{delta-d}.  Here we have three relevant dyadic frequencies,
denoted by $\lambda$ for $Z$, $\mu$ for $\partial_\sigma -\partial_\phi$ and $\nu$ for 
$\partial_\sigma +\partial_\phi$. Bernstein's inequality now yields the bound
\[
\| \partial^j  w_{\lambda \mu}\|_{L^\infty} \lesssim  \frac{   \lambda 
\min\{\mu,\nu\}^\frac12 (T^{-1} \lambda +
S^{-1} \mu)^{j}}{   \lambda^2 + (T^{-1} \lambda +
S^{-1} \mu)^{2\hh} + \nu ( \lambda + (T^{-1} \lambda + S^{-1} \mu)^{\hh})}.
\]
Here we need to maximize the right hand side above  with respect to the three 
parameters $\lambda, \mu,\nu \geq 1$. 

As before, after excluding the cases $\lambda = 1$ and $\mu = 1$, 
one sees that at the maximum point we must have the relation \eqref{max point}
in which case  the expression above simplifies to
\[
\frac{  \lambda^{\frac{j}{\hh}} \min\{\mu,\nu\}^\frac12 }{   \lambda + \nu}.
\]
The $\nu$ maximum is at $\nu = \lambda$, so we are left with
\[
\frac{  \lambda^{\frac{j}{\hh}} \min\{\mu,\lambda\}^\frac12 }{   \lambda }.
\]
We consider the same two cases (i) and (ii) as in Case I. Part (i) is identical,
whereas in part (ii) we get 
\[
\frac{S^{-j} \mu^{j}   \min\{\mu, (S^{-1} \mu)^\hh\}^\frac12} {(S^{-1} \mu)^{\hh}}.
\]
The $\mu$ maximum is attained when the two terms in the min are equal, which again gives 
the same outcome as in Case I (ii).

\appendix
\section{An interpolation lemma}
Here we prove the following interpolation Lemma:
\begin{lemma} \label{l:interpolation} Assume that $n\geq 0$ and 
\[
\frac{2}{p}=\frac{1}{2}+\frac{1}{q}, \qquad 2\leq q\leq \infty.
\]
Then we have
\begin{equation}
\label{interpolare-re}
\Vert  \partial^{n+1} Z \phi \Vert_{L^{p}(C_{TS})} \lesssim \|  Z^{\leq 2} \phi  \|_{L^2(C_{TS})}^\frac{1}{2} ( \| \partial^{\leq 2n} \partial^2 \phi \|_{L^{q}(C_{TS})} + S^{-1} \| \partial \phi\|_{L^q(C_{TS})})^{\frac{1}{2}}.
\end{equation}
 The same holds in $C_T^{int}$.
\end{lemma}
\begin{proof}  The case of $C_T^{int}$ is similar and is omitted. We prove the result in several modular steps:

\medskip

\textbf{ Step 1: Reduction to the case of a cube. }
Here we use hyperbolic polar coordinates adapted to $C_{TS}$ 
to view $C_{TS}$  as a unit cube $Q$, which in turn  we can view as a product $Q = Q_1  \times Q_2$, with coordinates denoted by 
$(s,y)$.

The differentiation operators in the unit cube, translated to the $C_{TS}$ setting, are 
\[
(\partial_s,\partial_y) \approx (S \partial_r, Z).
\]

Then we can represent the differentiation operators in the 
Minkowski space as 
\[
(\partial_{t},\partial_x) \approx (S^{-1} \partial_s, T^{-1} \partial_y),
\]
where we have the slight difficulty that the connection between the two bases has variable coefficients, which are smooth on the $T$ scale. Hence when we represent 
$\partial^j_{x,t}$ in the basis on the right, we also get 
lower order terms,
\[
| \partial^j \phi| \lesssim \sum_{l=1}^j  T^{l-j} 
| (S^{-1} \partial_s, T^{-1} \partial_y)^l \phi| .
\]

Hence we can estimate the left hand side in \eqref{interpolare-re} by 
\[
\Vert  \partial^{n+1} Z \phi \Vert_{L^{p}(C_{TS})}
\lesssim \sum_{j=0}^{n+1} 
T^{j-n-1} \Vert  (S^{-1} \partial_s, T^{-1} \partial_y)^{j} \partial_y \phi \Vert_{L^{p}(C_{TS})}.
\]
On the other hand for the terms on the right we have
\[
\|  Z^{\leq 2} \phi  \|_{L^2(C_{TS})}
\approx \| \partial_y^{\leq 2} \phi \Vert_{L^{2}(C_{TS})},
\]
respectively
\[
\begin{aligned}
&\| \partial^{\leq 2n} \partial^2 \phi \|_{L^{q}(C_{TS})} + S^{-1} \| \partial \phi\|_{L^q (C_{TS})}\\
&\quad \approx \| (S^{-1} \partial_s, T^{-1} \partial_y)^{\leq 2n} (S^{-1} \partial_s, T^{-1} \partial_y)^2 \phi \|_{L^{q}(C_{TS})} + S^{-1} \| (S^{-1} \partial_s, T^{-1} \partial_y) \phi\|_{L^q(C_{TS})}.
\end{aligned}
\]
Changing also the measure of integration, the bound 
\eqref{interpolare-re} is now reduced to
\[
\begin{aligned}
\sum_{j=0}^{n+1} 
T^{2(j-n-1)} \Vert  (S^{-1} \partial_s, T^{-1} \partial_y)^{j} \partial_y \phi \Vert_{L^{p}(Q)}^2
\lesssim &  \| \partial_y^{\leq 2} \phi \Vert_{L^{2}(Q)}
(\| (S^{-1} \partial_s, T^{-1} \partial_y)^{\leq 2n} (S^{-1} \partial_s, T^{-1} \partial_y)^2 \phi \|_{L^{q}(Q)} \\ & + S^{-1} \| (S^{-1} \partial_s, T^{-1} \partial_y) \phi\|_{L^q(Q)}).
\end{aligned}
\]
Here we are in a fixed unit size  cube, so $S$ and $T$ simply play the role of large parameters with the only constraint $1 \leq S \lesssim T$.
We need to estimate all components of the norm on the left.
By interpolating solely in $y$, it suffices
to consider on the left only the cases when either $j=n+1$ or when there are no $y$ derivatives in the 
$j$-th power. In the latter situation the case $j = 0$ is trivial, while if $j \geq 1$ we may redefine $n$ to a lower value $n := j-1$ and thus assume that $j = n+1$. Then we discard all $T^{-1} \partial_y$
on the right, arriving at
\begin{equation}\label{main-interp}
\| \partial_s^{n+1} \partial_y \phi \|_{L^p(Q)}^2
\lesssim  \| \partial_y^{\leq 2} \phi \Vert_{L^{2}(Q)} (\| \partial_s^{\leq 2n} \partial_s^2 \phi\|_{L^q(Q)} + \| \partial_s \phi\|_{L^q(Q)}).
\end{equation}
It remains to consider the case $j = n+1$
with at least one $T^{-1} \partial_y$ factor on the left. Then we can rewrite this as a bound for 
$\psi = \partial_y^2 \phi$, namely 
\[
 \|(S^{-1} \partial_s, T^{-1} \partial_y)^{n} \psi\|_{L^p}^2
\lesssim \|\psi\|_{L^2} \| (S^{-1} \partial_s, T^{-1} \partial_y)^{\leq 2n}  \psi\|_{L^q(Q)}.
\]
But after rescaling back to the original size this 
is just the classical Gagliardo-Nirenberg inequality
in a cube of size $S \times T \times T$.
It remains to prove \eqref{main-interp} in a unit sized cube.

\bigskip

\textbf{ Step 2: Reduction to a bound in $\R^3$.} Here we harmlessly subtract the $s$ average of $\phi$
from $\phi$. By Poincare's inequality this allows us to reduce to the case when $\phi$ has zero average in $s$, where the 
$L^q (Q)$ norm of $\phi$ is also under control. Thus 
\eqref{main-interp} reduces to
\begin{equation}\label{main-interp+}
\| \partial_s^{n+1} \partial_y \phi \|_{L^p(Q)}^2
\lesssim  \| \partial_y^{\leq 2} \phi \Vert_{L^{2}(Q)} \| \partial_s^{\leq 2n+2}  \phi\|_{L^q(Q)} .
\end{equation}

Having inhomogeneous norms on the right allows us to 
to extend $\phi$ to a double cube and then truncate,
thus reducing to proving  \eqref{main-interp+} in all of $\R^3$, for a compactly supported function $\phi$.

\textbf{ Step 3: Interpolation in $\R^3$.} Here we use Stein's interpolation theorem 
for the holomorphic family of operators
\[
T_z \phi = e^{(z-\frac12)^2}|D_y|^{2(1-z)}  |D_s|^{(2n+2)z} \phi
\]
for $z$ in the strip 
\[
S = \{ z\in \mathbb{C} \, :\,  0 \leq \Re z \leq 1\}.
\]
For this family we have the interpolation inequality
\begin{equation}
\| \partial_s^{n+1} \partial_y \phi \|_{L^p}^2
\lesssim \| |D_s|^{n+1} |D_y| \phi \|_{L^p}^2
= \| T_{\frac12} \phi \|_{L^p}^2
\lesssim \sup_{\Re z = 0} \|T_z \phi\|_{L^2}
\sup_{\Re z = 1} \|T_z \phi\|_{L^q},
\end{equation}
where for the first step we use that the Hilbert and Riesz transforms are bounded from $L^p\rightarrow L^p$, with $1<p< \infty$.
Hence it suffices to show that 
\begin{equation}
    \sup_{\Re z = 0} \|T_z \phi\|_{L^2}
    \lesssim \| \partial_y^2 \phi \|_{L^2},
\end{equation}
respectively 
\begin{equation}
 \sup_{\Re z = 1} \|T_z \phi\|_{L^q}
 \lesssim \| \partial_s^{\leq 2n+2}  \phi\|_{L^q} .
\end{equation}
The first bound is straightforward by Plancherel's theorem. For the second bound
we need to verify that 
\[
|D_s|^{i \theta} |D_y|^{i\sigma}: L^q \to L^q,
\]
with sub-Gaussian norm growth in $\theta$ and $\sigma$ at $\pm \infty$.

If $1 < q < \infty$ then we are in a special case of the Hormander-Mikhlin theorem applied separately with respect to the two variables.

It remains to consider 
the special case $q = \infty$, where we show instead 
that 
\[
|D_s|^{i \theta} |D_y|^{i\sigma}: L^\infty \to BMO.
\]
This is true separately for each factor, 
but not immediately true for the product.
To address this difficulty we use a $0$-homogeneous 
cutoff function $\chi$ with smooth symbol $\chi(\xi,\eta)$, where $\xi$ and $\eta$ are the Fourier variables for $s$, respectively $y$. This is chosen so that $\chi = 1$ near $\xi = 0$ respectively 
$\chi = 0$ near $\eta = 0$. We separate the above product into two parts,
\[
|D_s|^{i \theta} |D_y|^{i\sigma}
=  ( \chi(D) |D_y|^{i\sigma})|D_s|^{i \theta} +
((1-\chi(D))|D_s|^{i \theta}) |D_y|^{i\sigma}.
\]
These are similar so we estimate the first one. 
Here 
\[
|D_s|^{i \theta}: L^\infty_{s,y} \to L^\infty_y BMO_s \subset BMO_{sy},
\]
while $ \chi(D) |D_y|^{i\sigma}$ is a Hormander-Mikhlin multiplier so it maps BMO into BMO.
\end{proof}

\medskip


\end{document}